\documentclass{article}

\usepackage{amsfonts}
\usepackage{amsmath}
\usepackage{mathtools}
\usepackage{multirow}
\usepackage{epstopdf}
\usepackage{algorithm}
\usepackage{algorithmic}
\usepackage{bm}
\usepackage{microtype}
\usepackage{amssymb}
\usepackage{enumerate}

\newtheorem{theorem}{Theorem}

\newtheorem{remark}{Remark}
\newtheorem{lemma}{Lemma}
\newtheorem{proof}{Proof}
\newtheorem{corollary}{Corollary}
\newtheorem{definition}{Definition}
\newtheorem{assumption}{Assumption}

\newcommand{\x}{\mathbf{x}}
\newcommand{\y}{\mathbf{y}}
\newcommand{\z}{\mathbf{z}}

\newcommand{\s}{\mathbf{s}}
\renewcommand{\u}{\mathbf{u}}
\newcommand{\R}{\mathbb{R}}
\newcommand{\N}{\mathcal{N}}
\newcommand{\W}{W}
\newcommand{\I}{I}
\newcommand{\A}{A}
\newcommand{\1}{\mathbf{1}}
\newcommand{\M}{\mathcal{M}}
\newcommand{\cS}{\mathcal{S}}
\newcommand{\bO}{{\cal O}}

\newcommand{\<}{\left\langle}
\renewcommand{\>}{\right\rangle}
\newcommand{\tabincell}[2]{\begin{tabular}{@{}#1@{}}#2\end{tabular}}

\usepackage[accepted]{icml2020}

\icmltitlerunning{Accelerated Gradient Tracking over Time-varying Graphs}

\begin{document}

\onecolumn
\icmltitle{Accelerated Gradient Tracking  over Time-varying Graphs for Decentralized Optimization}

\begin{icmlauthorlist}
\icmlauthor{Huan Li}{to}
\icmlauthor{Zhouchen Lin}{goo}
\end{icmlauthorlist}

\icmlaffiliation{to}{Institute of Robotics and Automatic Information Systems, College of Artificial Intelligence, Nankai University, Tianjin, China (lihuanss@nankai.edu.cn).\\}
\icmlaffiliation{goo}{State Key Lab of General AI, School of Intelligence Science and Technology, Peking University, Beijing, China
    (zlin@pku.edu.cn).}





\vskip 0.3in



\printAffiliationsAndNotice{}  

\begin{abstract}
  Decentralized optimization over time-varying graphs has been increasingly common in modern machine learning with massive data stored on millions of mobile devices, such as in federated learning. This paper revisits the widely used accelerated gradient tracking and extends it to time-varying graphs. We prove that the practical single loop accelerated gradient tracking needs $\bO((\frac{\gamma}{1-\sigma_{\gamma}})^2\sqrt{\frac{L}{\epsilon}})$ and $\bO((\frac{\gamma}{1-\sigma_{\gamma}})^{1.5}\sqrt{\frac{L}{\mu}}\log\frac{1}{\epsilon})$ iterations to reach an $\epsilon$-optimal solution over time-varying graphs when the problems are nonstrongly convex and strongly convex, respectively, where $\gamma$ and $\sigma_{\gamma}$ are two common constants charactering the network connectivity, $L$ and $\mu$ are the smoothness and strong convexity constants, respectively, and one iteration corresponds to one gradient oracle call and one communication round. Our convergence rates improve significantly over the ones of $\bO(\frac{1}{\epsilon^{5/7}})$ and $\bO((\frac{L}{\mu})^{5/7}\frac{1}{(1-\sigma)^{1.5}}\log\frac{1}{\epsilon})$, respectively, which were proved in the original literature of accelerated gradient tracking only for static graphs, where $\frac{\gamma}{1-\sigma_{\gamma}}$ equals $\frac{1}{1-\sigma}$ when the network is time-invariant. When combining with a multiple consensus subroutine, the dependence on the network connectivity constants can be further improved to $\bO(1)$ and $\bO(\frac{\gamma}{1-\sigma_{\gamma}})$ for the gradient oracle and communication round complexities, respectively. When the network is static, by employing the Chebyshev acceleration, our complexities exactly match the lower bounds without hiding any poly-logarithmic factor for both nonstrongly convex and strongly convex problems.
\end{abstract}

\section{Introduction}
Distributed optimization has emerged as a promising framework in machine learning motivated by large-scale data being produced or stored in a network of nodes. Due to the popularity of smartphones and their growing computational power, time-varying graphs are increasingly common in modern distributed optimization, where the communication links in the network may vary with time, and the devices may not be active all the time such that the network may be even unconnected at each time. A typical example is federated learning \citep{federated20,federated19}, which involves training a global statistical model from data stored on millions of mobile devices. The physical constraints on each device typically result in only a small fraction of the devices being active at once, and it is possible for an active device to drop out at a given time \citep{federated-system}. Although centralized network is the predominant topology in most machine learning systems, such as TensorFlow, decentralized network has been a potential alternative because it reduces the high communication cost on the central server \citep{liu-2017}. This motivates us to study decentralized optimization over time-varying graphs. In this paper, we consider the following convex optimization problem:
\begin{eqnarray}
\min_{x\in\R^p} F(x)=\frac{1}{m}\sum_{i=1}^m f_{(i)}(x),\label{problem}
\end{eqnarray}
where the local objective functions $f_{(i)}$ are distributed separately over a network of nodes. The network is mathematically represented as a sequence of time-varying graphs $\{\mathcal{G}^0,\mathcal{G}^1,...\}$, and each graph instance $\mathcal{G}^k$ consists of a fixed set of agents $\mathcal{V}=\{1,...,m\}$ and a set of time-varying edges $\mathcal{E}^k$. Agents $i$ and $j$ can exchange information at time $k$ if and only if $(i,j)\in\mathcal{E}^k$. Each agent $i$ privately holds a local objective $f_{(i)}$, and makes its decision only based on the local computations on $f_{(i)}$ and the local information received from its neighbors. The local objective functions are assumed to be smooth. We consider both strongly convex and nonstrongly convex objectives in this paper.

Although decentralized optimization over static graphs has been well studied, for example, lower bounds on the number of communication rounds and gradient or stochastic gradient oracle calls for strongly convex and smooth problems are well-known \citep{dasent,scaman-2019-jmlr,ADFS}, and optimal accelerated algorithms with upper bounds exactly matching the lower bounds are developed \citep{richtaric-2020-dist,VRGT-Li20}, for the time-varying graphs, the problem is more challenging. It is unclear how to design practical accelerated methods with the optimal dependence on the precision $\epsilon$ and the condition number of the objectives, exactly matching that of the classical centralized accelerated gradient descent. In this paper, we aim to address this question.

\subsection{Notations and Assumptions}\label{sec:notation}
Throughout this article, we denote $x_{(i)}$ to be the local variable for agent $i$. We use the subscript $(i)$ to distinguish the $i$th element of vector $x$. To write the algorithm in a compact form, we introduce the aggregate objective function $f(\x)$ with its aggregate variable $\x\in\R^{m\times p}$ and aggregate gradient $\nabla f(\x)\in\R^{m\times p}$ as
\begin{eqnarray}
\qquad\begin{aligned}\label{aggregate}
f(\x)=\sum_{i=1}^m f_{(i)}(x_{(i)}),\quad\x=\left(
  \begin{array}{c}
    x_{(1)}^{\top}\\
    \vdots\\
    x_{(m)}^{\top}
  \end{array}
\right),\quad \nabla f(\x)=\left(
  \begin{array}{c}
    \nabla f_{(1)}(x_{(1)})^{\top}\\
    \vdots\\
    \nabla f_{(m)}(x_{(m)})^{\top}
  \end{array}
\right).
\end{aligned}
\end{eqnarray}
Denote $\x^k$ to be the value at iteration $k$. For scalars, for example, $\theta$, we use $\theta_k$ instead of $\theta^k$ to denote the value at iteration $k$, while the latter represents its $k$th power. Specially, $x^{\top}$ means the transpose of $x$. We denote $\|\cdot\|$ to be the Frobenius norm for matrices and the $\ell_2$ Euclidean norm for vectors uniformly, and $\|\cdot\|_2$ as the spectral norm of matrices. Denote $\I$ as the identity matrix and $\1$ as the column vector of $m$ ones. Assume that problem (\ref{problem}) has a solution, and let $x^*$ be any one of them. Define the average variable across all the local variables as
\begin{eqnarray}
\overline x=\frac{1}{m}\sum_{i=1}^mx_{(i)},\quad \overline y=\frac{1}{m}\sum_{i=1}^my_{(i)},\quad \overline z=\frac{1}{m}\sum_{i=1}^mz_{(i)},\quad \overline s=\frac{1}{m}\sum_{i=1}^ms_{(i)},\label{def_av_x}
\end{eqnarray}
where $x$, $y$, $z$, and $s$ will be used later in the development of the algorithm. Define operator $\Pi=\I-\frac{\1\1^{\top}}{m}$ to measure the consensus violation such that
\begin{equation}
\Pi\x=\left(
  \begin{array}{c}
    x_{(1)}^{\top}-\overline x^{\top}\\
    \cdots\\
    x_{(m)}^{\top}-\overline x^{\top}\\
  \end{array}
\right).\label{def_pi}
\end{equation}

We make the following assumptions for each local objective function in problem (\ref{problem}).
\begin{assumption}\label{assumption_f}
$\vspace*{-0.6cm}\linebreak$
\begin{enumerate}
\item Each $f_{(i)}(x)$ is $\mu$-strongly convex: $f_{(i)}(y)\geq f_{(i)}(x)+\<\nabla f_{(i)}(x),y-x\>+\frac{\mu}{2}\|y-x\|^2$. Especially, we allow $\mu$ to be zero throughout this paper, and in this case we say $f_{(i)}(x)$ is convex.
\item Each $f_{(i)}(x)$ is $L$-smooth, that is, $f_{(i)}(x)$ is differentiable and its gradient is $L$-Lipschitz continuous: $\|\nabla f_{(i)}(y)-\nabla f_{(i)}(x)\|\leq L\|y-x\|$.
\end{enumerate}
\end{assumption}
A direct consequence of the smoothness and convexity assumptions is the following property \citep{Nesterov-2004}:
\begin{eqnarray}
\frac{1}{2L}\|\nabla f_{(i)}(y)-\nabla f_{(i)}(x)\|^2\leq f_{(i)}(y)-f_{(i)}(x)-\<\nabla f_{(i)}(x),y-x\>\leq \frac{L}{2}\|y-x\|^2.\label{cont43}
\end{eqnarray}

The information exchange between different agents in the network is realized through a gossip matrix such that communication can be represented as a matrix multiplication with the gossip matrix. When the network is static, we make the following standard assumptions for the gossip matrix $\W\in\R^{m\times m}$ \citep{qu2017}:
\begin{assumption}\label{assumption_w}
$\vspace*{-0.6cm}\linebreak$
\begin{enumerate}
\item (Decentralized property) $\W_{i,j}>0$ if and only if $(i,j)\in\mathcal{E}$ or $i=j$. Otherwise, $\W_{i,j}=0$.
\item (Double stochasticity) $\W\1=\1$ and $\1^{\top}\W=\1^{\top}$.
\end{enumerate}
\end{assumption}
Note that we do not assume that $\W$ is symmetric. If the network is connected, Assumption \ref{assumption_w} implies that the second largest singular value $\sigma$ of $\W$ is less than 1 (its largest one equals 1), that is, $\sigma=\|W-\frac{1}{m}\1\1^{\top}\|_2<1$. Moreover, we have the following classical consensus contraction:
\begin{eqnarray}
\qquad\|\Pi W\x\|=\left\|\left(W-\frac{1}{m}\1\1^{\top}\right)\x\right\|=\left\|\left(W-\frac{1}{m}\1\1^{\top}\right)\left(\I-\frac{1}{m}\1\1^{\top}\right)\x\right\|\leq \sigma\|\Pi\x\|.\label{consensus-contraction}
\end{eqnarray}
We often use $\frac{1}{1-\sigma}$ as the condition number of the communication network.

When the network is time-varying, each graph instance $\mathcal{G}^k$ associates with a gossip matrix $\W^k$. Denote
\begin{equation}
W^{k,\gamma}=W^kW^{k-1}\cdots W^{k-\gamma+1},\quad\mbox{ for any }k\geq\gamma-1,\label{def_W_tau}
\end{equation}
$W^{k,0}=\I$, and we follow \citep{shi2017} to make the following standard assumptions for the sequence of gossip matrices $\{\W^k\}_{k=0}^{\infty}$.
\begin{assumption}\label{assumption_w_tv}
$\vspace*{-0.6cm}\linebreak$
\begin{enumerate}
\item (Decentralized property) $\W^k_{i,j}>0$ if and only if $(i,j)\in\mathcal{E}^k$ or $i=j$. Otherwise, $\W_{i,j}^k=0$.
\item (Double stochasticity) $\W^k\1=\1$ and $\1^{\top}\W^k=\1^{\top}$.
\item (Joint spectrum property) There exists a constant integer $\gamma$ such that
\begin{equation}
\sigma_{\gamma}<1,\quad\mbox{where}\quad \sigma_{\gamma}=\sup_{k\geq \gamma-1} \left\|W^{k,\gamma}-\frac{1}{m}\1\1^{\top}\right\|_2.\notag
\end{equation}
\end{enumerate}
\end{assumption}
Assumption \ref{assumption_w_tv} is weaker than the assumption that every graph $\mathcal{G}^k$ is connected. A typical example of the gossip matrix satisfying Assumption \ref{assumption_w_tv} is the Metropolis weight over $\gamma$-connected graphs. The former is defined as
\begin{equation}
W_{ij}^k=\left\{
  \begin{array}{ll}
    1/(1+\max\{d_i^k,d_j^k\}), &\mbox{if }(i,j)\in\mathcal{E}^k,\\
    0,&\mbox{if }(i,j)\neq\mathcal{E}^k\mbox{ and }i\neq j,\\
    1-\sum_{l\in\N_{(i)}^k}W_{il}^k, &\mbox{if } i=j,
  \end{array}
\right.\label{weight-matrix}
\end{equation}
where $\N_{(i)}^k$ is the set of neighbors of agent $i$ at time $k$, and $d_i^k=|\N_{(i)}^k|$ is the degree. The $\gamma$-connected graph sequence is defined as follows \citep{shi2017}.
\begin{definition}
The time-varying undirected graph sequence $\{\mathcal{V},\mathcal{E}^k\}_{k=0}^{\infty}$ is $\gamma$-connected if there exists some positive integer $\gamma$ such that the union of these $\gamma$ consecutive undirected graphs $\{\mathcal{V},\cup_{r=k}^{k+\gamma-1}\mathcal{E}^r\}$ is connected for all $k=0,1,...$.
\end{definition}
When Assumption \ref{assumption_w_tv} holds, we have the following $\gamma$-step consensus contraction:
\begin{equation}
\|\Pi W^{k,\gamma}\x\|\leq\sigma_{\gamma}\|\Pi\x\|,\qquad\mbox{for any }k\geq\gamma-1.\label{tau_consensus-contraction}
\end{equation}
When the algorithm proceeds less than $\gamma$ steps, we only have
\begin{equation}
\|\Pi W^{k,t}\x\|\leq\|\Pi\x\|,\qquad\mbox{for any }0\leq t<\gamma\mbox{ and }k\geq t-1.\label{tau_consensus-contraction2}
\end{equation}

In decentralized optimization, people often use communication round complexity and gradient oracle complexity to measure the convergence speed. The former means the number of communication rounds to reach an $\epsilon$-optimal solution with $F(x)-F(x^*)\leq \epsilon$, while the latter means the number of gradient oracle calls. In one communication round, all the agents can receive $\bO(1)$ vectors, such as $x_{(j)}^k$, from each of its neighbors in parallel, which can be represented as $W^k\x^k$ mathematically. In one gradient oracle call, all the agents call the oracle to compute their local gradients $\nabla f_{(i)}(x_{(i)}^k)$ in parallel.

\subsection{Literature Review}
In this section, we briefly review the decentralized algorithms over static graphs and time-varying graphs, mainly focusing on the accelerated methods. We emphasize gradient tracking \citep{shi2017} and its acceleration \citep{qu2017-2}, which are mostly relevant for our work. Tables \ref{table-comp} and \ref{table-comp2} sum up the complexity comparisons of the state-of-the-art methods.
\subsubsection{Decentralized Optimization over Static Graphs}

Decentralized optimization has been studied for a long time \citep{Bertsekas1983,Tsitsiklis1986}. The representative decentralized algorithms include distributed gradient/subgradient descent (DGD) \citep{Nedic-2009,nedic2011asynchronous,Ram-2010,Yuan-2016}, EXTRA \citep{shi2015extra,Shi-2015-2}, gradient tracking \citep{shi2017,qu2017,aug-dgm,ranxin2018}, NIDS \citep{NIDS}, as well as the dual based methods, such as dual ascent \citep{Terelius-2011}, dual averaging \citep{Duchi12-da}, ADMM \citep{Wei-13-asy,Iutzeler-2016,makhdoumi-2017}, and the primal-dual method \citep{Lam-2017,scaman-2018,hong-2017,jakovetic-2017}. Among these methods, gradient tracking has the $\bO((\frac{L}{\mu}+\frac{1}{(1-\sigma)^2})\log\frac{1}{\epsilon})$ communication round and gradient oracle complexities for strongly convex problems and the $\bO(\frac{L}{\epsilon(1-\sigma)^2})$ complexities for nonstrongly convex ones. Recently, accelerated decentralized methods have gained significant attention due to their provable faster convergence rates.

\emph{Accelerated Methods for Strongly Convex and Smooth Decentralized Optimization.} The accelerated methods which can be applied to this scenario include the accelerated distributed Nesterov gradient descent (Acc-DNGD) \citep{qu2017-2}, the robust distributed accelerated stochastic gradient method \citep{Fallah19}, the multi-step dual accelerated method \citep{dasent,scaman-2019-jmlr}, accelerated penalty method (APM) \citep{li-2018-pm,Dvinskikh2019}, the multi-consensus decentralized accelerated gradient descent (Mudag) \citep{mudag20,mudag2-20}, accelerated EXTRA \citep{li-2019-extra,VRGT-Li20}, the decentralized accelerated augmented Lagrangian method \citep{Arjevani20}, and the accelerated proximal alternating predictor-corrector method (APAPC) \citep{richtaric-2020-dist}. \citet{dasent,scaman-2019-jmlr} proved the $\varOmega(\sqrt{\frac{L}{\mu(1-\sigma)}}\log\frac{1}{\epsilon})$ and $\varOmega(\sqrt{\frac{L}{\mu}}\log\frac{1}{\epsilon})$ lower bounds for communication rounds and gradient oracle calls, respectively. To the best of our knowledge, APAPC combined with the Chebyshev acceleration (CA) \citep{Arioli-2014} is the first to exactly achieve these lower bounds without hiding any poly-logarithmic factor. Although gradient tracking has been widely used in practice, its accelerated variant, Acc-DNGD, only has the $\bO((\frac{L}{\mu})^{5/7}\frac{1}{(1-\sigma)^{1.5}}\log\frac{1}{\epsilon})$ communication round and gradient oracle complexities \citep{qu2017-2}.

\begin{table}[t]
\caption{Comparisons among the state-of-the-art complexities of decentralized methods over static graphs, as well as those of gradient tracking and its accelerated variant Acc-DNGD. Double loop means the method needs to call a subroutine with multiple steps at each iteration, such as the Chebyshev acceleration, the multiple consensus, the gradient evaluation of Fenchel conjugate, or the minimization of a subproblem.}\label{table-comp}
\begin{center}
\begin{tabular}{|c|c|c|c|}
\hline
 \footnotesize Methods & \tabincell{c}{\footnotesize gradient oracle\\\footnotesize complexity}      & \tabincell{c}{\footnotesize communication round\\\footnotesize complexity} & \tabincell{c}{\footnotesize single or\\\footnotesize double loop}\\
\hline
\multicolumn{4}{c}{\footnotesize Nonstrongly convex and smooth functions}\\
\hline
 \tabincell{c}{\footnotesize Gradient tracking\\ \footnotesize\citep{qu2017}} & \footnotesize$\bO\left(\frac{L}{\epsilon(1-\sigma)^2}\right)$ & \footnotesize$\bO\left(\frac{L}{\epsilon(1-\sigma)^2}\right)$ & \footnotesize single\\
 \tabincell{c}{\footnotesize Acc-DNGD\\ \footnotesize\citep{qu2017-2}} & \footnotesize$\bO\left(\frac{1}{\epsilon^{5/7}}\right)$ & \footnotesize$\bO\left(\frac{1}{\epsilon^{5/7}}\right)$ & \footnotesize single\\
 \tabincell{c}{\footnotesize APM\\ \footnotesize\citep{li-2018-pm}\\ \hspace*{-0.2cm}\footnotesize\citep{Dvinskikh2019}} \hspace*{-0.3cm} & \footnotesize$\bO\left(\sqrt{\frac{L}{\epsilon}}\right)$ & \footnotesize$\bO\left(\sqrt{\frac{L}{\epsilon(1-\sigma)}}\log\frac{1}{\epsilon}\right)$ & \footnotesize double\\
 \tabincell{c}{\footnotesize Acc-EXTRA\\ \footnotesize\citep{li-2019-extra}}  & \footnotesize$\bO\left(\sqrt{\frac{L}{\epsilon(1-\sigma)}}\log\frac{1}{\epsilon}\right)$ & \footnotesize$\bO\left(\sqrt{\frac{L}{\epsilon(1-\sigma)}}\log\frac{1}{\epsilon}\right)$ & \footnotesize double\\
 \hline
 \tabincell{c}{\footnotesize Our results for Acc-GT}  & \footnotesize$\bO\left(\frac{1}{(1-\sigma)^2}\sqrt{\frac{L}{\epsilon}}\right)$ & \footnotesize$\bO\left(\frac{1}{(1-\sigma)^2}\sqrt{\frac{L}{\epsilon}}\right)$ & \footnotesize single\\
 \tabincell{c}{\footnotesize Our results for Acc-GT+CA}  & \footnotesize$\bO\left(\sqrt{\frac{L}{\epsilon}}\right)$ & \footnotesize$\bO\left(\sqrt{\frac{L}{\epsilon(1-\sigma)}}\right)$ & \footnotesize double\\
 \hline
  \tabincell{c}{\footnotesize Lower bounds \\ \footnotesize\citep{scaman-2019-jmlr}}  & \footnotesize$\bO\left(\sqrt{\frac{L}{\epsilon}}\right)$ & \footnotesize$\bO\left(\sqrt{\frac{L}{\epsilon(1-\sigma)}}\right)$ & \footnotesize $\backslash$\\
\hline
\multicolumn{4}{c}{\footnotesize Strongly convex and smooth functions}\\
\hline
 \tabincell{c}{\footnotesize Gradient tracking\\ \footnotesize\citep{sulaiman2020}} & \footnotesize$\bO\left(\left(\frac{L}{\mu}+\frac{1}{(1-\sigma)^2}\right)\log\frac{1}{\epsilon}\right)$ & \footnotesize$\bO\left(\left(\frac{L}{\mu}+\frac{1}{(1-\sigma)^2}\right)\log\frac{1}{\epsilon}\right)$ & \footnotesize single\\
 \tabincell{c}{\footnotesize Acc-DNGD\\ \footnotesize\citep{qu2017-2}} & \footnotesize$\bO\left(\left(\frac{L}{\mu}\right)^{5/7}\frac{1}{(1-\sigma)^{1.5}}\log\frac{1}{\epsilon}\right)$ & \footnotesize$\bO\left(\left(\frac{L}{\mu}\right)^{5/7}\frac{1}{(1-\sigma)^{1.5}}\log\frac{1}{\epsilon}\right)$ & \footnotesize single\\
 \tabincell{c}{\footnotesize APAPC+CA\\ \footnotesize\citep{richtaric-2020-dist}}  & \footnotesize$\bO\left(\sqrt{\frac{L}{\mu}}\log\frac{1}{\epsilon}\right)$ & \footnotesize$\bO\left(\sqrt{\frac{L}{\mu(1-\sigma)}}\log\frac{1}{\epsilon}\right)$ & \footnotesize double\\
 \hline
 \tabincell{c}{\footnotesize Our results for Acc-GT}  & \footnotesize$\bO\left(\sqrt{\frac{L}{\mu(1-\sigma)^3}}\log\frac{1}{\epsilon}\right)$ & \footnotesize$\bO\left(\sqrt{\frac{L}{\mu(1-\sigma)^3}}\log\frac{1}{\epsilon}\right)$ & \footnotesize single\\
 \tabincell{c}{\footnotesize Our results for Acc-GT+CA}  & \footnotesize$\bO\left(\sqrt{\frac{L}{\mu}}\log\frac{1}{\epsilon}\right)$ & \footnotesize$\bO\left(\sqrt{\frac{L}{\mu(1-\sigma)}}\log\frac{1}{\epsilon}\right)$ & \footnotesize double\\
 \hline
 \tabincell{c}{\footnotesize Lower bounds \\ \footnotesize\citep{scaman-2019-jmlr}}  & \footnotesize$\bO\left(\sqrt{\frac{L}{\mu}}\log\frac{1}{\epsilon}\right)$ & \footnotesize$\bO\left(\sqrt{\frac{L}{\mu(1-\sigma)}}\log\frac{1}{\epsilon}\right)$ & \footnotesize $\backslash$\\
\hline
\hline
\end{tabular}
\end{center}
\end{table}

\emph{Accelerated Methods for Nonstrongly Convex and Smooth Decentralized Optimization.} The accelerated methods for this scenario are much scarcer. Examples include the distributed Nesterov gradient with consensus \citep{Jakovetic-2014}, Acc-DNGD \citep{qu2017-2}, APM \citep{li-2018-pm,Dvinskikh2019}, accelerated EXTRA \citep{li-2019-extra}, and the accelerated dual ascent \citep{Uribe-2017}, where the last one adds a small regularizer to translate the problem to a strongly convex and smooth one. \citet{scaman-2019-jmlr} proved the $\varOmega(\sqrt{\frac{L}{\epsilon(1-\sigma)}})$ communication round complexity lower bound and the $\varOmega(\sqrt{\frac{L}{\epsilon}})$ gradient oracle complexity lower bound. To the best of our knowledge, there is no method matching these lower bounds exactly without hiding any poly-logarithmic factor. APM comes close to this target, but with an additional $\bO(\log\frac{1}{\epsilon})$ factor in the communication round complexity. Acc-DNGD, acceleration of gradient tracking, only has the $\bO(\frac{1}{\epsilon^{5/7}})$ complexities of communication rounds and gradient oracles, originally proved in \citep{qu2017-2}. Note that the dependence on $1-\sigma$, a small constant charactering the network connectivity, was not explicitly given in \citep{qu2017-2}. \citet{xu20aistate} proposed an accelerated primal dual method, however, their complexities remain $\bO(\frac{1}{\epsilon})$.

\subsubsection{Decentralized Optimization over Time-varying Graphs}

We review the decentralized algorithms over time-varying graphs in two scenarios. In the first scenario, the network may not be connected at every time, but it is assumed to be $\gamma$-connected. In the second scenario, the network is assumed to be connected at every time.

\emph{Not Connected at Every Time but $\gamma$-connected.} In this scenario, DIGing (that is, gradient tracking over time-varying graphs) \citep{shi2017}, PANDA \citep{panda18,panda19}, the time-varying $\mathcal{A}\mathcal{B}$/push-pull method \citep{tvab20}, the decentralized stochastic gradient descent (SGD) \citep{Koloskova20}, and the push-sum based methods \citep{pushsum16,pushsum15,shi2017} are the representative non-accelerated methods over time-varying graphs for convex problems, as well as NEXT \citep{next16} and SONATA \citep{sonata19} for nonconvex problems. When combing with Nesterov's acceleration, to the best of our knowledge, the decentralized accelerated gradient descent with consensus subroutine (DAGD-C) \citep{Alexander2221,Alexander21} is the only accelerated method for strongly convex and smooth objectives with explicit complexities in this general time-varying setting. However, the communication round complexity of DAGD-C has an additional $\bO(\log\frac{1}{\epsilon})$ factor compared with the classical centralized accelerated gradient method. For nonstrongly convex and smooth problems, no literature studies the accelerated methods over time-varying graphs. While APM \citep{li-2018-pm} was originally designed for static graphs, it can be easily extended to the time-varying case. However, as introduced in the previous section, APM also has an additional $\bO(\log\frac{1}{\epsilon})$ factor in the communication round complexity. Both DAGD-C and APM are double-loop methods, where one gradient is computed at each iteration of the outer loop, and multiple rounds of consensus communications follow up in the inner loop. The multiple consensus double loop may limit the applications of DAGD-C and APM. See the discussions in Remark \ref{remark4}.

\emph{Connected at Every Time.} In this scenario, the literature is rich and many distributed methods originally designed over static graphs, such as Acc-DNGD, can be directly used. \citet{Kovalev21,Kovalev21tvlb} proposed a dual based method named ADOM and its primal-only extension ADOM+, where the latter has the state-of-the-art $\bO(\frac{1}{1-\sigma}\sqrt{\frac{L}{\mu}}\log\frac{1}{\epsilon})$ communication round complexity and the $\bO(\sqrt{\frac{L}{\mu}}\log\frac{1}{\epsilon})$ gradient oracle complexity for strongly convex and smooth problems. \citet{Kovalev21tvlb} also established the lower bounds showing that their ADOM+ is optimal. \citet{Alexander20slow} gave the complexity of $\bO(\sqrt{\frac{L}{\mu(1-\sigma)}}\log\frac{1}{\epsilon})$ under a stronger assumption that the network changes slowly in the sense that the number of network changes cannot exceed some percentage of the number of total iterations. \citet{Nedic24-ac} studied the accelerated AB/push-pull method over directed graphs, but no accelerated rate is proved.

\begin{table}[t]
\caption{Comparisons among the state-of-the-art complexities of decentralized methods over time-varying graphs. We only compare with the methods working over $\gamma$-connected graphs.}\label{table-comp2}
\begin{center}
\begin{tabular}{|c|c|c|c|}
\hline
 \footnotesize Methods & \tabincell{c}{\footnotesize gradient oracle\\\footnotesize complexity}      & \tabincell{c}{\footnotesize communication round\\\footnotesize complexity} & \tabincell{c}{\footnotesize single or\\\footnotesize double loop}\\
\hline
\multicolumn{4}{c}{\footnotesize Nonstrongly convex and smooth functions}\\
\hline
\tabincell{c}{\footnotesize DGD\footnotemark[1]\\ \footnotesize\citep{Koloskova20}}  & \footnotesize$\bO\left(\frac{\gamma \overline \zeta\sqrt{L}}{(1-\sigma_{\gamma})\epsilon^{3/2}}+\frac{\gamma}{1-\sigma_{\gamma}}\frac{L}{\epsilon}\right)$ & \footnotesize$\bO\left(\frac{\gamma \overline \zeta\sqrt{L}}{(1-\sigma_{\gamma})\epsilon^{3/2}}+\frac{\gamma}{1-\sigma_{\gamma}}\frac{L}{\epsilon}\right)$ & \footnotesize single\\
 \tabincell{c}{\footnotesize APM\\ \footnotesize\citep{li-2018-pm}}  & \footnotesize$\bO\left(\sqrt{\frac{L}{\epsilon}}\right)$ & \footnotesize$\bO\left(\frac{\gamma}{1-\sigma_{\gamma}}\sqrt{\frac{L}{\epsilon}}\log\frac{1}{\epsilon}\right)$ & \footnotesize double\\
 \hline
 \tabincell{c}{\footnotesize Our results for Acc-GT}  & \footnotesize$\bO\left(\frac{\gamma^2}{(1-\sigma_{\gamma})^2}\sqrt{\frac{L}{\epsilon}}\right)$ & \footnotesize$\bO\left(\frac{\gamma^2}{(1-\sigma_{\gamma})^2}\sqrt{\frac{L}{\epsilon}}\right)$ & \footnotesize single\\
 \tabincell{c}{\footnotesize Our results for Acc-GT+\\ \footnotesize multiple consensus}  & \footnotesize$\bO\left(\sqrt{\frac{L}{\epsilon}}\right)$ & \footnotesize$\bO\left(\frac{\gamma}{1-\sigma_{\gamma}}\sqrt{\frac{L}{\epsilon}}\right)$ & \footnotesize double\\
\hline
\multicolumn{4}{c}{\footnotesize Strongly convex and smooth functions}\\
\hline
\tabincell{c}{\footnotesize DGD\\ \footnotesize\citep{Koloskova20}}  & \hspace*{-0.3cm}\footnotesize$\bO\left(\frac{\gamma \overline \zeta\sqrt{L}}{\mu(1-\sigma_{\gamma})\sqrt{\epsilon}}+\frac{\gamma}{1-\sigma_{\gamma}}\frac{L}{\mu}\log\frac{1}{\epsilon}\right)$\hspace*{-0.3cm} & \hspace*{-0.3cm} \footnotesize$\bO\left(\frac{\gamma \overline \zeta\sqrt{L}}{\mu(1-\sigma_{\gamma})\sqrt{\epsilon}}+\frac{\gamma }{1-\sigma_{\gamma}}\frac{L}{\mu}\log\frac{1}{\epsilon}\right)$\hspace*{-0.2cm} & \footnotesize single\\
 \tabincell{c}{\footnotesize DIGing\\ \footnotesize\citep{shi2017}} & \hspace*{-0.15cm}\footnotesize$\bO\left(\sqrt{m}\left(\frac{L}{\mu}\right)^{1.5}\frac{\gamma^3}{(1-\sigma_{\gamma})^2}\log\frac{1}{\epsilon}\right)$ \hspace*{-0.3cm} & \hspace*{-0.3cm} \footnotesize$\bO\left(\sqrt{m}\left(\frac{L}{\mu}\right)^{1.5}\frac{\gamma^3}{(1-\sigma_{\gamma})^2}\log\frac{1}{\epsilon}\right)$ \hspace*{-0.3cm}& \footnotesize single\\
 \tabincell{c}{\footnotesize DAGD-C\\ \footnotesize\citep{Alexander2221}}  & \footnotesize$\bO\left(\sqrt{\frac{L}{\mu}}\log\frac{1}{\epsilon}\right)$ & \footnotesize$\bO\left(\frac{\gamma}{1-\sigma_{\gamma}}\sqrt{\frac{L}{\mu}}\left(\log\frac{1}{\epsilon}\right)^2\right)$ & \footnotesize double\\
 \hline
 \tabincell{c}{\footnotesize Our results for Acc-GT}  & \footnotesize$\bO\left(\left(\frac{\gamma}{1-\sigma_{\gamma}}\right)^{1.5}\sqrt{\frac{L}{\mu}}\log\frac{1}{\epsilon}\right)$ & \footnotesize$\bO\left(\left(\frac{\gamma}{1-\sigma_{\gamma}}\right)^{1.5}\sqrt{\frac{L}{\mu}}\log\frac{1}{\epsilon}\right)$ & \footnotesize single\\
 \tabincell{c}{\footnotesize Our results for Acc-GT+\\ \footnotesize multiple consensus}  & \footnotesize$\bO\left(\sqrt{\frac{L}{\mu}}\log\frac{1}{\epsilon}\right)$ & \footnotesize$\bO\left(\frac{\gamma}{1-\sigma_{\gamma}}\sqrt{\frac{L}{\mu}}\log\frac{1}{\epsilon}\right)$ & \footnotesize double\\
\hline
\hline
\end{tabular}
\end{center}
\end{table}

\subsection{Contributions}

In this paper, we study accelerated gradient tracking over time-varying graphs with sharper complexities. We give our analysis over static graphs and time-varying graphs in a unified framework. The former scenario provides the basis and insights for the latter. Our contributions are summarized as follows:\footnotetext[1]{The method in \citep{Koloskova20} was designed for stochastic decentralized optimization. We recover the complexities for deterministic optimization by setting the variance of the stochastic gradient to be zero. On the other hand, $\overline\zeta=\frac{1}{m}\sum_{i=1}^m\|\nabla f_{(i)}(x^*)\|^2$.}
\begin{enumerate}
\item For time-varying graphs, our contributions include:
\begin{enumerate}
\item When the local objective functions are nonstrongly convex and smooth, we prove the $\bO(\frac{\gamma^2}{(1-\sigma_{\gamma})^2}\sqrt{\frac{L}{\epsilon}})$ complexities of communication rounds and gradient oracle calls for the practical single loop accelerated gradient tracking (Acc-GT). When combing with a multiple consensus subroutine, our complexities can be improved to $\bO(\frac{\gamma}{1-\sigma_{\gamma}}\sqrt{\frac{L}{\epsilon}})$ for communication rounds and $\bO(\sqrt{\frac{L}{\epsilon}})$ for gradient oracles. The number of our communication rounds is less than that of the state-of-the-art APM \citep{li-2018-pm} by a $\bO(\log\frac{1}{\epsilon})$ factor, while that of our gradient oracle calls is the same as that of APM.
\item When the local objective functions are strongly convex and smooth, we prove the $\bO((\frac{\gamma}{1-\sigma_{\gamma}})^{1.5}\sqrt{\frac{L}{\mu}}\log\frac{1}{\epsilon})$ communication round and gradient oracle complexities for the practical single loop Acc-GT. When combing with the multiple consensus subroutine, we can improve the communication round complexity to $\bO(\frac{\gamma}{1-\sigma_{\gamma}}\sqrt{\frac{L}{\mu}}\log\frac{1}{\epsilon})$ and the gradient oracle complexity to $\bO(\sqrt{\frac{L}{\mu}}\log\frac{1}{\epsilon})$. The number of our communication rounds is less than that of the state-of-the-art DAGD-C \citep{Alexander2221} by a $\bO(\log\frac{1}{\epsilon})$ factor, while our gradient oracle calls remain the same as that of DAGD-C.
\item To the best of our knowledge, this is the first time that the communication round upper bound with the optimal dependence on the precision $\epsilon$ and condition number $L/\mu$ is given for both nonstrongly convex and strongly convex problems. More importantly, they are established for a practical single loop algorithm.
\end{enumerate}
\item For static graphs as a special case, our contributions include:
\begin{enumerate}
\item When the local objective functions are nonstrongly convex and smooth, we prove the $\bO(\frac{1}{(1-\sigma)^2}\sqrt{\frac{L}{\epsilon}})$ complexities of communication rounds and gradient oracles for the practical single loop Acc-GT, which significantly improve over the existing $\bO(\frac{1}{\epsilon^{5/7}})$ ones originally proved in \citep{qu2017-2}. When combing with the Chebyshev acceleration, we can improve the complexities to $\bO(\sqrt{\frac{L}{\epsilon(1-\sigma)}})$ for communication rounds and $\bO(\sqrt{\frac{L}{\epsilon}})$ for gradient oracles, which exactly match the complexity lower bounds. As far as we know, we are the first to establish the optimal upper bounds for nonstrongly convex and smooth problems, which exactly match the corresponding lower bounds without hiding any poly-logarithmic factor.
\item When the local objective functions are strongly convex and smooth, we prove the $\bO(\sqrt{\frac{L}{\mu(1-\sigma)^3}}\log\frac{1}{\epsilon})$ communication round and gradient oracle complexities for the practical single loop Acc-GT, which improves over the existing $\bO((\frac{L}{\mu})^{5/7}\frac{1}{(1-\sigma)^{1.5}}\log\frac{1}{\epsilon})$ ones originally given in \citep{qu2017-2}. When combing with the Chebyshev acceleration, the complexities can be further improved to match the corresponding lower bounds and existing optimal upper bounds.
\end{enumerate}
\end{enumerate}

\section{Accelerated Gradient Tracking over Time-varying Graphs}\label{acc-gt-section}
We first review the gradient tracking and its accelerated variant, where the latter was only designed over static graphs, and then give our extensions of the accelerated gradient tracking to time-varying graphs with sharper complexities.

\subsection{Review of Gradient Tracking and Its Acceleration}
Gradient tracking \citep{shi2017,qu2017,aug-dgm,ranxin2018} keeps an auxiliary variable $s_{(i)}^k$ at each iteration for each agent $i$ to track the average of the gradients $\nabla f_{(j)}(x_{(j)}^k)$ for all $j=1,...,m$, such that if $x_{(i)}^k$ converges to some point $x^{\infty}$, $s_{(i)}^k$ converges to $\frac{1}{m}\sum_{i=1}^m\nabla f_{(i)}(x^{\infty})$. The auxiliary variable is updated recursively as follows:
\begin{eqnarray}
s_{(i)}^k=\sum_{j\in\N_{(i)}}W_{ij}s_{(j)}^{k-1}+\nabla f_{(i)}(x_{(i)}^k)-\nabla f_{(i)}(x_{(i)}^{k-1}),\notag
\end{eqnarray}
and each agent uses this auxiliary variable as the descent direction in the general distributed gradient descent framework:
\begin{eqnarray}
x_{(i)}^{k+1}=\sum_{j\in\N_{(i)}}W_{ij}x_{(j)}^k-\alpha s_{(i)}^k,\notag
\end{eqnarray}
where $\alpha$ is the step size. Writing gradient tracking in the compact form, it reads as follows:
\begin{subequations}
\begin{align}
&\s^k=W\s^{k-1}+\nabla f(\x^k)-\nabla f(\x^{k-1}),\notag\\
&\x^{k+1}=W\x^k-\alpha\s^k.\notag
\end{align}
\end{subequations}
Gradient tracking can be used over both static graphs and time-varying graphs \citep{shi2017}.

To further accelerate gradient tracking, \citet{qu2017-2} employed Nesterov's acceleration technique \citep{Nesterov-2004} and proposed the following accelerated distributed Nesterov gradient descent for nonstrongly convex problems:
\begin{subequations}
\begin{align}
&\y^k=\theta_k\z^k+(1-\theta_k)\x^k,\label{qu-s1}\\
&\s^k=\W\s^{k-1}+\nabla f(\y^k)-\nabla f(\y^{k-1}),\label{qu-s2}\\
&\x^{k+1}=W\y^k-\alpha\s^k\label{qu-s3},\\
&\z^{k+1}=W\z^k-\frac{\alpha}{\theta_k}\s^k.\label{qu-s4}
\end{align}
\end{subequations}
It can be checked that step (\ref{qu-s3}) is equivalent to the following one:
\begin{equation}
\x^{k+1}=\theta_k\z^{k+1}+(1-\theta_k)\W\x^k.\notag
\end{equation}
When strong convexity is assumed, \citet{qu2017-2} fixed $\theta_k$ at each iteration and replaced steps (\ref{qu-s1}) and (\ref{qu-s4}) by the following two steps:
\begin{eqnarray}\notag
\y^k=\frac{\x^k+\theta\z^k}{1+\theta},\qquad\z^{k+1}=(1-\theta)W\z^k+\theta W\y^k-\frac{\alpha}{\theta}\s^k.
\end{eqnarray}
The main idea behind the development of the above accelerated algorithms is to relate it to the inexact accelerated gradient descent \citep{devolder2014first} by taking average of the local variables over all $i=1,...,m$. See Section \ref{sec:proof1} for the details. Tables \ref{table-comp} and \ref{table-comp2} list the complexities of gradient tracking and its accelerated variant.
\subsection{Extension of Accelerated Gradient Tracking to Time-varying Graphs}\label{sec:acttv}
\begin{algorithm}[t]
   \caption{Accelerated Gradient Tracking (Acc-GT)}
   \label{accgt}
\begin{algorithmic}
   \STATE Initialize: $x_{(i)}^0=y_{(i)}^0=z_{(i)}^0=x_{int}$, $s_{(i)}^0=\nabla f_{(i)}(y_{(i)}^0)$, $z_{(i)}^1=\sum_{j\in\N_{(i)}^0}\W_{ij}^0z_{(j)}^0-\frac{\alpha}{\theta_0+\mu\alpha}s_{(i)}^0$, and $x_{(i)}^1=\theta_0z_{(i)}^1+(1-\theta_0)\sum_{j\in\N_{(i)}^0}\W_{ij}^0x_{(j)}^0$.
   \FOR{$k=1,2,...$}
   \STATE
   \vspace*{-0.5cm}\begin{eqnarray}
   \begin{aligned}\notag
   &y_{(i)}^k=\theta_kz_{(i)}^k+(1-\theta_k)x_{(i)}^k,\\
   &s_{(i)}^k=\sum_{j\in\N_{(i)}^k}\W_{ij}^ks_{(j)}^{k-1}+\nabla f_{(i)}(y_{(i)}^k)-\nabla f_{(i)}(y_{(i)}^{k-1}),\\
   &z_{(i)}^{k+1}=\frac{1}{1+\frac{\mu\alpha}{\theta_k}}\left(\sum_{j\in\N_{(i)}^k}\W_{ij}^k\left(\frac{\mu\alpha}{\theta_k}y_{(j)}^k+z_{(j)}^k\right)-\frac{\alpha}{\theta_k}s_{(i)}^k\right),\\
   &x_{(i)}^{k+1}=\theta_kz_{(i)}^{k+1}+(1-\theta_k)\sum_{j\in\N_{(i)}^k}\W_{ij}^kx_{(j)}^k.
   \end{aligned}
   \end{eqnarray}
   \ENDFOR
\end{algorithmic}
\end{algorithm}

In this paper, we study the following accelerated gradient tracking with time-varying gossip matrices:
\begin{subequations}
\begin{align}
&\y^k=\theta_k\z^k+(1-\theta_k)\x^k,\label{alg-tv-s1}\\
&\s^k=\W^k\s^{k-1}+\nabla f(\y^k)-\nabla f(\y^{k-1}),\label{alg-tv-s2}\\
&\z^{k+1}=\frac{1}{1+\frac{\mu\alpha}{\theta_k}}\left(\W^k\left(\frac{\mu\alpha}{\theta_k}\y^k+\z^k\right)-\frac{\alpha}{\theta_k}\s^k\right),\label{alg-tv-s3}\\
&\x^{k+1}=\theta_k\z^{k+1}+(1-\theta_k)\W^k\x^k,\label{alg-tv-s4}
\end{align}
\end{subequations}
where we initialize $\x^0$ such that $\Pi\x^0=0$. We give the specific descriptions of the method in Algorithm \ref{accgt} in a distributed way. Step (\ref{alg-tv-s2}) is the standard gradient tracking, while steps (\ref{alg-tv-s1}), (\ref{alg-tv-s3}), and (\ref{alg-tv-s4}) come from Nesterov's classical accelerated gradient descent \citep{Nesterov-2004}, except that one round of consensus communication is performed by multiplying the aggregate variables with a gossip matrix. We see that algorithm (\ref{alg-tv-s1})-(\ref{alg-tv-s4}) is equivalent to (\ref{qu-s1})-(\ref{qu-s4}) when the gossip matrix is fixed and $\mu=0$. However, when $\mu>0$, it is not equivalent to the method proposed in \citep{qu2017-2}. In fact, Nesterov's accelerated gradient methods have several variants, and we choose the one in the form of (\ref{alg-tv-s1})-(\ref{alg-tv-s4}) due to its simple convergence proof.

We follow the proof idea in \citep{Jakovetic-2014,qu2017-2} to rewrite the distributed algorithm in the form of inexact accelerated gradient descent. However, we use a different proof framework from \citep{qu2017-2} with much simpler proofs, and give sharper complexities. See Remark \ref{remark3} for the differences and the reasons of the convergence rates improvement. On the other hand, for time-varying graphs, unlike the classical analysis relying on the small gain theorem \citep{shi2017}, we construct a different way to bound the consensus errors such that the proof framework over static graphs can be extended to time-varying graphs. See the proof of Lemma \ref{lemma3} and the remark following it. Our proof technique may shed new light to decentralized optimization over time-varying graphs, and gives an alternative to the small gain theorem. There are two advantages of our proof technique: it can be embedded into many algorithm frameworks from the perspective of error analysis, and it can be applied to both strongly convex and nonstrongly convex problems, while the small gain theorem only applies to strongly convex ones.

Our main technical results concerning the convergence rates of the accelerated gradient tracking are summarized in the following two theorems for nonstrongly convex and strongly convex problems, respectively.
\begin{theorem}\label{theorem3}
Suppose that Assumption \ref{assumption_f} holds with $\mu=0$ and Assumption \ref{assumption_w_tv} holds for the sequence $\{\W^k\}_{k=0}^{T\gamma}$. Let the sequence $\{\theta_k\}_{k=0}^{T\gamma}$ satisfy $\frac{1-\theta_k}{\theta_k^2}=\frac{1}{\theta_{k-1}^2}$ with $\theta_0=1$, let $\alpha\leq\frac{(1-\sigma_{\gamma})^4}{21675L\gamma^4}$. Then for algorithm (\ref{alg-tv-s1})-(\ref{alg-tv-s4}), we have for any $T\geq 1$,
\begin{eqnarray}
F(\overline x^{T\gamma+1})-F(x^*)\leq \frac{2C}{\alpha(T\gamma+1)^2},\notag
\end{eqnarray}
and
\begin{eqnarray}
\frac{1}{m}\|\Pi\x^{T\gamma}\|^2\leq\frac{9C}{\alpha L(T\gamma+1)^2},\notag
\end{eqnarray}
where $C=\|\overline z^0-x^*\|^2+\frac{\alpha(1-\sigma_{\gamma})}{10mL\gamma}\max_{r=0,...,\gamma}\|\Pi\s^r\|^2$.
\end{theorem}

\begin{theorem}\label{theorem4}
Suppose that Assumption \ref{assumption_f} holds with $\mu>0$ and Assumption \ref{assumption_w_tv} holds for the sequences $\{\W^k\}_{k=0}^{T\gamma}$. Let $\alpha\leq\frac{(1-\sigma_{\gamma})^3}{4244L\gamma^3}$ and $\theta_k\equiv\theta=\frac{\sqrt{\mu\alpha}}{2}$. Then for algorithm (\ref{alg-tv-s1})-(\ref{alg-tv-s4}), we have for any $T\geq 1$,
\begin{eqnarray}
F(\overline x^{T\gamma+1})-F(x^*)+\left(\frac{\theta^2}{2\alpha}+\frac{\mu\theta}{2}\right)\|\overline z^{T\gamma+1}-x^*\|^2\leq (1-\theta)^{T\gamma+1}C,\notag
\end{eqnarray}
and
\begin{eqnarray}
\frac{1}{m}\|\Pi\x^{T\gamma}\|^2\leq (1-\theta)^{T\gamma+1}\frac{4C}{L},\notag
\end{eqnarray}
where $C=F(\overline x^0)-F(x^*)+\left(\frac{\theta^2}{2\alpha}+\frac{\mu\theta}{2}\right)\|\overline z^0-x^*\|^2+\frac{1-\sigma_{\gamma}}{49mL\gamma(1-\theta)}\M_{\s}^{\gamma,\gamma}+\frac{1459L\gamma^3}{m(1-\theta)(1-\sigma_{\gamma})^3}\M_{\z}^{\gamma,\gamma}+\frac{6.6L\gamma}{m(1-\theta)(1-\sigma_{\gamma})}\M_{\x}^{\gamma,\gamma}$, $\M_{\s}^{\gamma,\gamma}=\max_{r=1,...,\gamma}\|\Pi\s^r\|^2$, and similarly for $\M_{\z}^{\gamma,\gamma}$ and $\M_{\x}^{\gamma,\gamma}$.
\end{theorem}

When the local objectives are nonstrongly convex, we see from Theorem \ref{theorem3} that algorithm (\ref{alg-tv-s1})-(\ref{alg-tv-s4}) needs $\bO((\frac{\gamma}{1-\sigma_{\gamma}})^2\sqrt{\frac{LC}{\epsilon}})$ communication rounds and gradient oracle calls to find an $\epsilon$-optimal averaged solution (see Remark \ref{complexity-remark}). When strong convexity is assumed, we see from Theorem \ref{theorem4} that both the communication round and gradient oracle complexities are $\bO((\frac{\gamma}{1-\sigma_{\gamma}})^{1.5}\sqrt{\frac{L}{\mu}}\log\frac{1}{\epsilon})$. Our complexities have the optimal dependence on the precision $\epsilon$ and the condition number $L/\mu$, matching that of the classical centralized accelerated gradient method. As illustrated in Table \ref{table-comp2}, our communication round complexities improve over the state-of-the-art APM \citep{li-2018-pm} and DAGD-C \citep{Alexander2221} on the dependence of $\epsilon$ since they have an additional $\bO(\log\frac{1}{\epsilon})$ factor. However, our dependence on $\frac{\gamma}{1-\sigma_{\gamma}}$ is not state-of-the-art. We will improve it in Section \ref{sec:subroutine}.

\begin{remark}
In order to establish the proof, we use very small stepsizes with huge constants in the theorems, which is impractical. We suggest to tune the best stepsize in practice, rather than the ones used in the theorems.
\end{remark}
\begin{remark}
We measure the convergence rates at the averaged solution, which can be obtained by an additional consensus average routine $\u^{t+1}=\W^t\u^t$ initialized at $\u^0=\x^{T\gamma+1}$, and $\bO(\frac{\gamma}{1-\sigma_{\gamma}}\log\frac{1}{\epsilon})$ rounds of communications are enough. So the total complexities are $\bO((\frac{\gamma}{1-\sigma_{\gamma}})^2\sqrt{\frac{LC}{\epsilon}})+\bO(\frac{\gamma}{1-\sigma_{\gamma}}\log\frac{1}{\epsilon})$ and $\bO((\frac{\gamma}{1-\sigma_{\gamma}})^{1.5}\sqrt{\frac{L}{\mu}}\log\frac{1}{\epsilon})+\bO(\frac{\gamma}{1-\sigma_{\gamma}}\log\frac{1}{\epsilon})$ for nonstrongly convex and strongly convex problems, respectively, and they are dominated by the first parts.
\end{remark}
\begin{remark}\label{remark1}
For nonstrongly convex problems, we can also prove the convergence rate measured at the point $x_{(i)}^{T\gamma+1}$ for any $i$:
\begin{eqnarray}
\begin{aligned}\notag
F(x_{(i)}^{T\gamma+1})-F(x^*)\leq \frac{2C}{\alpha(T\gamma+1)^2}\max\left\{\frac{\sqrt{m}(1-\sigma_{\gamma})}{L\alpha\gamma},8m\right\}.
\end{aligned}
\end{eqnarray}
However, the complexities increase to $\bO(\max\{\sqrt{m},\sqrt[4]{m}(\frac{\gamma}{1-\sigma_{\gamma}})^{1.5}\}(\frac{\gamma}{1-\sigma_{\gamma}})^2\sqrt{\frac{L}{\epsilon}})$. For strongly convex problems, the complexities stay the same no matter measured at $\overline x^{T\gamma+1}$ or $x_{(i)}^{T\gamma+1}$ because the additional terms, such as $\max\{\sqrt{m},\sqrt[4]{m}(\frac{\gamma}{1-\sigma_{\gamma}})^{1.5}\}$ in the nonstrongly convex case, appear in the constant $C'$ in $\bO((\frac{\gamma}{1-\sigma_{\gamma}})^{1.5}\sqrt{\frac{L}{\mu}}\log\frac{C'}{\epsilon})$.
\end{remark}

\begin{remark}
In Theorems \ref{theorem3} and \ref{theorem4}, we measure the convergence rates at the $(T\gamma+1)$th iteration for simplicity. In fact, the same rates hold for any $K=T\gamma+r$ with $1\leq r\leq\gamma$ by regarding the $(r-1)$th iteration as the virtual initialization, which only influences the constant $C$ in Theorems \ref{theorem3} and \ref{theorem4}. In addition, since $\theta_{r-1}<1$, the constant $C$ in Theorem \ref{theorem3} contains an additional term $\frac{\alpha(1-\theta_{r-1})}{\theta_{r-1}^2}(F(\overline x^{r-1})-F(x^*))$.
\end{remark}


\begin{remark}\label{remark2}
Due to the physical constraints such as the battery dies, the device shuts down, or the WiFi network is unavailable, the agents may not be active all the time. Most literature let the agents wait and use the old iterates when rejoining the network. Alternatively, we can formulate this case by local updates \citep{Stich19,Koloskova20} and use our analysis framework to ensure the convergence. Mathematically, letting $\W^k_{ii}=1$ and $\W^k_{ij}=0$ for all $j\neq i$ and $k=t+1,t+2,..., t'$, which means that agent $i$ drops out from the communication network during the time $[t+1,t']$, algorithm (\ref{alg-tv-s1})-(\ref{alg-tv-s4}) reduces to the following steps for agent $i$ at iterations $k=t+1,t+2,..., t'$:
\begin{subequations}
\begin{align}
&y_{(i)}^k=\theta_k z_{(i)}^k+(1-\theta_k)x_{(i)}^k,\label{loca-s1}\\
&s_{(i)}^k=s_{(i)}^{k-1}+\nabla f_{(i)}(y_{(i)}^k)-\nabla f_{(i)}(y_{(i)}^{k-1}),\label{loca-s2}\\
&z_{(i)}^{k+1}=\frac{1}{1+\frac{\mu\alpha}{\theta_k}}\left(\left(\frac{\mu\alpha}{\theta_k}y_{(i)}^k+z_{(i)}^k\right)-\frac{\alpha}{\theta_k}s_{(i)}^k\right),\label{loca-s3}\\
&x_{(i)}^{k+1}=\theta_kz_{(i)}^{k+1}+(1-\theta_k)x_{(i)}^k,\label{loca-s4}
\end{align}
\end{subequations}
which are a serious of local updates without communications. When joining the network again, we require agent $i$ to make up the delayed computations by performing (\ref{loca-s1})-(\ref{loca-s4}) for $t'-t$ iterations. Note that (\ref{loca-s1})-(\ref{loca-s4}) has much lower cost than the same number of iterations (\ref{alg-tv-s1})-(\ref{alg-tv-s4}) because the CPU speed is much faster than the communication speed over TCP/IP or the slow WiFi \citep{Lam-2017}.
\end{remark}

\subsection{Special Cases over Static Graphs}\label{sec:actsg}

When we fix $\W^k=\W$, algorithm (\ref{alg-tv-s1})-(\ref{alg-tv-s4}) can be applied to static graphs. As a special case of Theorems \ref{theorem3} and \ref{theorem4}, we have the following theorems over static graphs.
\begin{theorem}\label{theorem1}
Suppose that Assumptions \ref{assumption_f} and \ref{assumption_w} hold with connected graphs and $\mu=0$. Let the sequence $\{\theta_k\}_{k=0}^K$ satisfy $\frac{1-\theta_k}{\theta_k^2}=\frac{1}{\theta_{k-1}^2}$ with $\theta_0=1$, let $\alpha\leq\frac{(1-\sigma)^4}{537L}$. Then for algorithm (\ref{alg-tv-s1})-(\ref{alg-tv-s4}) with fixed gossip matrix $\W$, we have for any $K\geq 1$
\begin{eqnarray}
F(\overline x^{K+1})-F(x^*)\leq \frac{1}{\alpha(K+1)^2}\left(2\|\overline z^0-x^*\|^2+\frac{\alpha(1-\sigma)}{2mL}\|\Pi\s^0\|^2\right),\notag
\end{eqnarray}
and
\begin{eqnarray}
\frac{1}{m}\|\Pi\x^K\|^2\leq\frac{1}{\alpha L(K+1)^2}\left(5\|\overline z^0-x^*\|^2+\frac{9\alpha(1-\sigma)}{4mL}\|\Pi\s^0\|^2\right).\notag
\end{eqnarray}
\end{theorem}

\begin{theorem}\label{theorem2}
Suppose that Assumptions \ref{assumption_f} and \ref{assumption_w} hold with connected graphs and $\mu>0$. Let $\alpha\leq\frac{(1-\sigma)^3}{119L}$ and $\theta_k\equiv\theta=\frac{\sqrt{\mu\alpha}}{2}$. Then for algorithm (\ref{alg-tv-s1})-(\ref{alg-tv-s4}) with fixed gossip matrix $\W$, we have for any $K\geq 1$
\begin{eqnarray}
F(\overline x^{K+1})-F(x^*)+\left(\frac{\theta^2}{2\alpha}+\frac{\mu\theta}{2}\right)\|\overline z^{K+1}-x^*\|^2\leq (1-\theta)^{K+1}C,\notag
\end{eqnarray}
and
\begin{eqnarray}
\frac{1}{m}\|\Pi\x^K\|^2\leq (1-\theta)^{K+1}\frac{4C}{L},\notag
\end{eqnarray}
where $C=F(\overline x^0)-F(x^*)+\left(\frac{\theta^2}{2\alpha}+\frac{\mu\theta}{2}\right)\|\overline z^0-x^*\|^2+\frac{4(1-\sigma)}{59mL(1-\theta)}\|\Pi\s^0\|^2$.
\end{theorem}
The above theorems give the $\bO(\frac{1}{(1-\sigma)^2}\sqrt{\frac{L}{\epsilon}})$ and $\bO(\sqrt{\frac{L}{\mu(1-\sigma)^3}}\log\frac{1}{\epsilon})$ convergence rates for nonstrongly convex and strongly convex problems, respectively. As illustrated in Table \ref{table-comp}, our convergence rates significantly improve over the ones of $\bO(\frac{1}{\epsilon^{5/7}})$ and $\bO((\frac{L}{\mu})^{5/7}\frac{1}{(1-\sigma)^{1.5}}\log\frac{1}{\epsilon})$, respectively, which were originally proved in \citep{qu2017-2}.
\begin{remark}
In our Theorem \ref{theorem2}, we require each $f_{(i)}(x)$ to be strongly convex. Some literatures study the weaker assumptions where only $F(x)$ is required to be strongly convex and each $f_{(i)}$ can be convex and smooth. \citet{Sun-2022} established the $\bO((\frac{L}{\mu(1-\sigma)})^2\log\frac{1}{\epsilon})$ complexity for gradient tracking over general undirected graphs. As a comparison, when each $f_{(i)}$ is strongly convex, the state-of-the-art complexity of gradient tracking is $\bO((\frac{L}{\mu}+\frac{1}{(1-\sigma)^2})\log\frac{1}{\epsilon})$ \citep{sulaiman2020}. Currently, it is unclear how to combine our techniques with those in \citep{Sun-2022} and we conjecture that the complexity would be higher than the one given in Theorem \ref{theorem2}. On the other hand, for some algorithms relying on multi-consensus \citep{mudag20,mudag2-20}, the weaker assumptions have no influence on the complexity.
\end{remark}

\subsection{Improved Dependence on the Network Connectivity Constants}\label{sec:subroutine}

As shown in Tables \ref{table-comp} and \ref{table-comp2}, the dependence on the network connectivity constants in our complexities is not optimal. We improve it over static graphs and time-varying graphs in the next two sections, respectively.
\subsubsection{Chebyshev Acceleration over Static Graphs}

Chebyshev acceleration was first used to accelerate distributed algorithms by \citet{dasent}, and it has become a standard technique now. Define the Chebyshev polynomials as $T_0(x)=1$, $T_1(x)=x$, and $T_{k+1}(x)=2xT_k(x)-T_{k-1}(x)$ for all $k\geq 1$. For symmetric $\W$, define $A=\I-\W$ with $2\geq\lambda_1\geq\lambda_2\geq ...\geq \lambda_{m-1}>\lambda_m=0$ being its eigenvalues. We know $\lambda_{m-1}=1-\sigma$. Define $\nu=\frac{\lambda_{m-1}}{\lambda_1}$, $c_1=\frac{1-\sqrt{\nu}}{1+\sqrt{\nu}}$, $c_2=\frac{1+\nu}{1-\nu}$, $c_3=\frac{2}{\lambda_1+\lambda_{m-1}}$, and $P_t(x)=1-\frac{T_t(c_2(1-x))}{T_t(c_2)}$. Then, $P_t(c_3A)$ is a symmetric matrix satisfying $P_t(c_3A)\1=0$ with its spectrum in $[1-\frac{2c_1^t}{1+c_1^{2t}},1+\frac{2c_1^t}{1+c_1^{2t}}]\cup 0$ \citep{Auzinger}. Let $t=\frac{1}{\sqrt{\nu}}$ so to have $c_1^t\leq \frac{1}{e}$ and $[1-\frac{2c_1^t}{1+c_1^{2t}},1+\frac{2c_1^t}{1+c_1^{2t}}]\subseteq [0.35,1.65]$. Thus, we can replace the fixed gossip matrix $W$ in algorithm (\ref{alg-tv-s1})-(\ref{alg-tv-s4}) by $\I-P_t(c_3A)$ because its second largest singular value $\sigma'$ satisfies $\sigma'\leq 0.65$, which is independent of $1-\sigma$. From Theorems \ref{theorem1} and \ref{theorem2} with $\sigma$ replaced by $\sigma'$, we see that the algorithm needs $\bO(\sqrt{\frac{L}{\epsilon}})$ iterations for nonstrongly convex problems and $\bO(\sqrt{\frac{L}{\mu}}\log\frac{1}{\epsilon})$ iterations for strongly convex ones to find an $\epsilon$-optimal solution, which corresponds to the gradient oracle complexity. On the other hand, we can compute $(\I-P_t(c_3A))\x$ by the following procedure \citep{dasent}:
\begin{algorithmic}
   \STATE Input: $\x$. Initialize: $a_0=1$, $a_1=c_2$, $\z^0=\x$, $\z^1=c_2(\I-c_3A)\x$.
   \FOR{$s=1,2,...,t-1$}
   \STATE $a_{s+1}=2c_2a_s-a_{s-1}$,
   \STATE $\z^{s+1}=2c_2(\I-c_3A)\z^s-\z^{s-1}$.
   \ENDFOR
   \STATE Output: $(\I-P_t(c_3A))\x=\frac{\z^t}{a_t}$.
\end{algorithmic}
Thus, the communication round complexities for nonstrongly convex and strongly convex problems are $\bO(\sqrt{\frac{L}{\epsilon(1-\sigma)}})$ and  $\bO(\sqrt{\frac{L}{\mu(1-\sigma)}}\log\frac{1}{\epsilon})$, respectively.

\begin{corollary}
Under the settings of Theorem \ref{theorem1} with symmetric and fixed gossip matrix $\W$, algorithm (\ref{alg-tv-s1})-(\ref{alg-tv-s4}) with Chebyshev acceleration requires time of $\bO(\sqrt{\frac{L}{\epsilon(1-\sigma)}})$ communication rounds and $\bO(\sqrt{\frac{L}{\epsilon}})$ gradient oracle calls to find an $\epsilon$-optimal averaged solution such that $F(\overline x)-F(x^*)\leq\epsilon$.
\end{corollary}
\begin{corollary}
Under the settings of Theorem \ref{theorem2} with symmetric and fixed gossip matrix $\W$, algorithm (\ref{alg-tv-s1})-(\ref{alg-tv-s4}) with Chebyshev acceleration requires time of $\bO(\sqrt{\frac{L}{\mu(1-\sigma)}}\log\frac{1}{\epsilon})$ communication rounds and $\bO(\sqrt{\frac{L}{\mu}}\log\frac{1}{\epsilon})$ gradient oracle calls to find an $\epsilon$-optimal averaged solution such that $F(\overline x)-F(x^*)\leq\epsilon$.
\end{corollary}

\subsubsection{Multiple Consensus over Time-varying Graphs}

Although Chebyshev acceleration has been widely used in decentralized optimization, it is unclear how to extend it to time-varying graphs. In this section, we use a multiple consensus subroutine as an alternative to improve the dependence on the network connectivity constants. Motivated by Chebyshev acceleration, our idea is to replace $\W^k$ in (\ref{alg-tv-s1})-(\ref{alg-tv-s4}) by virtual gossip matrices $\W^{k,\zeta}$ with carefully designed $\zeta$ such that
\begin{eqnarray}
\|\Pi\W^{k,\zeta}\x\|\leq\frac{1}{e}\|\Pi\x\|,\quad r=1,2,3.\notag
\end{eqnarray}
Here, $\frac{1}{e}$ can be replaced by any constant not close to 1. Then, it can be regarded as running the resultant algorithm over time-varying graphs with each graph instance being connected at every time, and moreover, $\sigma=\frac{1}{e}$. Note that we do not require the symmetry of the gossip matrices in Assumptions \ref{assumption_w} and \ref{assumption_w_tv}, thus our theorems apply to the virtual gossip matrices $\W^{k,\zeta}$. From Theorems \ref{theorem3} and \ref{theorem4} with $\gamma=1$ and $\sigma_{\gamma}=\frac{1}{e}$, we see that $\bO(\sqrt{\frac{L}{\epsilon}})$ iterations for nonstrongly convex problems and $\bO(\sqrt{\frac{L}{\mu}}\log\frac{1}{\epsilon})$ for strongly convex problems suffice to find an $\epsilon$-optimal solution, which correspond to the gradient oracle complexity. Next, we consider the communication round complexity. Letting $\zeta=\lceil\frac{\gamma}{1-\sigma_{\gamma}}\rceil$, it follows from (\ref{tau_consensus-contraction}) that
\begin{eqnarray}
\|\Pi \W^{k,\zeta}\x\|\leq \sigma_{\gamma}^{\frac{1}{1-\sigma_{\gamma}}}\|\Pi\x\|=\left(1-\left(1-\sigma_{\gamma}\right)\right)^{\frac{1}{1-\sigma_{\gamma}}}\|\Pi\x\|\leq\frac{1}{e}\|\Pi\x\|,\notag
\end{eqnarray}
where we use the fact that $(1-x)^{1/x}\leq 1/e$ for any $x\in (0,1)$. Since $\W^{k,\zeta}\x$ can be implemented by the multiple consensus subroutine
\begin{eqnarray}
\u^{t+1}=\W^t\u^t\notag
\end{eqnarray}
with $\zeta$ rounds of communications initialized at $\u^0=\x$, the communication round complexity is $\bO(\frac{\gamma}{1-\sigma_{\gamma}}\sqrt{\frac{L}{\epsilon}})$ for nonstrongly convex problems and $\bO(\frac{\gamma}{1-\sigma_{\gamma}}\sqrt{\frac{L}{\mu}}\log\frac{1}{\epsilon})$ for strongly convex ones, respectively.

\begin{corollary}
Under the settings of Theorem \ref{theorem3}, algorithm (\ref{alg-tv-s1})-(\ref{alg-tv-s4}) combined with the multiple consensus subroutine requires time of $\bO(\frac{\gamma}{1-\sigma_{\gamma}}\sqrt{\frac{L}{\epsilon}})$ communication rounds and $\bO(\sqrt{\frac{L}{\epsilon}})$ gradient oracle calls to find an $\epsilon$-optimal averaged solution such that $F(\overline x)-F(x^*)\leq\epsilon$.
\end{corollary}
\begin{corollary}
Under the settings of Theorem \ref{theorem4}, algorithm (\ref{alg-tv-s1})-(\ref{alg-tv-s4}) combined with the multiple consensus subroutine requires time of $\bO(\frac{\gamma}{1-\sigma_{\gamma}}\sqrt{\frac{L}{\mu}}\log\frac{1}{\epsilon})$ communication rounds and $\bO(\sqrt{\frac{L}{\mu}}\log\frac{1}{\epsilon})$ gradient oracle calls to find an $\epsilon$-optimal averaged solution such that $F(\overline x)-F(x^*)\leq\epsilon$.
\end{corollary}

\begin{remark}\label{remark4}
The multiple consensus subroutine is only for the theoretical purpose. It may be impractical in a realistic time-varying network because communication has been recognized as the major bottleneck in distributed optimization. The multiple consensus may place a larger communication burden in practice, although it gives theoretically lower communication round complexities. A similar issue also happens in APM \citep{li-2018-pm} and DAGD-C \citep{Alexander2221}, which also need a multiple consensus subroutine.

On the other hand, decentralized optimization over time-varying graphs is important because of two reasons. Firstly, in many applications, the communication network varies with time, and algorithms for this scenario are needed. Secondly, many other scenarios can be reformulated as optimization over time-varying graphs, such as asynchrony \citep{Spiridonoff20jmlr}, local SGD \citep{Koloskova20}, and sparsification \citep{chen20sparsification}. In these scenarios, the real network may be fixed, and the time-varying graphs are only used for analysis. So the single loop methods are much more favored.
\end{remark}

\begin{remark}
Unlike the scenario over static graphs, the communication round complexity lower bounds over $\gamma$-connected time-varying graphs have not been established, and it is unclear whether the $\bO(\frac{\gamma}{1-\sigma_{\gamma}}\sqrt{\frac{L}{\epsilon}})$ and $\bO(\frac{\gamma}{1-\sigma_{\gamma}}\sqrt{\frac{L}{\mu}}\log\frac{1}{\epsilon})$ communication round complexities can be further improved. We leave it as an open problem. On the other hand, \citet{Kovalev21tvlb} established the $\bO(\frac{1}{1-\sigma}\sqrt{\frac{L}{\mu}}\log\frac{1}{\epsilon})$ communication round complexity lower bound and the $\bO(\sqrt{\frac{L}{\mu}}\log\frac{1}{\epsilon})$ gradient oracle complexity lower bound for the special scenario of connected graphs at every time. That is, $\gamma=1$ in our scenario.
\end{remark}

\section{Proofs of Theorems}

In this section, we prove the theorems in Sections \ref{sec:acttv} and \ref{sec:actsg}. We first reformulate algorithm (\ref{alg-tv-s1})-(\ref{alg-tv-s4}) as the inexact accelerated gradient descent and give its convergence rates in Section \ref{sec:proof1}, and then bound the consensus errors. To help the readers get a quick start on our proof framework, we first bound the consensus errors over static graphs in Sections \ref{sec:proof2}, and then extend it to the time-varying graphs in Section \ref{sec:proof3}. The former scenario provides some basis and insights for the complex proofs of the latter.

\subsection{Convergence Rates of the Inexact Accelerated Gradient Descent}\label{sec:proof1}
Following the proof framework in \citep{Jakovetic-2014,qu2017-2}, we  multiply both sides of (\ref{alg-tv-s1})-(\ref{alg-tv-s4}) by $\frac{1}{m}\1^{\top}$ and use the definitions in (\ref{def_av_x}) and (\ref{aggregate}) to yield
\begin{subequations}
\begin{align}
&\overline y^k=\theta_k\overline z^k+(1-\theta_k)\overline x^k,\label{alg-ag-s1}\\
&\overline s^k-\frac{1}{m}\sum_{i=1}^m\nabla f_{(i)}(y_{(i)}^k)=\overline s^{k-1}-\frac{1}{m}\sum_{i=1}^m\nabla f_{(i)}(y_{(i)}^{k-1}),\label{alg-ag-s2}\\
&\overline z^{k+1}=\frac{1}{1+\frac{\mu\alpha}{\theta_k}}\left(\frac{\mu\alpha}{\theta_k}\overline y^k+\overline z^k-\frac{\alpha}{\theta_k}\overline s^k\right),\label{alg-ag-s3}\\
&\overline x^{k+1}=\theta_k\overline z^{k+1}+(1-\theta_k)\overline x^k,\label{alg-ag-s4}
\end{align}
\end{subequations}
where we use the column stochasticity of $\1^{\top}\W^k=\1^{\top}$. From the initialization $\s^0=\nabla f(\y^0)$ and (\ref{alg-ag-s2}), we have the following standard but important property in gradient tracking:
\begin{equation}
\overline s^k=\frac{1}{m}\sum_{i=1}^m\nabla f_{(i)}(y_{(i)}^k).\label{s-relation}
\end{equation}
Iterations (\ref{alg-ag-s1})-(\ref{alg-ag-s4}) can be regarded as the inexact accelerated gradient descent \citep{devolder2014first} in the sense that we use $\frac{1}{m}\sum_{i=1}^m\nabla f_{(i)}(y_{(i)}^k)$ as the descent direction, rather than the true gradient $\frac{1}{m}\sum_{i=1}^m\nabla f_{(i)}(\overline y^k)$. In fact, when we replace $\overline s^k$ in step (\ref{alg-ag-s3}) by the true gradient, steps (\ref{alg-ag-s1}), (\ref{alg-ag-s3}), and (\ref{alg-ag-s4}) reduce to the updates of the standard accelerated gradient descent, see \citep{Nesterov-2004,lin-book} for example.

The next lemma demonstrates the analogy properties of convexity and smoothness with the inexact gradients. The proof can be found in \citep{Jakovetic-2014,qu2017-2}. For the completeness and the readers' convenience, we give the proof in the appendix.
\begin{lemma}\label{lemma_inexact}
Define
\begin{eqnarray}
f(\overline y^k,\y^k)=\frac{1}{m}\sum_{i=1}^m \left( f_{(i)}(y_{(i)}^k)+\<\nabla f_{(i)}(y_{(i)}^k),\overline y^k-y_{(i)}^k\> \right).\label{def_inexact_f}
\end{eqnarray}
Suppose that Assumption \ref{assumption_f} holds. Then, we have for any $w$,
\begin{eqnarray}
&&F(w)\geq f(\overline y^k,\y^k)+\<\overline s^k,w-\overline y^k\>+\frac{\mu}{2}\|w-\overline y^k\|^2,\label{inexact_sc}\\
&&F(w)\leq f(\overline y^k,\y^k)+\<\overline s^k,w-\overline y^k\>+\frac{L}{2}\|w-\overline y^k\|^2+\frac{L}{2m}\|\Pi\y^k\|^2.\label{inexact_sm}
\end{eqnarray}
Especially, we allow $\mu$ to be zero.
\end{lemma}

Define the Bregman divergence as follows:
\begin{equation}
D_f(x,\y^k)=\frac{1}{m}\sum_{i=1}^m\left( f_{(i)}(x)-f_{(i)}(y_{(i)}^k)-\<\nabla f_{(i)}(y_{(i)}^k),x-y_{(i)}^k\> \right).\label{def_bregman}
\end{equation}

The next lemma gives the convergence rates of the inexact accelerated gradient descent. The techniques in this proof are standard, see \citep{lin-book} for example. The crucial difference is that we keep the Bregman divergence term $D_f(\overline x^k,\y^k)$ in (\ref{cont3}) and (\ref{cont4}), which is motivated by \citep{Tseng-2008}.

Compared with the standard accelerated gradient descent, for example, see \citep{Nesterov-2004,lin-book}, there are two additional error terms $(a)$ and $(c)$ in our lemma due to the inexact gradients. In the next two sections, we bound the two terms carefully by $(b)$ and $(d)$, respectively, such that the convergence rates of the accelerated gradient tracking match those of the classical centralized accelerated gradient descent, which is the main technical contribution of this paper compared with the existing work on accelerated gradient tracking in \citep{qu2017-2}.
\begin{lemma}\label{lemma4}
Suppose that Assumption \ref{assumption_f} with $\mu=0$ and part 2 of Assumption \ref{assumption_w_tv} hold. Let the sequence $\{\theta_k\}_{k=0}^K$ satisfy $\frac{1-\theta_k}{\theta_k^2}=\frac{1}{\theta_{k-1}^2}$ with $\theta_0=1$. Then for algorithm (\ref{alg-tv-s1})-(\ref{alg-tv-s4}), we have
\begin{eqnarray}
\begin{aligned}\label{cont3}
&\frac{F(\overline x^{K+1})-F(x^*)}{\theta_K^2}+\frac{1}{2\alpha}\|\overline z^{K+1}-x^*\|^2\leq \frac{1}{2\alpha}\|\overline z^0-x^*\|^2\\
&~~\qquad+\underbrace{\sum_{k=0}^K\frac{L}{2m\theta_k^2}\|\Pi\y^k\|^2}_{\text{\rm term (a)}}-\underbrace{\sum_{k=0}^K\left(\left(\frac{1}{2\alpha}-\frac{L}{2}\right)\|\overline z^{k+1}-\overline z^k\|^2+\frac{1}{\theta_{k-1}^2}D_f(\overline x^k,\y^k)\right)}_{\text{\rm term (b)}}.
\end{aligned}
\end{eqnarray}
Suppose that Assumption \ref{assumption_f} with $\mu>0$ and part 2 of Assumption \ref{assumption_w_tv} hold. Let $\theta_k=\theta=\frac{\sqrt{\alpha\mu}}{2}$ for all $k$ and assume that $\alpha\mu\leq 1$. Then for algorithm (\ref{alg-tv-s1})-(\ref{alg-tv-s4}), we have
\begin{eqnarray}
\begin{aligned}\label{cont4}
&\frac{1}{(1-\theta)^{K+1}}\left(F(\overline x^{K+1})-F(x^*)+\left(\frac{\theta^2}{2\alpha}+\frac{\mu\theta}{2}\right) \|\overline z^{K+1}-x^*\|^2\right)\\
&~~\leq F(\overline x^0)-F(x^*)+\left(\frac{\theta^2}{2\alpha}+\frac{\mu\theta}{2}\right)\|\overline z^0-x^*\|^2+\underbrace{\sum_{k=0}^K\frac{L}{2m(1-\theta)^{k+1}}\|\Pi\y^k\|^2}_{\text{\rm term (c)}}\\
&~~\quad-\underbrace{\sum_{k=0}^K\left(\frac{1}{(1-\theta)^{k+1}}\left(\frac{\theta^2}{2\alpha}-\frac{L\theta^2}{2}\right)\|\overline z^{k+1}-\overline z^k\|^2+\frac{1}{(1-\theta)^k}D_f(\overline x^k,\y^k)\right)}_{\text{\rm term (d)}}.
\end{aligned}
\end{eqnarray}
\end{lemma}
\begin{proof}
From the inexact smoothness (\ref{inexact_sm}), we have
\begin{eqnarray}
\begin{aligned}\label{cont1}
F(\overline x^{k+1})\leq& f(\overline y^k,\y^k)+\<\overline s^k,\overline x^{k+1}-\overline y^k\>+\frac{L}{2}\|\overline x^{k+1}-\overline y^k\|^2+\frac{L}{2m}\|\Pi\y^k\|^2\\
\overset{a}=& f(\overline y^k,\y^k)+\theta_k\<\overline s^k,\overline z^{k+1}-\overline z^k\>+\frac{L\theta_k^2}{2}\|\overline z^{k+1}-\overline z^k\|^2+\frac{L}{2m}\|\Pi\y^k\|^2\\
=& f(\overline y^k,\y^k)+\theta_k\<\overline s^k,x^*-\overline z^k\>+\theta_k\<\overline s^k,\overline z^{k+1}-x^*\>+\frac{L\theta_k^2}{2}\|\overline z^{k+1}-\overline z^k\|^2+\frac{L}{2m}\|\Pi\y^k\|^2,
\end{aligned}
\end{eqnarray}
where we use (\ref{alg-ag-s1}) and (\ref{alg-ag-s4}) in $\overset{a}=$. Next, we bound the two inner product terms. For the first inner product, we have
\begin{eqnarray}
\begin{aligned}\notag
&f(\overline y^k,\y^k)+\theta_k\<\overline s^k,x^*-\overline z^k\>\\
&~~\overset{b}=f(\overline y^k,\y^k)+\<\overline s^k,\theta_k x^*+(1-\theta_k)\overline x^k-\overline y^k\>\\
&~~=\theta_k\left(f(\overline y^k,\y^k)+\<\overline s^k,x^*-\overline y^k\>\right)+(1-\theta_k)\left(f(\overline y^k,\y^k)+\<\overline s^k,\overline x^k-\overline y^k\>\right)\\
&~~\overset{c}\leq\theta_k F(x^*)-\frac{\mu\theta_k}{2}\|\overline y^k-x^*\|^2+\frac{1-\theta_k}{m}\sum_{i=1}^m \left( f_{(i)}(y_{(i)}^k)+\<\nabla f_{(i)}(y_{(i)}^k),\overline x^k-y_{(i)}^k\> \right)\\
&~~=\theta_k F(x^*)-\frac{\mu\theta_k}{2}\|\overline y^k-x^*\|^2+(1-\theta_k)F(\overline x^k)\\
&~~\quad-\frac{1-\theta_k}{m}\sum_{i=1}^m \left( f_{(i)}(\overline x^k)-f_{(i)}(y_{(i)}^k)-\<\nabla f_{(i)}(y_{(i)}^k),\overline x^k-y_{(i)}^k\> \right)\\
&~~=\theta_k F(x^*)-\frac{\mu\theta_k}{2}\|\overline y^k-x^*\|^2+(1-\theta_k)F(\overline x^k)-(1-\theta_k)D_f(\overline x^k,\y^k),
\end{aligned}
\end{eqnarray}
where we use (\ref{alg-ag-s1}) in $\overset{b}=$, (\ref{inexact_sc}), (\ref{def_inexact_f}), and (\ref{s-relation}) in $\overset{c}\leq$. For the second inner product, we have
\begin{eqnarray}
\begin{aligned}\notag
\theta_k\<\overline s^k,\overline z^{k+1}-x^*\>\overset{d}=&-\frac{\theta_k^2}{\alpha}\<\overline z^{k+1}-\overline z^k+\frac{\mu\alpha}{\theta_k}(\overline z^{k+1}-\overline y^k),\overline z^{k+1}-x^*\>\\
=&\frac{\theta_k^2}{2\alpha}\left( \|\overline z^k-x^*\|^2-\|\overline z^{k+1}-x^*\|^2-\|\overline z^{k+1}-\overline z^k\|^2 \right)\\
&+\frac{\mu\theta_k}{2}\left( \|\overline y^k-x^*\|^2-\|\overline z^{k+1}-x^*\|^2-\|\overline z^{k+1}-\overline y^k\|^2 \right),
\end{aligned}
\end{eqnarray}
where we use (\ref{alg-ag-s3}) in $\overset{d}=$. Plugging into (\ref{cont1}) and rearranging the terms, it gives
\begin{eqnarray}
\begin{aligned}\label{cont2}
&F(\overline x^{k+1})-F(x^*)+\left(\frac{\theta_k^2}{2\alpha}+\frac{\mu\theta_k}{2}\right) \|\overline z^{k+1}-x^*\|^2\\
&~~\leq (1-\theta_k)(F(\overline x^k)-F(x^*))+\frac{\theta_k^2}{2\alpha}\|\overline z^k-x^*\|^2\\
&~~\quad-\left(\frac{\theta_k^2}{2\alpha}-\frac{L\theta_k^2}{2}\right)\|\overline z^{k+1}-\overline z^k\|^2-(1-\theta_k)D_f(\overline x^k,\y^k)+\frac{L}{2m}\|\Pi\y^k\|^2.
\end{aligned}
\end{eqnarray}

Case 1: Each $f_{(i)}$ is nonstrongly convex. In this case, (\ref{cont2}) holds with $\mu=0$. Dividing both sides of (\ref{cont2}) by $\theta_k^2$, summing over $k=0,1,...,K$, using $\frac{1-\theta_k}{\theta_k^2}=\frac{1}{\theta_{k-1}^2}$ and $\theta_0=1$, we have (\ref{cont3}).

Case 2: Each $f_{(i)}$ is $\mu$-strongly convex. Letting $\theta_k=\theta=\frac{\sqrt{\alpha\mu}}{2}$ for all $k$, we know $\frac{\theta^2}{2\alpha}\leq\left(\frac{\theta^2}{2\alpha}+\frac{\mu\theta}{2}\right)(1-\theta)$ holds if $\alpha\mu\leq 1$. Dividing both sides of (\ref{cont2}) by $(1-\theta)^{k+1}$ and summing over $k=0,1,...,K$, it gives (\ref{cont4}).
\end{proof}

\subsection{Bounding the Consensus Errors over Static Graphs}\label{sec:proof2}
In this section, we bound the term $(a)$ by $(b)$ appeared in (\ref{cont3}), and the term $(c)$ by $(d)$ in (\ref{cont4}) over static graphs. We first bound $\|\Pi\y^k\|^2$ in the next lemma. The crucial trick is to construct a linear combination of the consensus errors with carefully designed weights such that it shrinks geometrically with an additional error term. Moreover, the step size $\alpha$ remains to be a constant of the order $\bO(\frac{1}{L})$ as large as possible. Another trick is that we use a constant $\tau$ to balance $D_f(\overline x^{r+1},\y^{r+1})$ and $\|\overline z^{r+1}-\overline z^r\|^2$ in $\Phi^r$, which is generated by Young's inequality and will be specified later.

\begin{lemma}\label{lemma2}
Suppose that Assumptions \ref{assumption_f} and \ref{assumption_w} hold with $\mu\geq 0$. Let $\alpha\leq\frac{(1-\sigma)^3}{80L\sqrt{1+\frac{1}{\tau}}}$ and the sequence $\{\theta_k\}_{k=0}^K$ satisfy $\theta_{k+1}\leq \theta_k\leq 1$. Then for algorithm (\ref{alg-tv-s1})-(\ref{alg-tv-s4}) with fixed gossip matrix $\W$, we have
\begin{eqnarray}
\qquad\begin{aligned}\label{cont9}
\max\left\{\|\Pi\y^{k+1}\|^2,\|\Pi\x^{k+1}\|^2\right\}\leq C_1\rho^{k+1}+C_2\sum_{r=0}^k\rho^{k-r}\theta_r^2\Phi^r,
\end{aligned}
\end{eqnarray}
where $\rho=1-\frac{1-\sigma}{4}$, $C_1=\frac{(1-\sigma)^2}{18(1+\frac{1}{\tau})L^2}\|\Pi\s^0\|^2$, $C_2=\frac{1-\sigma}{9(1+\frac{1}{\tau})L^2}$,
\begin{equation}
\Phi^r=\frac{2mL(1+\tau)}{\theta_r^2}D_f(\overline x^{r+1},\y^{r+1})+2mL^2\left(1+\frac{1}{\tau}\right)\|\overline z^{r+1}-\overline z^r\|^2,\label{cont18}
\end{equation}
and $\tau$ can be any positive constant.
\end{lemma}
\begin{proof}
Multiplying both sides of (\ref{alg-tv-s1})-(\ref{alg-tv-s4}) by $\Pi$, using (\ref{consensus-contraction}) and $\|\Pi\x\|\leq\|\x\|$, we have
\begin{eqnarray}
&&\|\Pi\y^k\|\leq\theta_k\|\Pi\z^k\|+(1-\theta_k)\|\Pi\x^k\|\leq\theta_k\|\Pi\z^k\|+\|\Pi\x^k\|,\label{cont5}\\
&&\|\Pi\s^{k+1}\|\leq\sigma\|\Pi\s^k\|+\|\nabla f(\y^{k+1})-\nabla f(\y^k)\|,\label{cont6}\\
&&\begin{aligned}\label{cont7}
\|\Pi\z^{k+1}\|\leq&\sigma\left(\frac{\mu\alpha}{\theta_k+\mu\alpha}\|\Pi\y^k\|+\frac{\theta_k}{\theta_k+\mu\alpha}\|\Pi\z^k\|\right)+\frac{\alpha}{\theta_k+\mu\alpha}\|\Pi\s^k\|\\
\overset{a}\leq&\frac{\sigma(\mu\alpha+1)\theta_k}{\theta_k+\mu\alpha}\|\Pi\z^k\|+\frac{\sigma\mu\alpha}{\theta_k+\mu\alpha}\|\Pi\x^k\|+\frac{\alpha}{\theta_k+\mu\alpha}\|\Pi\s^k\|\\
\overset{b}\leq&\sigma\|\Pi\z^k\|+\frac{\mu\alpha}{\theta_k}\|\Pi\x^k\|+\frac{\alpha}{\theta_k}\|\Pi\s^k\|,
\end{aligned}\\
&&\|\Pi\x^{k+1}\|\leq\theta_k\|\Pi\z^{k+1}\|+\sigma\|\Pi\x^k\|,\label{cont8}
\end{eqnarray}
where $\overset{a}\leq$ uses (\ref{cont5}), $\overset{b}\leq$ uses $\sigma<1$ and $\frac{(\mu\alpha+1)\theta_k}{\theta_k+\mu\alpha}\leq 1$ with $\theta_k\leq 1$. Next, we bound $\|\nabla f(\y^{k+1})-\nabla f(\y^k)\|$.
\begin{eqnarray}
\begin{aligned}\label{cont15}
&\|\nabla f(\y^{k+1})-\nabla f(\y^k)\|^2=\sum_{i=1}^m\|\nabla f_{(i)}(y_{(i)}^{k+1})-\nabla f_{(i)}(y_{(i)}^k)\|^2\\
&~~\overset{c}\leq\sum_{i=1}^m(1+\tau)\|\nabla f_{(i)}(y_{(i)}^{k+1})-\nabla f_{(i)}(\overline x^{k+1})\|^2\\
&~~\quad+\sum_{i=1}^m\left(1+\frac{1}{\tau}\right)2\left(\|\nabla f_{(i)}(\overline x^{k+1})-\nabla f_{(i)}(\overline y^k)\|^2+\|\nabla f_{(i)}(\overline y^k)-\nabla f_{(i)}(y_{(i)}^k)\|^2\right)\\
&~~\overset{d}\leq2mL(1+\tau)D_f(\overline x^{k+1},\y^{k+1})+2L^2\left(1+\frac{1}{\tau}\right)\sum_{i=1}^m\left(\|\overline x^{k+1}-\overline y^k\|^2+\|\overline y^k-y_{(i)}^k\|^2\right)\\
&~~\overset{e}=2mL(1+\tau)D_f(\overline x^{k+1},\y^{k+1})+2L^2\left(1+\frac{1}{\tau}\right)\left(m\theta_k^2\|\overline z^{k+1}-\overline z^k\|^2+\|\Pi\y^k\|^2\right)\\
&~~=\theta_k^2\Phi^k+2L^2\left(1+\frac{1}{\tau}\right)\|\Pi\y^k\|^2\\
&~~\overset{f}\leq\theta_k^2\Phi^k+4L^2\left(1+\frac{1}{\tau}\right)\left(\theta_k^2\|\Pi\z^k\|^2+\|\Pi\x^k\|^2\right),
\end{aligned}
\end{eqnarray}
where $\overset{c}\leq$ uses Young's inequality of $\|a-b\|^2\leq (1+\tau)\|a\|^2+(1+\frac{1}{\tau})\|b\|^2$ for any $\tau>0$, $\overset{d}\leq$ uses (\ref{cont43}), the smoothness of $f_{(i)}$, and the definition of $D_f$ in (\ref{def_bregman}), $\overset{e}=$ uses (\ref{alg-ag-s1}), (\ref{alg-ag-s4}), and the definition of $\Pi\y$ in (\ref{def_pi}), $\overset{f}\leq$ uses (\ref{cont5}). Denote $c_0=4L^2\left(1+\frac{1}{\tau}\right)$ for simplicity in the remaining proof of this lemma.

\noindent Squaring both sides of (\ref{cont6}), it follows that
\begin{eqnarray}
\begin{aligned}\label{cont10}
\|\Pi\s^{k+1}\|^2\leq&\left(1+\frac{1-\sigma}{2\sigma}\right)\sigma^2\|\Pi\s^k\|^2+\left(1+\frac{2\sigma}{1-\sigma}\right)\|\nabla f(\y^{k+1})-\nabla f(\y^k)\|^2\\
=&\frac{\sigma+\sigma^2}{2}\|\Pi\s^k\|^2+\frac{1+\sigma}{1-\sigma}\|\nabla f(\y^{k+1})-\nabla f(\y^k)\|^2\\
\overset{g}\leq&\frac{1+\sigma}{2}\|\Pi\s^k\|^2+\frac{2}{1-\sigma}\left(\theta_k^2\Phi^k+c_0\theta_k^2\|\Pi\z^k\|^2+c_0\|\Pi\x^k\|^2\right),
\end{aligned}
\end{eqnarray}
where we use $\sigma<1$ and (\ref{cont15}) in $\overset{g}\leq$. Similarly, for (\ref{cont7}) and (\ref{cont8}), we also have
\begin{eqnarray}
&&\|\Pi\z^{k+1}\|^2\leq\frac{1+\sigma}{2}\|\Pi\z^k\|^2+\frac{4}{1-\sigma}\left(\frac{\mu^2\alpha^2}{\theta_k^2}\|\Pi\x^k\|^2+\frac{\alpha^2}{\theta_k^2}\|\Pi\s^k\|^2\right),\label{cont11}\\
&&\|\Pi\x^{k+1}\|^2\leq\frac{1+\sigma}{2}\|\Pi\x^k\|^2+\frac{2\theta_k^2}{1-\sigma}\|\Pi\z^{k+1}\|^2.\label{cont12}
\end{eqnarray}
Adding (\ref{cont10}), (\ref{cont11}), and (\ref{cont12}) together with the weights $c_1$, $c_2\theta_{k+1}^2$, and $c_3$, respectively, we have
\begin{eqnarray}
\begin{aligned}
&c_1\|\Pi\s^{k+1}\|^2+c_2\theta_{k+1}^2\|\Pi\z^{k+1}\|^2+c_3\|\Pi\x^{k+1}\|^2\notag\\
&~~\overset{h}\leq c_1\|\Pi\s^{k+1}\|^2+\left(c_2\theta_k^2+\frac{2c_3\theta_k^2}{1-\sigma}\right)\|\Pi\z^{k+1}\|^2+\frac{c_3(1+\sigma)}{2}\|\Pi\x^k\|^2\notag\\
&~~\leq\left(\frac{c_1(1+\sigma)}{2}+\left(c_2+\frac{2c_3}{1-\sigma}\right)\frac{4\alpha^2}{1-\sigma}\right)\|\Pi\s^k\|^2\notag\\
&~~\quad+\left(\frac{2c_0c_1}{1-\sigma}+\left(c_2+\frac{2c_3}{1-\sigma}\right)\frac{1+\sigma}{2}\right)\theta_k^2\|\Pi\z^k\|^2\notag\\
&~~\quad+\left(\frac{2c_0c_1}{1-\sigma}+\left(c_2+\frac{2c_3}{1-\sigma}\right)\frac{4\mu^2\alpha^2}{1-\sigma}+\frac{c_3(1+\sigma)}{2}\right)\|\Pi\x^k\|^2+\frac{2c_1}{1-\sigma}\theta_k^2\Phi^k,\notag
\end{aligned}
\end{eqnarray}
where we use $\theta_{k+1}\leq\theta_k$ and (\ref{cont12}) in $\overset{h}\leq$. Letting $c_3=\frac{9c_0c_1}{(1-\sigma)^2}$, $c_2=\frac{80c_0c_1}{(1-\sigma)^4}\geq\frac{8c_0c_1}{(1-\sigma)^2}+\frac{8c_3}{(1-\sigma)^2}$, and $\alpha^2\leq\min\left\{\frac{(1-\sigma)^6}{1600c_0},\frac{(1-\sigma)^4}{1600\mu^2}\right\}$ such that
\begin{eqnarray}
&&\frac{c_1(1+\sigma)}{2}+\left(c_2+\frac{2c_3}{1-\sigma}\right)\frac{4\alpha^2}{1-\sigma}\leq\frac{c_1(1+\sigma)}{2}+\frac{400c_0c_1\alpha^2}{(1-\sigma)^5}\leq \frac{c_1(3+\sigma)}{4},\notag\\
&&\frac{2c_0c_1}{1-\sigma}+\left(c_2+\frac{2c_3}{1-\sigma}\right)\frac{1+\sigma}{2}\leq \frac{2c_0c_1}{1-\sigma}+\frac{c_2(1+\sigma)}{2}+\frac{2c_3}{1-\sigma}\leq\frac{c_2(3+\sigma)}{4},\notag\\
&&\frac{2c_0c_1}{1\hspace*{-0.03cm}-\hspace*{-0.03cm}\sigma}\hspace*{-0.03cm}+\hspace*{-0.03cm}\left(\hspace*{-0.03cm}c_2\hspace*{-0.03cm}+\hspace*{-0.03cm}\frac{2c_3}{1\hspace*{-0.03cm}-\hspace*{-0.03cm}\sigma}\hspace*{-0.03cm}\right)\hspace*{-0.03cm}\frac{4\mu^2\alpha^2}{1\hspace*{-0.03cm}-\hspace*{-0.03cm}\sigma}\hspace*{-0.03cm}+\hspace*{-0.03cm}\frac{c_3(1\hspace*{-0.03cm}+\hspace*{-0.03cm}\sigma)}{2}\hspace*{-0.03cm}\leq\hspace*{-0.03cm}\frac{2c_0c_1}{1\hspace*{-0.03cm}-\hspace*{-0.03cm}\sigma}\hspace*{-0.03cm}+\hspace*{-0.03cm}\frac{400c_0c_1\mu^2\alpha^2}{(1\hspace*{-0.03cm}-\hspace*{-0.03cm}\sigma)^5}\hspace*{-0.03cm}+\hspace*{-0.03cm}\frac{c_3(1\hspace*{-0.03cm}+\hspace*{-0.03cm}\sigma)}{2}\hspace*{-0.03cm}\leq\hspace*{-0.03cm} \frac{c_3(3\hspace*{-0.03cm}+\hspace*{-0.03cm}\sigma)}{4}\hspace*{-0.03cm},\notag
\end{eqnarray}
we have
\begin{eqnarray}
\begin{aligned}\notag
&c_1\|\Pi\s^{k+1}\|^2+c_2\theta_{k+1}^2\|\Pi\z^{k+1}\|^2+c_3\|\Pi\x^{k+1}\|^2\\
&~~\leq\frac{3+\sigma}{4}\left(c_1\|\Pi\s^k\|^2+c_2\theta_k^2\|\Pi\z^k\|^2+c_3\|\Pi\x^k\|^2\right)+\frac{2c_1}{1-\sigma}\theta_k^2\Phi^k\\
&~~\leq\left(\frac{3+\sigma}{4}\right)^{k+1}\left(c_1\|\Pi\s^0\|^2+c_2\theta_0^2\|\Pi\z^0\|^2+c_3\|\Pi\x^0\|^2\right)+\frac{2c_1}{1-\sigma}\sum_{r=0}^k\left(\frac{3+\sigma}{4}\right)^{k-r}\theta_r^2\Phi^r.
\end{aligned}
\end{eqnarray}
From (\ref{cont5}), $c_2>c_3$, and the initialization such that $\Pi\x^0=\Pi\y^0=\Pi\z^0=0$, we have
\begin{eqnarray}
\begin{aligned}\notag
\|\Pi\y^{k+1}\|^2\leq&\frac{2}{c_3}\left(c_1\|\Pi\s^{k+1}\|^2+c_2\theta_{k+1}^2\|\Pi\z^{k+1}\|^2+c_3\|\Pi\x^{k+1}\|^2\right)\\
\leq&\left(\frac{3+\sigma}{4}\right)^{k+1}\frac{2c_1}{c_3}\|\Pi\s^0\|^2+\frac{4c_1}{c_3(1-\sigma)} \sum_{r=0}^k\left(\frac{3+\sigma}{4}\right)^{k-r}\theta_r^2\Phi^r\\
=&\left(\frac{3+\sigma}{4}\right)^{k+1}\frac{(1-\sigma)^2}{18(1+\frac{1}{\tau})L^2}\|\Pi\s^0\|^2+\frac{1-\sigma}{9(1+\frac{1}{\tau})L^2} \sum_{r=0}^k\left(\frac{3+\sigma}{4}\right)^{k-r}\theta_r^2\Phi^r,
\end{aligned}
\end{eqnarray}
which is exactly (\ref{cont9}).
\end{proof}

Having (\ref{cont9}) at hand, we are ready to bound the term $(a)$ by $(b)$ appeared in (\ref{cont3}). The remaining challenge is to upper bound the weighted cumulative consensus errors.
\begin{lemma}\label{lemma6}
Suppose that Assumptions \ref{assumption_f} and \ref{assumption_w} hold with $\mu=0$. Let the sequence $\{\theta_k\}_{k=0}^K$ satisfy $\frac{1-\theta_k}{\theta_k^2}=\frac{1}{\theta_{k-1}^2}$ with $\theta_0=1$, let $\alpha\leq\frac{(1-\sigma)^3}{80L\sqrt{1+\frac{1}{\tau}}}$. Then for algorithm (\ref{alg-tv-s1})-(\ref{alg-tv-s4}) with fixed gossip matrix $\W$, we have
\begin{eqnarray}
\begin{aligned}
&\max\left\{\sum_{k=0}^K\frac{L}{2m\theta_k^2}\|\Pi\y^k\|^2,\sum_{k=0}^K\frac{L}{2m\theta_k^2}\|\Pi\x^k\|^2\right\}\\
&~~\leq\frac{16}{3mL(1+\frac{1}{\tau})(1-\sigma)}\|\Pi\s^0\|^2+\frac{11}{mL(1+\frac{1}{\tau})(1-\sigma)^2}\sum_{r=0}^{K-1}\Phi^r,\label{cont28}
\end{aligned}
\end{eqnarray}
where $\tau$ and $\Phi^r$ are defined in Lemma \ref{lemma2}.
\end{lemma}
\begin{proof}
We first give some properties of the sequence $\{\theta_k\}_{k=0}^K$. From $\frac{1-\theta_k}{\theta_k^2}=\frac{1}{\theta_{k-1}^2}$ and $\theta_0=1$, we have $\theta_k\leq\theta_{k-1}$, $\frac{1}{\theta_k}-1\leq \frac{1}{\theta_{k-1}}$, and $\frac{1}{\theta_k}-\frac{1}{2}\geq \frac{1}{\theta_{k-1}}$, which further give
\begin{eqnarray}
\frac{1}{\theta_k}-\frac{1}{\theta_{k-1}}\leq 1,\qquad \frac{1}{k+1}\leq\theta_k\leq \frac{2}{k+1}.\label{theta_pro}
\end{eqnarray}
From (\ref{cont9}), we get
\begin{eqnarray}
\begin{aligned}\label{cont27}
\sum_{k=1}^K\frac{L}{2m\theta_k^2}\|\Pi\y^k\|^2\leq& \sum_{k=1}^K\frac{C_1L\rho^k}{2m\theta_k^2}+\sum_{k=1}^K\frac{C_2L}{2m\theta_k^2} \sum_{r=0}^{k-1}\rho^{k-1-r}\theta_r^2\Phi^r\\
=&\frac{C_1L}{2m}\sum_{k=1}^K\frac{\rho^k}{\theta_k^2}+\frac{C_2L}{2m\rho}\sum_{k=1}^K\frac{\rho^k}{\theta_k^2} \sum_{r=0}^{k-1}\frac{\theta_r^2}{\rho^r}\Phi^r\\
=&\frac{C_1L}{2m}\sum_{k=1}^K\frac{\rho^k}{\theta_k^2}+\frac{C_2L}{2m\rho}\sum_{r=0}^{K-1}\frac{\theta_r^2}{\rho^r}\Phi^r\sum_{k=r+1}^K\frac{\rho^k}{\theta_k^2}.
\end{aligned}
\end{eqnarray}
Recall that for scalars, $\theta_k$ means the value at iteration $k$, while $\rho^k$ is its $k$th power. Next, we compute $\sum_{k=r+1}^K\frac{\rho^k}{\theta_k^2}$ for any $r\geq 0$. Denote $S=\sum_{k=r+1}^K\frac{\rho^k}{\theta_k^2}$ for simplicity. We have
\begin{equation}
\rho S=\sum_{k=r+1}^K\frac{\rho^{k+1}}{\theta_k^2}=\sum_{k=r+1}^K\frac{\rho^k}{\theta_{k-1}^2}-\frac{\rho^{r+1}}{\theta_r^2}+\frac{\rho^{K+1}}{\theta_K^2},\notag
\end{equation}
and
\begin{equation}
S-\rho S=\sum_{k=r+1}^K\rho^k\left(\frac{1}{\theta_k^2}-\frac{1}{\theta_{k-1}^2}\right)+\frac{\rho^{r+1}}{\theta_r^2}-\frac{\rho^{K+1}}{\theta_K^2}\overset{a}=\sum_{k=r+1}^K\frac{\rho^k}{\theta_k}+\frac{\rho^{r+1}}{\theta_r^2}-\frac{\rho^{K+1}}{\theta_K^2},\notag
\end{equation}
where we use $\frac{1-\theta_k}{\theta_k^2}=\frac{1}{\theta_{k-1}^2}$ in $\overset{a}=$. It further gives
\begin{equation}
\rho(1-\rho) S=\sum_{k=r+1}^K\frac{\rho^{k+1}}{\theta_k}+\frac{\rho^{r+2}}{\theta_r^2}-\frac{\rho^{K+2}}{\theta_K^2}=\sum_{k=r+1}^K\frac{\rho^k}{\theta_{k-1}}-\frac{\rho^{r+1}}{\theta_r}+\frac{\rho^{K+1}}{\theta_K}+\frac{\rho^{r+2}}{\theta_r^2}-\frac{\rho^{K+2}}{\theta_K^2},\notag
\end{equation}
and
\begin{eqnarray}
\begin{aligned}
(1-\rho)^2 S=&(1-\rho)S-\rho (1-\rho) S\\
=&\sum_{k=r+1}^K\rho^k\left(\frac{1}{\theta_k}-\frac{1}{\theta_{k-1}}\right)+\frac{\rho^{r+1}}{\theta_r^2}-\frac{\rho^{K+1}}{\theta_K^2}+\frac{\rho^{r+1}}{\theta_r}-\frac{\rho^{K+1}}{\theta_K}-\frac{\rho^{r+2}}{\theta_r^2}+\frac{\rho^{K+2}}{\theta_K^2}\notag\\
=&\sum_{k=r+1}^K\rho^k\left(\frac{1}{\theta_k}-\frac{1}{\theta_{k-1}}\right)+\frac{(1-\rho)\rho^{r+1}}{\theta_r^2}-\frac{(1-\rho)\rho^{K+1}}{\theta_K^2}+\frac{\rho^{r+1}}{\theta_r}-\frac{\rho^{K+1}}{\theta_K}\\
\overset{b}\leq&\sum_{k=r+1}^K\rho^k+\frac{(1-\rho)\rho^{r+1}}{\theta_r^2}+\frac{\rho^{r+1}}{\theta_r}\leq \frac{\rho^{r+1}}{1-\rho}+\frac{2\rho^{r+1}}{\theta_r^2},
\end{aligned}
\end{eqnarray}
where we use (\ref{theta_pro}) in $\overset{b}\leq$. Thus, we get
\begin{eqnarray}
\sum_{k=r+1}^K\frac{\rho^k}{\theta_k^2}\leq \frac{1}{(1-\rho)^2}\left(\frac{\rho^{r+1}}{1-\rho}+\frac{2\rho^{r+1}}{\theta_r^2}\right)\leq\frac{3\rho^{r+1}}{(1-\rho)^3\theta_r^2}.\label{cont25}
\end{eqnarray}
Plugging into (\ref{cont27}), it follows from $\Pi\y^0=0$ and $\theta_0=1$ that
\begin{eqnarray}
\begin{aligned}\notag
\sum_{k=0}^K\frac{L}{2m\theta_k^2}\|\Pi\y^k\|^2\leq&\frac{3C_1L\rho}{2m(1-\rho)^3}+\frac{3C_2L}{2m(1-\rho)^3}\sum_{r=0}^{K-1}\Phi^r\\
\leq&\frac{16}{3mL(1+\frac{1}{\tau})(1-\sigma)}\|\Pi\s^0\|^2+\frac{11}{mL(1+\frac{1}{\tau})(1-\sigma)^2}\sum_{r=0}^{K-1}\Phi^r,
\end{aligned}
\end{eqnarray}
where the last inequality uses the definitions of $C_1$, $C_2$, and $\rho$ given in Lemma \ref{lemma2}. Replacing $\|\Pi\y^k\|$ by $\|\Pi\x^k\|$ in the above analysis, we have the same bound for $\sum_{k=0}^K\frac{L}{2m\theta_k^2}\|\Pi\x^k\|^2$.
\end{proof}

In the next lemma, we bound the term $(c)$ by $(d)$ appeared in (\ref{cont4}) in a similar way to the proof of Lemma \ref{lemma6}.
\begin{lemma}\label{lemma7}
Suppose that Assumptions \ref{assumption_f} and \ref{assumption_w} hold with $\mu>0$. Let $\alpha\leq\frac{(1-\sigma)^3}{80L\sqrt{1+\frac{1}{\tau}}}$ and $\theta_k\equiv\theta=\frac{\sqrt{\mu\alpha}}{2}$. Then for algorithm (\ref{alg-tv-s1})-(\ref{alg-tv-s4}) with fixed gossip matrix $\W$, we have
\begin{eqnarray}
\begin{aligned}\label{cont29}
&\max\left\{\sum_{k=0}^K\frac{L}{2m(1-\theta)^{k+1}}\|\Pi\y^k\|^2,\sum_{k=0}^K\frac{L}{2m(1-\theta)^{k+1}}\|\Pi\x^k\|^2\right\}\\
&~~\leq \frac{4(1-\sigma)}{27mL(1+\frac{1}{\tau})(1-\theta)}\|\Pi\s^0\|^2+\frac{8\theta^2}{27mL(1+\frac{1}{\tau})}\sum_{r=0}^{K-1}\frac{\Phi^r}{(1-\theta)^{r+1}},
\end{aligned}
\end{eqnarray}
where $\tau$ and $\Phi^r$ are defined in Lemma \ref{lemma2}.
\end{lemma}
\begin{proof}
From (\ref{cont9}), we get
\begin{eqnarray}
\begin{aligned}\notag
&\sum_{k=1}^K\frac{L}{2m(1-\theta)^{k+1}}\|\Pi\y^k\|^2\\
&~~\leq \sum_{k=1}^K\frac{C_1L\rho^k}{2m(1-\theta)^{k+1}}+\sum_{k=1}^K\frac{C_2L}{2m(1-\theta)^{k+1}} \sum_{r=0}^{k-1}\rho^{k-1-r}\theta^2\Phi^r\\
&~~=\frac{C_1L}{2m(1-\theta)}\sum_{k=1}^K\left(\frac{\rho}{1-\theta}\right)^k+\frac{\theta^2C_2L}{2m\rho(1-\theta)}\sum_{k=1}^K\left(\frac{\rho}{1-\theta}\right)^k\sum_{r=0}^{k-1}\frac{\Phi^r}{\rho^{r}}\\
&~~=\frac{C_1L}{2m(1-\theta)}\sum_{k=1}^K\left(\frac{\rho}{1-\theta}\right)^k+\frac{\theta^2C_2L}{2m\rho(1-\theta)}\sum_{r=0}^{K-1}\frac{\Phi^r}{\rho^{r}}\sum_{k=r+1}^K\left(\frac{\rho}{1-\theta}\right)^k.
\end{aligned}
\end{eqnarray}
From the settings of $\theta$ and $\alpha$, we know $\theta\leq\frac{1-\sigma}{16}$. Thus we have $\frac{\rho}{1-\theta}< 1$, $1-\theta-\rho\geq\frac{3(1-\sigma)}{16}$, and $\sum_{k=r+1}^K\left(\frac{\rho}{1-\theta}\right)^k\leq \left(\frac{\rho}{1-\theta}\right)^r\frac{\rho}{1-\theta-\rho}$ with $\rho=1-\frac{1-\sigma}{4}$. It follows from $\Pi\y^0=0$ that
\begin{eqnarray}
\begin{aligned}\notag
\sum_{k=0}^K\frac{L}{2m(1-\theta)^{k+1}}\|\Pi\y^k\|^2\leq& \frac{C_1L}{2m(1-\theta)}\frac{\rho}{1-\theta-\rho}+\frac{\theta^2C_2L}{2m(1-\theta-\rho)}\sum_{r=0}^{K-1}\frac{\Phi^r}{(1-\theta)^{r+1}}\\
\leq& \frac{4(1-\sigma)}{27mL(1+\frac{1}{\tau})(1-\theta)}\|\Pi\s^0\|^2+\frac{8\theta^2}{27mL(1+\frac{1}{\tau})}\sum_{r=0}^{K-1}\frac{\Phi^r}{(1-\theta)^{r+1}},
\end{aligned}
\end{eqnarray}
where the last inequality uses the definitions of $C_1$ and $C_2$ given in Lemma \ref{lemma2} and $1-\theta-\rho\geq\frac{3(1-\sigma)}{16}$. Replacing $\|\Pi\y^k\|$ by $\|\Pi\x^k\|$ in the above analysis, we have the same bound for $\Pi\x^k$.
\end{proof}

Now, we are ready to prove Theorems \ref{theorem1} and \ref{theorem2}. We first prove Theorem \ref{theorem1}. The crucial trick in this proof is to make the constant before $D_f(\overline x^k,\y^k)$ positive by setting the constant $\tau$ small, and make the constant before $\|\overline z^{t+1}-\overline z^t\|^2$ positive by setting the step size $\alpha$ small. This is the reason why we introduce the constant $\tau$ in the definition of $\Psi^r$ in (\ref{cont18}).
\begin{proof}
Plugging (\ref{cont28}) into (\ref{cont3}) and using the definition of $\Phi^r$ in (\ref{cont18}), we obtain
\begin{eqnarray}
\begin{aligned}\notag
&\frac{F(\overline x^{K+1})-F(x^*)}{\theta_K^2}+\frac{1}{2\alpha}\|\overline z^{K+1}-x^*\|^2\\
&~~\leq \frac{1}{2\alpha}\|\overline z^0-x^*\|^2+\frac{16}{3mL(1+\frac{1}{\tau})(1-\sigma)}\|\Pi\s^0\|^2\\
&~~\quad-\sum_{k=0}^K\left(\hspace*{-0.05cm}\left(\frac{1}{2\alpha}-\frac{L}{2}-\frac{22L}{(1-\sigma)^2}\right)\|\overline z^{t+1}-\overline z^t\|^2+\frac{1}{\theta_{k-1}^2}\left(\hspace*{-0.05cm}1-\frac{22(1+\tau)}{(1+\frac{1}{\tau})(1-\sigma)^2}\right)D_f(\overline x^k,\y^k)\hspace*{-0.05cm}\right)\\
&~~\overset{a}\leq \frac{1}{2\alpha}\|\overline z^0-x^*\|^2+\hspace*{-0.05cm}\frac{16}{3mL(1+\frac{1}{\tau})(1-\sigma)}\|\Pi\s^0\|^2-\hspace*{-0.05cm}\sum_{k=0}^K\hspace*{-0.05cm}\left(\hspace*{-0.05cm}\frac{1}{4\alpha}\|\overline z^{t+1}-\overline z^t\|^2+\frac{1}{2\theta_{k-1}^2}D_f(\overline x^k,\y^k)\hspace*{-0.05cm}\right)\\
&~~\leq \frac{1}{2\alpha}\|\overline z^0-x^*\|^2+\frac{1-\sigma}{8mL}\|\Pi\s^0\|^2-\frac{1}{5mL}\sum_{r=0}^{K-1}\Phi^r,
\end{aligned}
\end{eqnarray}
where in $\overset{a}\leq$ we let $\tau=\frac{(1-\sigma)^2}{44}$ so to have $\frac{22(1+\tau)}{(1+\frac{1}{\tau})(1-\sigma)^2}=\frac{1}{2}$, $\alpha\leq\frac{(1-\sigma)^4}{537L}\leq\frac{(1-\sigma)^3}{80L\sqrt{1+\frac{1}{\tau}}}$, and $\frac{1}{4\alpha}\geq\frac{L}{2}+\frac{22L}{(1-\sigma)^2}$. So we have
\begin{eqnarray}
\begin{aligned}\notag
&F(\overline x^{K+1})-F(x^*)\leq\theta_K^2\left(\frac{1}{2\alpha}\|\overline z^0-x^*\|^2+\frac{1-\sigma}{8mL}\|\Pi\s^0\|^2\right),\\
&\frac{1}{5mL}\sum_{r=0}^{K-1}\Phi^r\leq \frac{1}{2\alpha}\|\overline z^0-x^*\|^2+\frac{1-\sigma}{8mL}\|\Pi\s^0\|^2.
\end{aligned}
\end{eqnarray}
It follows from (\ref{cont28}) that
\begin{eqnarray}
\begin{aligned}\notag
&\max\left\{\sum_{k=0}^K\frac{L}{2m\theta_k^2}\|\Pi\y^k\|^2,\sum_{k=0}^K\frac{L}{2m\theta_k^2}\|\Pi\x^k\|^2\right\}\\
&~~\leq\frac{16}{3mL(1+\frac{1}{\tau})(1-\sigma)}\|\Pi\s^0\|^2+\frac{11}{mL(1+\frac{1}{\tau})(1-\sigma)^2}\sum_{r=0}^{K-1}\Phi^r\\
&~~\leq\frac{1-\sigma}{8mL}\|\Pi\s^0\|^2+\frac{1}{4mL}\sum_{r=0}^{K-1}\Phi^r\\
&~~\leq\frac{5}{8\alpha}\|\overline z^0-x^*\|^2+\frac{9(1-\sigma)}{32mL}\|\Pi\s^0\|^2.
\end{aligned}
\end{eqnarray}
From (\ref{theta_pro}), we have the conclusions.
\end{proof}

Next, we prove Theorem \ref{theorem2}.
\begin{proof}
Plugging (\ref{cont29}) into (\ref{cont4}) and using the definition of $\Phi^r$ in (\ref{cont18}), we have
\begin{eqnarray}
\begin{aligned}\notag
&\frac{1}{(1-\theta)^{K+1}}\left(F(\overline x^{K+1})-F(x^*)+\left(\frac{\theta^2}{2\alpha}+\frac{\mu\theta}{2}\right)\|\overline z^{K+1}-x^*\|^2\right)\\
&~~\leq F(\overline x^0)-F(x^*)+\left(\frac{\theta^2}{2\alpha}+\frac{\mu\theta}{2}\right)\|\overline z^0-x^*\|^2+\frac{4(1-\sigma)}{27mL(1+\frac{1}{\tau})(1-\theta)}\|\Pi\s^0\|^2\\
&~~\quad-\sum_{k=0}^K\left(\frac{1}{(1-\theta)^k}\left(1-\frac{16(1+\tau)}{27(1+\frac{1}{\tau})}\right)D_f(\overline x^k,\y^k)+\frac{1}{(1-\theta)^{k+1}}\left(\frac{\theta^2}{2\alpha}-\frac{L\theta^2}{2}-\frac{16L\theta^2}{27}\right)\|\overline z^{k+1}-\overline z^k\|^2\right)\\
&~~\overset{a}\leq F(\overline x^0)-F(x^*)+\left(\frac{\theta^2}{2\alpha}+\frac{\mu\theta}{2}\right)\|\overline z^0-x^*\|^2+\frac{4(1-\sigma)}{27m(1+\frac{1}{\tau})L(1-\theta)}\|\Pi\s^0\|^2\\
&~~\quad-\sum_{k=0}^K\left(\frac{1}{2(1-\theta)^k}D_f(\overline x^k,\y^k)+\frac{\theta^2}{4\alpha(1-\theta)^{k+1}}\|\overline z^{k+1}-\overline z^k\|^2\right)\\
&~~\leq F(\overline x^0)\hspace*{-0.025cm}-\hspace*{-0.025cm}F(x^*)\hspace*{-0.025cm}+\hspace*{-0.025cm}\left(\hspace*{-0.025cm}\frac{\theta^2}{2\alpha}\hspace*{-0.025cm}+\hspace*{-0.025cm}\frac{\mu\theta}{2}\hspace*{-0.025cm}\right)\hspace*{-0.025cm}\|\overline z^0\hspace*{-0.025cm}-\hspace*{-0.025cm}x^*\|^2\hspace*{-0.025cm}+\hspace*{-0.025cm}\frac{4(1-\sigma)}{59mL(1\hspace*{-0.025cm}-\hspace*{-0.025cm}\theta)}\|\Pi\s^0\|^2\hspace*{-0.025cm}-\hspace*{-0.025cm}\frac{8\theta^2}{59mL}\sum_{r=0}^{K-1}\frac{\Phi^r}{(1\hspace*{-0.025cm}-\hspace*{-0.025cm}\theta)^{r+1}},
\end{aligned}
\end{eqnarray}
where in $\overset{a}\leq$ we let $\tau=\frac{27}{32}$ so to have $\frac{16(1+\tau)}{27(1+\frac{1}{\tau})}=\frac{1}{2}$, $\alpha\leq\frac{(1-\sigma)^3}{119L}\leq\frac{(1-\sigma)^3}{80L\sqrt{1+\frac{1}{\tau}}}$, and $\frac{1}{4\alpha}\geq\frac{L}{2}+\frac{16L}{27}$. Thus, we have the first conclusion and
\begin{eqnarray}
\frac{8\theta^2}{59mL}\sum_{r=0}^{K-1}\frac{\Phi^r}{(1-\theta)^{r+1}}\leq F(\overline x^0)-F(x^*)+\left(\frac{\theta^2}{2\alpha}+\frac{\mu\theta}{2}\right)\|\overline z^0-x^*\|^2+\frac{4(1-\sigma)}{59mL(1-\theta)}\|\Pi\s^0\|^2.\notag
\end{eqnarray}
It follows from (\ref{cont29}) that
\begin{eqnarray}
\begin{aligned}\notag
&\sum_{k=0}^K\frac{L}{2m(1-\theta)^{k+1}}\|\Pi\x^k\|^2\\
&~~\leq \frac{4(1-\sigma)}{27mL(1+\frac{1}{\tau})(1-\theta)}\|\Pi\s^0\|^2+\frac{8\theta^2}{27mL(1+\frac{1}{\tau})}\sum_{r=0}^{K-1}\frac{\Phi^r}{(1-\theta)^{r+1}}\\
&~~=\frac{4(1-\sigma)}{59mL(1-\theta)}\|\Pi\s^0\|^2+\frac{8\theta^2}{59mL}\sum_{r=0}^{K-1}\frac{\Phi^r}{(1-\theta)^{r+1}}\\
&~~\leq2\left(F(\overline x^0)-F(x^*)+\left(\frac{\theta^2}{2\alpha}+\frac{\mu\theta}{2}\right)\|\overline z^0-x^*\|^2+\frac{4(1-\sigma)}{59mL(1-\theta)}\|\Pi\s^0\|^2\right).
\end{aligned}
\end{eqnarray}
Thus, we have the second conclusion.
\end{proof}

We end this section by summarizing the differences from \citep{qu2017-2} and the reasons of the convergence rates improvement.
\begin{remark}\label{remark3}
As shown in Lemmas \ref{lemma4} and \ref{lemma2}, we keep the Bregman divergence term $D_f(\overline x^k,\y^k)$, and use a constant $\tau$ to balance this divergence term and $\|\overline z^{k+1}-\overline z^k\|^2$. As shown in the proofs of Theorems \ref{theorem1} and \ref{theorem2}, we make the constant before $D_f(\overline x^k,\y^k)$ positive by setting $\tau$ small. As a comparison, \citet{qu2017-2} did not use this Bregman divergence term, and they bounded the term $\|\overline x^k-\overline y^k\|^2$, which is generated by the consensus errors and is an analogy to our term $D_f(\overline x^k,\y^k)$ generated in (\ref{cont15}), by setting much smaller step sizes than ours. See (32) and (53) in \citep{qu2017-2} for the details. To make the constant $A_4$ in their (32) positive, \citet{qu2017-2} set the step size of the order $\alpha=\bO(\frac{1}{L}(\frac{\mu}{L})^{3/7})$. Since $\sqrt{\mu\alpha}$ dominates the convergence rate for strongly convex problems, \citet{qu2017-2} only got the slower convergence rate of $\bO((1-(\frac{\mu}{L})^{5/7})^k)$. For nonstrongly convex problems, \citet{qu2017-2} set the step size of the order $\bO(\frac{1}{k^{0.6+\epsilon}})$ to bound the corresponding term in their (53), which gives the slower convergence rate of $\bO(\frac{1}{k^{1.4-\epsilon}})$.

As shown in Lemma \ref{lemma2}, to bound the consensus errors, we construct a linear combination of the consensus errors such that it shrinks geometrically with an additional error term. As a comparison, \citet{qu2017-2} used the linear system inequality, which needs to upper bound the spectral radius of a system matrix and thus it is quite involved. See the proofs of Lemmas 7 and 13-15 in \citep{qu2017-2}. Our proof is much simpler than those in \citep{qu2017-2}, and it can be extended to the time-varying graphs in a unified framework.
\end{remark}

\subsection{Bounding the Consensus Errors over Time-varying Graphs}\label{sec:proof3}
In this section, we consider algorithm (\ref{alg-tv-s1})-(\ref{alg-tv-s4}) over time-varying graphs. Our analysis follows the same proof framework in the previous section for static graphs, but with more involved details. In the next lemma, we first give the analogy counterparts of (\ref{cont10})-(\ref{cont12}).

\begin{lemma}\label{lemma1}
Suppose that Assumptions \ref{assumption_f} and \ref{assumption_w_tv} hold with $\mu\geq 0$. Let the sequence $\{\theta_k\}_{k=0}^K$ satisfy $\frac{\theta_k}{1.62}\leq \theta_{k+1}\leq \theta_k\leq 1$. Then, we have for any $k\geq\gamma-1$,
\begin{eqnarray}
&&\hspace*{-0.5cm}\|\Pi\s^{k+1}\|^2\leq \frac{1+\sigma_{\gamma}}{2}\|\Pi\s^{k-\gamma+1}\|^2+\frac{2\gamma}{1-\sigma_{\gamma}}\sum_{r=k-\gamma+1}^k\left(\theta_r^2\Phi^r+c_0\theta_r^2\|\Pi\z^r\|^2+c_0\|\Pi\x^r\|^2\right),\label{cont23}\\
&&\hspace*{-0.5cm}\|\Pi\x^{k+1}\|^2\leq \frac{1+\sigma_{\gamma}}{2}\|\Pi\x^{k-\gamma+1}\|^2+\frac{5.5\gamma}{1-\sigma_{\gamma}}\sum_{r=k-\gamma+2}^{k+1}\theta_r^2\|\Pi\z^r\|^2,\label{cont20}\\
&&\hspace*{-0.5cm}\theta_{k+1}^2\|\Pi\z^{k+1}\|^2\leq\frac{1+\sigma_{\gamma}}{2}\theta_{k-\gamma+1}^2\|\Pi\z^{k-\gamma+1}\|^2+\frac{4\gamma}{1-\sigma_{\gamma}}\sum_{r=k-\gamma+1}^k\left(\mu^2\alpha^2\|\Pi\x^r\|^2+\alpha^2\|\Pi\s^r\|^2\right),\label{cont21}
\end{eqnarray}
where we denote $c_0=4L^2\left(1+\frac{1}{\tau}\right)$, and $\tau$ and $\Phi^r$ are defined in Lemma \ref{lemma2}.
\end{lemma}
\begin{proof}
From (\ref{alg-tv-s2}) and the definition of $W^{k,\gamma}$ in (\ref{def_W_tau}), we have for any $k\geq\gamma-1$,
\begin{eqnarray}
\begin{aligned}\notag
\s^{k+1}=&W^{k+1}\s^k+\nabla f(\y^{k+1})-\nabla f(\y^k)\\
=&\left(\prod_{t=k-\gamma+2}^{k+1}W^t\right)\s^{k-\gamma+1}+\sum_{r=k-\gamma+1}^k \left(\prod_{t=r+1}^kW^{t+1}\right)(\nabla f(\y^{r+1})-\nabla f(\y^r))\\
=&W^{k+1,\gamma}\s^{k-\gamma+1}+\sum_{r=k-\gamma+1}^k W^{k+1,k-r}(\nabla f(\y^{r+1})-\nabla f(\y^r)),
\end{aligned}
\end{eqnarray}
where we denote $\prod_{t=k+1}^{k}W^{t+1}=\I$. Multiplying both sides by $\Pi$, using (\ref{tau_consensus-contraction}) and (\ref{tau_consensus-contraction2}), it gives
\begin{eqnarray}
\|\Pi\s^{k+1}\|\leq \sigma_{\gamma}\|\Pi\s^{k-\gamma+1}\|+\sum_{r=k-\gamma+1}^k \|\nabla f(\y^{r+1})-\nabla f(\y^r)\|.\label{cont13}
\end{eqnarray}
Similar to (\ref{cont10}), squaring both sides of (\ref{cont13}) yields
\begin{eqnarray}
\qquad\begin{aligned}\label{cont19}
\|\Pi\s^{k+1}\|^2\leq& \frac{1+\sigma_{\gamma}}{2}\|\Pi\s^{k-\gamma+1}\|^2+\frac{2}{1-\sigma_{\gamma}}\left(\sum_{r=k-\gamma+1}^k\|\nabla f(\y^{r+1})-\nabla f(\y^r)\|\right)^2\\
\leq& \frac{1+\sigma_{\gamma}}{2}\|\Pi\s^{k-\gamma+1}\|^2+\frac{2\gamma}{1-\sigma_{\gamma}}\sum_{r=k-\gamma+1}^k\|\nabla f(\y^{r+1})-\nabla f(\y^r)\|^2.
\end{aligned}
\end{eqnarray}
From (\ref{cont15}), we have (\ref{cont23}).  It follows from (\ref{alg-tv-s4}) that
\begin{eqnarray}
\begin{aligned}\notag
\x^{k+1}=&(1-\theta_k)W^k\x^k+\theta_k\z^{k+1}\\
=&\left(\prod_{t=k-\gamma+1}^k(1-\theta_t)W^t\right)\x^{k-\gamma+1}+\sum_{r=k-\gamma+1}^k\left( \prod_{t=r+1}^k(1-\theta_t)W^t\right)\theta_r\z^{r+1}\\
=&W^{k,\gamma}\x^{k-\gamma+1}\prod_{t=k-\gamma+1}^k(1-\theta_t)+\sum_{r=k-\gamma+1}^kW^{k,k-r}\theta_r\z^{r+1}\prod_{t=r+1}^k(1-\theta_t).
\end{aligned}
\end{eqnarray}
Similar to (\ref{cont13}) and (\ref{cont19}), we also have
\begin{eqnarray}
\|\Pi\x^{k+1}\|\leq\sigma_{\gamma}\|\Pi\x^{k-\gamma+1}\|+\sum_{r=k-\gamma+1}^k\theta_r\|\Pi\z^{r+1}\|=\sigma_{\gamma}\|\Pi\x^{k-\gamma+1}\|+\sum_{r=k-\gamma+2}^{k+1}\theta_{r-1}\|\Pi\z^r\|,\notag
\end{eqnarray}
and
\begin{eqnarray}
\begin{aligned}\notag
\|\Pi\x^{k+1}\|^2\leq& \frac{1+\sigma_{\gamma}}{2}\|\Pi\x^{k-\gamma+1}\|^2+\frac{2\gamma}{1-\sigma_{\gamma}}\sum_{r=k-\gamma+2}^{k+1}\theta_{r-1}^2\|\Pi\z^r\|^2.
\end{aligned}
\end{eqnarray}
Using $\theta_{r-1}\leq 1.62\theta_r$, we obtain (\ref{cont20}). Similarly, for (\ref{alg-tv-s3}), we have
\begin{eqnarray}
\begin{aligned}\notag
\z^{k+1}=&\frac{\theta_k}{\theta_k+\mu\alpha}\W^k\z^k+\frac{\mu\alpha}{\theta_k+\mu\alpha}W^k\y^k-\frac{\alpha}{\theta_k+\mu\alpha}\s^k\\
\overset{a}=&\frac{\theta_k(1+\mu\alpha)}{\theta_k+\mu\alpha}\W^k\z^k+\frac{\mu\alpha(1-\theta_k)}{\theta_k+\mu\alpha}W^k\x^k-\frac{\alpha}{\theta_k+\mu\alpha}\s^k\\
=&\left(\prod_{t=k-\gamma+1}^k\frac{\theta_t(1+\mu\alpha)}{\theta_t+\mu\alpha}W^t\right)\z^{k-\gamma+1}\\
&+\sum_{r=k-\gamma+1}^k \left(\prod_{t=r+1}^k\frac{\theta_t(1+\mu\alpha)}{\theta_t+\mu\alpha}W^t\right)\left(\frac{\mu\alpha(1-\theta_r)}{\theta_r+\mu\alpha}W^r\x^r-\frac{\alpha}{\theta_r+\mu\alpha}\s^r\right)\\
=&W^{k,\gamma}\z^{k-\gamma+1}\prod_{t=k-\gamma+1}^k\frac{\theta_t(1+\mu\alpha)}{\theta_t+\mu\alpha}\\
&+\sum_{r=k-\gamma+1}^kW^{k,k-r}\left(\frac{\mu\alpha(1-\theta_r)}{\theta_r+\mu\alpha}W^r\x^r-\frac{\alpha}{\theta_r+\mu\alpha}\s^r\right)\prod_{t=r+1}^k\frac{\theta_t(1+\mu\alpha)}{\theta_t+\mu\alpha},
\end{aligned}
\end{eqnarray}
and
\begin{eqnarray}
\begin{aligned}\notag
\|\Pi\z^{k+1}\|\overset{b}\leq& \sigma_{\gamma}\|\Pi\z^{k-\gamma+1}\|+\sum_{r=k-\gamma+1}^k\left(\frac{\mu\alpha(1-\theta_r)}{\theta_r+\mu\alpha}\|\Pi\x^r\|+\frac{\alpha}{\theta_r+\mu\alpha}\|\Pi\s^r\|\right)\\
\leq&\sigma_{\gamma}\|\Pi\z^{k-\gamma+1}\|+\sum_{r=k-\gamma+1}^k\left(\frac{\mu\alpha}{\theta_r}\|\Pi\x^r\|+\frac{\alpha}{\theta_r}\|\Pi\s^r\|\right),
\end{aligned}
\end{eqnarray}
where we use (\ref{alg-tv-s1}) in $\overset{a}=$, $\frac{\theta_t(1+\mu\alpha)}{\theta_t+\mu\alpha}\leq1$ with $\theta_t\leq 1$ in $\overset{b}\leq$. Similar to (\ref{cont19}), squaring both sides yields
\begin{eqnarray}
\|\Pi\z^{k+1}\|^2\leq \frac{1+\sigma_{\gamma}}{2}\|\Pi\z^{k-\gamma+1}\|^2+\frac{4\gamma}{1-\sigma_{\gamma}}\sum_{r=k-\gamma+1}^k\left(\frac{\mu^2\alpha^2}{\theta_r^2}\|\Pi\x^r\|^2+\frac{\alpha^2}{\theta_r^2}\|\Pi\s^r\|^2\right).\notag
\end{eqnarray}
Multiplying both sides by $\theta_{k+1}^2$ and using the non-increasing of $\{\theta_k\}$, it further gives (\ref{cont21})
\end{proof}

Motivated by the proof of Lemma \ref{lemma2}, we want to construct a linear combination of the consensus errors. However, due to the time-varying graphs and the $\gamma$-step joint spectrum property in Assumption \ref{assumption_w_tv}, we see from (\ref{cont23})-(\ref{cont21}) that they shrink every $\gamma$ iterations, rather than every iteration. By exploiting the special structures in (\ref{cont23})-(\ref{cont21}), we define the following quantities:
\begin{eqnarray}
\begin{aligned}\notag
&\M_{\s}^{k+\gamma,\gamma}=\max_{r=k+1,...,k+\gamma}\|\Pi\s^r\|^2,\quad
\M_{\x}^{k+\gamma,\gamma}=\max_{r=k+1,...,k+\gamma}\|\Pi\x^r\|^2,\\
&\M_{\y}^{k+\gamma,\gamma}=\max_{r=k+1,...,k+\gamma}\|\Pi\y^r\|^2,\quad
\M_{\z}^{k+\gamma,\gamma}=\max_{r=k+1,...,k+\gamma}\theta_r^2\|\Pi\z^r\|^2.
\end{aligned}
\end{eqnarray}
Motivated by (\ref{cont9}), we define the following quantity in the form of summation, instead of the maximum, and we sum up to $k+\gamma-1$, rather than $k+\gamma$,
\begin{equation}
\cS_{\phi}^{k+\gamma-1,\gamma}=\sum_{r=k}^{k+\gamma-1}\theta_r^2\Phi^r.\notag
\end{equation}
The next lemma is an analogy counterpart of Lemma \ref{lemma2}. Unlike the classical analysis relying on the small gain theorem \citep{shi2017}, which is unclear how to be used to the accelerated methods, and especially for nonstrongly convex problems, our main idea is to construct a linear combination of $\M_{\s}^{k+\gamma,\gamma}$, $\M_{\x}^{k+\gamma,\gamma}$, and $\M_{\z}^{k+\gamma,\gamma}$ with carefully designed weights such that it shrinks geometrically with the additional error term $\cS_{\phi}^{k+\gamma-1,\gamma}$, which is crucial to extend our analysis over static graphs to time-varying graphs in a unified framework, both for nonstrongly convex and strongly convex problems. Moreover, our proof technique to bound the consensus errors can be embedded into many algorithm frameworks, because it is separated from the analysis of the inexact accelerated gradient descent in Lemma \ref{lemma4}.
\begin{lemma}\label{lemma3}
Under the settings of Lemma \ref{lemma1}, letting $\alpha\leq\frac{(1-\sigma_{\gamma})^3}{3385L\gamma^3\sqrt{1+\frac{1}{\tau}}}$, we have for any $t\geq 0$,
\begin{eqnarray}
\begin{aligned}\label{cont24}
\max\left\{\M_{\y}^{(t+1)\gamma,\gamma},\M_{\x}^{(t+1)\gamma,\gamma}\right\}\leq C_3\rho^{t\gamma}+C_4\sum_{s=0}^{(t+1)\gamma-1}\rho^{(t-1)\gamma-s}\theta_s^2\Phi^s.
\end{aligned}
\end{eqnarray}
where $\rho=\sqrt[\gamma]{1-\frac{1-\sigma_{\gamma}}{5}}$, $C_3=\left(\frac{(1-\sigma_{\gamma})^2}{162L^2\gamma^2(1+\frac{1}{\tau})}\M_{\s}^{\gamma,\gamma}+\frac{442\gamma^2}{(1-\sigma_{\gamma})^2}\M_{\z}^{\gamma,\gamma}+2\M_{\x}^{\gamma,\gamma}\right)$, and $C_4=\frac{1-\sigma_{\gamma}}{38L^2\gamma(1+\frac{1}{\tau})}$.
\end{lemma}

\begin{proof}
For any $t$ satisfying $k\leq t\leq k+\gamma-1$ with $k\geq \gamma$, we can relax (\ref{cont23}) to
\begin{eqnarray}
\begin{aligned}\notag
\|\Pi\s^{t+1}\|^2\leq& \frac{1+\sigma_{\gamma}}{2}\|\Pi\s^{t-\gamma+1}\|^2+\frac{2\gamma}{1-\sigma_{\gamma}}\sum_{r=t-\gamma+1}^{t}\left(\theta_r^2\Phi^r+c_0\theta_r^2\|\Pi\z^r\|^2+c_0\|\Pi\x^r\|^2\right)\\
\leq& \frac{1\hspace*{-0.035cm}+\hspace*{-0.035cm}\sigma_{\gamma}}{2}\|\Pi\s^{t-\gamma+1}\|^2\hspace*{-0.035cm}+\hspace*{-0.035cm}\frac{2\gamma}{1\hspace*{-0.035cm}-\hspace*{-0.035cm}\sigma_{\gamma}}\hspace*{-0.035cm}\sum_{r=k-\gamma}^{k+\gamma-1}\hspace*{-0.035cm}\theta_r^2\Phi^r\hspace*{-0.035cm}+\hspace*{-0.035cm}\frac{2\gamma}{1\hspace*{-0.035cm}-\hspace*{-0.035cm}\sigma_{\gamma}}\hspace*{-0.035cm}\sum_{r=k-\gamma+1}^{k+\gamma}\hspace*{-0.06cm}\left(c_0\theta_r^2\|\Pi\z^r\|^2\hspace*{-0.035cm}+\hspace*{-0.035cm}c_0\|\Pi\x^r\|^2\right)\\
\leq& \frac{1+\sigma_{\gamma}}{2}\|\Pi\s^{t-\gamma+1}\|^2+\frac{2\gamma}{1-\sigma_{\gamma}}\left(\cS_{\phi}^{k-1,\gamma}+\cS_{\phi}^{k+\gamma-1,\gamma}\right)\\
&+\frac{2\gamma}{1-\sigma_{\gamma}}\sum_{r=k-\gamma+1}^{k+\gamma}\left(c_0\M_{\z}^{k,\gamma}+c_0\M_{\z}^{k+\gamma,\gamma}+c_0\M_{\x}^{k,\gamma}+c_0\M_{\x}^{k+\gamma,\gamma}\right)\\
=& \frac{1+\sigma_{\gamma}}{2}\|\Pi\s^{t-\gamma+1}\|^2+\frac{2\gamma}{1-\sigma_{\gamma}}\left(\cS_{\phi}^{k-1,\gamma}+\cS_{\phi}^{k+\gamma-1,\gamma}\right)\\
&+\frac{4\gamma^2}{1-\sigma_{\gamma}}\left(c_0\M_{\z}^{k,\gamma}+c_0\M_{\z}^{k+\gamma,\gamma}+c_0\M_{\x}^{k,\gamma}+c_0\M_{\x}^{k+\gamma,\gamma}\right).
\end{aligned}
\end{eqnarray}
Taking the maximum over $t=k,k+1,...,k+\gamma-1$ on both sides, we have
\begin{eqnarray}
\begin{aligned}\notag
\M_{\s}^{k+\gamma,\gamma}\leq& \frac{1+\sigma_{\gamma}}{2}\M_{\s}^{k,\gamma}+\frac{2\gamma}{1-\sigma_{\gamma}}\left(\cS_{\phi}^{k-1,\gamma}+\cS_{\phi}^{k+\gamma-1,\gamma}\right)\\
&+\frac{4\gamma^2}{1-\sigma_{\gamma}}\left(c_0\M_{\z}^{k,\gamma}+c_0\M_{\z}^{k+\gamma,\gamma}+c_0\M_{\x}^{k,\gamma}+c_0\M_{\x}^{k+\gamma,\gamma}\right).
\end{aligned}
\end{eqnarray}
Similarly, for (\ref{cont21}) and (\ref{cont20}), we also have
\begin{eqnarray}
\begin{aligned}\notag
\M_{\z}^{k+\gamma,\gamma}\leq& \frac{1+\sigma_{\gamma}}{2}\M_{\z}^{k,\gamma}+\frac{8\gamma^2}{1-\sigma_{\gamma}}\left(\mu^2\alpha^2\M_{\x}^{k,\gamma}+\mu^2\alpha^2\M_{\x}^{k+\gamma,\gamma}+\alpha^2\M_{\s}^{k,\gamma}+\alpha^2\M_{\s}^{k+\gamma,\gamma}\right),\\
\M_{\x}^{k+\gamma,\gamma}\leq& \frac{1+\sigma_{\gamma}}{2}\M_{\x}^{k,\gamma}+\frac{11\gamma^2}{1-\sigma_{\gamma}}\left(\M_{\z}^{k,\gamma}+\M_{\z}^{k+\gamma,\gamma}\right),
\end{aligned}
\end{eqnarray}
where for the second one, the relaxation of $\sum_{r=t-\gamma+2}^{t+1}\theta_r^2\|\Pi\z^r\|^2\leq \sum_{r=k-\gamma+1}^{k+\gamma}\theta_r^2\|\Pi\z^r\|^2$ also holds for any $t$ satisfying $k\leq t\leq k+\gamma-1$.

Adding the above three inequalities together with weights $c_1$, $c_2$, and $c_3$, respectively, we have
\begin{eqnarray}
\begin{aligned}\notag
&c_1\M_{\s}^{k+\gamma,\gamma}+c_2\M_{\z}^{k+\gamma,\gamma}+c_3\M_{\x}^{k+\gamma,\gamma}\\
&~~\leq \frac{8c_2\gamma^2\alpha^2}{1-\sigma_{\gamma}}\M_{\s}^{k+\gamma,\gamma}+\left( \frac{4c_1c_0\gamma^2}{1-\sigma_{\gamma}}+\frac{11c_3\gamma^2}{1-\sigma_{\gamma}} \right)\M_{\z}^{k+\gamma,\gamma}+\left( \frac{4c_1c_0\gamma^2}{1-\sigma_{\gamma}}+\frac{8c_2\gamma^2\mu^2\alpha^2}{1-\sigma_{\gamma}} \right)\M_{\x}^{k+\gamma,\gamma}\\
&~~\quad+\left(\frac{c_1(1+\sigma_{\gamma})}{2}+\frac{8c_2\gamma^2\alpha^2}{1-\sigma_{\gamma}}  \right)\M_{\s}^{k,\gamma}+\left(\frac{c_2(1+\sigma_{\gamma})}{2}+\frac{4c_1c_0\gamma^2}{1-\sigma_{\gamma}}+\frac{11c_3\gamma^2}{1-\sigma_{\gamma}}  \right)\M_{\z}^{k,\gamma}\\
&~~\quad +\left(\frac{c_3(1+\sigma_{\gamma})}{2}+\frac{4c_1c_0\gamma^2}{1-\sigma_{\gamma}}+\frac{8c_2\gamma^2\mu^2\alpha^2}{1-\sigma_{\gamma}}  \right)\M_{\x}^{k,\gamma}+\frac{2c_1\gamma}{1-\sigma_{\gamma}}\left(\cS_{\phi}^{k-1,\gamma}+\cS_{\phi}^{k+\gamma-1,\gamma}\right).
\end{aligned}
\end{eqnarray}
We want to choose $c_1$, $c_2$, $c_3$, and $\alpha$ such that the following inequalities hold,
\begin{eqnarray}
\begin{aligned}\notag
&\frac{8c_2\gamma^2\alpha^2}{1-\sigma_{\gamma}}\leq\frac{c_1(1-\sigma_{\gamma})}{20},\hspace*{2.7cm}\frac{c_1(1+\sigma_{\gamma})}{2}+\frac{8c_2\gamma^2\alpha^2}{1-\sigma_{\gamma}}\leq\frac{c_1(3+\sigma_{\gamma})}{4},\\
\\
& \frac{4c_1c_0\gamma^2}{1-\sigma_{\gamma}}+\frac{11c_3\gamma^2}{1-\sigma_{\gamma}}\leq\frac{c_2(1-\sigma_{\gamma})}{20},\hspace*{1.25cm}\frac{c_2(1+\sigma_{\gamma})}{2}+\frac{4c_1c_0\gamma^2}{1-\sigma_{\gamma}}+\frac{11c_3\gamma^2}{1-\sigma_{\gamma}}\leq \frac{c_2(3+\sigma_{\gamma})}{4},\\
\\
&\frac{4c_1c_0\gamma^2}{1-\sigma_{\gamma}}+\frac{8c_2\gamma^2\mu^2\alpha^2}{1-\sigma_{\gamma}}\leq\frac{c_3(1-\sigma_{\gamma})}{20},\qquad \frac{c_3(1+\sigma_{\gamma})}{2}+\frac{4c_1c_0\gamma^2}{1-\sigma_{\gamma}}+\frac{8c_2\gamma^2\mu^2\alpha^2}{1-\sigma_{\gamma}}\leq\frac{c_3(3+\sigma_{\gamma})}{4},
\end{aligned}
\end{eqnarray}
which are satisfied if the three inequalities in the left column hold. Accordingly, we can choose $c_3=\frac{81c_1c_0\gamma^2}{(1-\sigma_{\gamma})^2}$, $c_2=\frac{17900c_1c_0\gamma^4}{(1-\sigma_{\gamma})^4}\geq\frac{80c_1c_0\gamma^2}{(1-\sigma_{\gamma})^2}+\frac{220c_3\gamma^2}{(1-\sigma_{\gamma})^2}$, and $\alpha^2\leq\min\left\{\frac{(1-\sigma_{\gamma})^6}{2864000c_0\gamma^6},\frac{(1-\sigma_{\gamma})^4}{2864000\mu^2\gamma^4}\right\}$. Thus,  we have for any $k\geq\gamma$,
\begin{eqnarray}
\begin{aligned}\notag
&\frac{19+\sigma_{\gamma}}{20}\left(c_1\M_{\s}^{k+\gamma,\gamma}+c_2\M_{\z}^{k+\gamma,\gamma}+c_3\M_{\x}^{k+\gamma,\gamma}\right)\\
&~~\leq \frac{3+\sigma_{\gamma}}{4}\left(c_1\M_{\s}^{k,\gamma}+c_2\M_{\z}^{k,\gamma}+c_3\M_{\x}^{k,\gamma}\right)+\frac{2c_1\gamma}{1-\sigma_{\gamma}}\left(\cS_{\phi}^{k-1,\gamma}+\cS_{\phi}^{k+\gamma-1,\gamma}\right)\\
&~~\leq \frac{19+\sigma_{\gamma}}{20}\left(1-\frac{1-\sigma_{\gamma}}{5}\right)\left(c_1\M_{\s}^{k,\gamma}+c_2\M_{\z}^{k,\gamma}+c_3\M_{\x}^{k,\gamma}\right)+\frac{2c_1\gamma}{1-\sigma_{\gamma}}\left(\cS_{\phi}^{k-1,\gamma}+\cS_{\phi}^{k+\gamma-1,\gamma}\right),
\end{aligned}
\end{eqnarray}
and
\begin{eqnarray}
\begin{aligned}\label{cont22}
&c_1\M_{\s}^{(t+1)\gamma,\gamma}+c_2\M_{\z}^{(t+1)\gamma,\gamma}+c_3\M_{\x}^{(t+1)\gamma,\gamma}\\
&~~\leq\left(1-\frac{1-\sigma_{\gamma}}{5}\right)^t\left(c_1\M_{\s}^{\gamma,\gamma}+c_2\M_{\z}^{\gamma,\gamma}+c_3\M_{\x}^{\gamma,\gamma}\right)\\
&~~\quad+\frac{40c_1\gamma}{19(1-\sigma_{\gamma})}\sum_{r=1}^t\left(1-\frac{1-\sigma_{\gamma}}{5}\right)^{t-r}\left(\cS_{\phi}^{r\gamma-1,\gamma}+\cS_{\phi}^{(r+1)\gamma-1,\gamma}\right).
\end{aligned}
\end{eqnarray}
It follows from (\ref{cont5}) and $c_2>c_3$ that
\begin{eqnarray}
\frac{c_3}{2}\max\left\{\M_{\y}^{(t+1)\gamma,\gamma},\M_{\x}^{(t+1)\gamma,\gamma}\right\}\leq c_1\M_{\s}^{(t+1)\gamma,\gamma}+c_2\M_{\z}^{(t+1)\gamma,\gamma}+c_3\M_{\x}^{(t+1)\gamma,\gamma}.\notag
\end{eqnarray}
On the other hand, denoting $\rho=\sqrt[\gamma]{1-\frac{1-\sigma_{\gamma}}{5}}$, we have
\begin{eqnarray}
\begin{aligned}\notag
&\sum_{r=1}^t\rho^{\gamma(t-r)}\left(\cS_{\phi}^{r\gamma-1,\gamma}+\cS_{\phi}^{(r+1)\gamma-1,\gamma}\right)\\
&~~=\rho^{\gamma t}\sum_{r=1}^t\left(\frac{1}{\rho^{\gamma}}\right)^r\left(\sum_{s=(r-1)\gamma}^{r\gamma-1}\theta_s^2\Phi^s+\sum_{s=r\gamma}^{(r+1)\gamma-1}\theta_s^2\Phi^s\right)\\
&~~=\rho^{\gamma t}\sum_{s=0}^{t\gamma-1}\left(\frac{1}{\rho^{\gamma}}\right)^{\lfloor\frac{s}{\gamma}\rfloor+1}\theta_s^2\Phi^s+\rho^{\gamma t}\sum_{s=\gamma}^{(t+1)\gamma-1}\left(\frac{1}{\rho^{\gamma}}\right)^{\lfloor\frac{s}{\gamma}\rfloor}\theta_s^2\Phi^s\\
&~~\leq2\rho^{\gamma t}\sum_{s=0}^{(t+1)\gamma-1}\left(\frac{1}{\rho^{\gamma}}\right)^{\frac{s}{\gamma}+1}\theta_s^2\Phi^s=2\sum_{s=0}^{(t+1)\gamma-1}\rho^{(t-1)\gamma-s}\theta_s^2\Phi^s.
\end{aligned}
\end{eqnarray}
Plugging the above two inequalities and the settings of $c_3$ and $c_0$ into (\ref{cont22}), we have the conclusion.
\end{proof}

\begin{remark}
We briefly demonstrate the advantage of introducing the quantities of $\M_{\s}^{k+\gamma,\gamma}$, $\M_{\x}^{k+\gamma,\gamma}$, $\M_{\y}^{k+\gamma,\gamma}$, and $\M_{\z}^{k+\gamma,\gamma}$. As discussed in Remark \ref{remark3}, researchers in the control community often use linear system inequality to prove the convergence, which is quite challenging to use over time-varying graphs. For example, \citet{tvab20} constructed a $\gamma$th order linear system inequality in the form of
    \begin{eqnarray}
    \left(
  \begin{array}{c}
    \alpha^{k+\gamma}\\
    \alpha^{k+\gamma-1}\\
    \alpha^{k+\gamma-2}\\
    \vdots\\
    \alpha^{k+1}
  \end{array}
\right)\leq \left(
  \begin{array}{lllll}
    M_1 & M_2 & \cdots & M_{\gamma-1} & M_{\gamma}\\
    \I  &     &        &              &           \\
        & \I  &        &              &           \\
        &     &  \ddots&              &           \\
        &     &        &    \I        &           \\
  \end{array}
\right)\left(
  \begin{array}{c}
    \alpha^{k+\gamma-1}\\
    \alpha^{k+\gamma-2}\\
    \alpha^{k+\gamma-3}\\
    \vdots\\
    \alpha^{k}
  \end{array}
\right)\label{cont44}
    \end{eqnarray}
    for the $\mathcal{A}\mathcal{B}$/push-pull method, which is an extension of gradient tracking to time-varying directed graphs. They only proved that the spectral radius of the system matrix is strictly less than 1 without any explicit upper bound. Thus, no explicit convergence rate was given in \citep{tvab20}.

    On the other hand, the system (\ref{cont44}) can be simplified by defining similar quantities of $\M_{\s}^{k+\gamma,\gamma}$, $\M_{\x}^{k+\gamma,\gamma}$, $\M_{\y}^{k+\gamma,\gamma}$, and $\M_{\z}^{k+\gamma,\gamma}$. Moreover, the proof can be further simplified by avoiding  analyzing the spectral radius if our technical trick of constructing the linear combination is used.
\end{remark}

Following the same proof framework over static graphs, our next step is to bound the weighted cumulative consensus errors.  However, the details are much more complex. The proof of Lemma \ref{lemma6} provides some insights.

\begin{lemma}\label{lemma9}
Suppose that Assumptions \ref{assumption_f} and \ref{assumption_w_tv} hold with $\mu=0$. Let the sequence $\{\theta_k\}_{k=0}^{T\gamma}$ satisfy $\frac{1-\theta_k}{\theta_k^2}=\frac{1}{\theta_{k-1}^2}$ with $\theta_0=1$, let $\alpha\leq\frac{(1-\sigma_{\gamma})^3}{3385L\gamma^3\sqrt{1+\frac{1}{\tau}}}$. Then for algorithm (\ref{alg-tv-s1})-(\ref{alg-tv-s4}), we have
\begin{eqnarray}
\begin{aligned}\label{cont30}
&\max\left\{\sum_{k=0}^{T\gamma}\frac{L}{2m\theta_k^2}\|\Pi\y^k\|^2,\sum_{k=0}^{T\gamma}\frac{L}{2m\theta_k^2}\|\Pi\x^k\|^2\right\}\\
&~~\leq\frac{235\gamma^3C_3L}{m(1-\sigma_{\gamma})^3}+\frac{10\gamma^2}{mL(1+\frac{1}{\tau})(1-\sigma_{\gamma})^2}\sum_{s=0}^{T\gamma-1}\Phi^s,
\end{aligned}
\end{eqnarray}
where $\tau$ and $\Phi^r$ are defined in Lemma \ref{lemma2}, and $C_3$ is defined in Lemma \ref{lemma3}.
\end{lemma}
\begin{proof}
We first verify $\theta_k\leq 1.62\theta_{k+1}$ for all $k\geq 0$, which is required in Lemmas \ref{lemma1} and \ref{lemma3}. In fact, from $\frac{1-\theta_{k+1}}{\theta_{k+1}^2}=\frac{1}{\theta_k^2}$ and $\theta_0=1$, we have  $\frac{\theta_k}{\theta_{k+1}}=\frac{1}{\sqrt{1-\theta_{k+1}}}\in (1,\frac{1}{\sqrt{1-\theta_1}}]\in (1,1.62]$ for any $k\geq 0$. Next, we upper and lower bound $\rho$. From the definition of $\rho=\sqrt[\gamma]{1-\frac{1-\sigma_{\tau}}{5}}$ and the fact that $(1-\frac{x}{\gamma})^{\gamma}\geq 1-x$ for any $x\in (0,1)$ and $\gamma\geq1$, we know
\begin{equation}
\rho\leq 1-\frac{1-\sigma_{\gamma}}{5\gamma},\qquad \rho^{\gamma}\geq\frac{4}{5}.\label{cont26}
\end{equation}
The remaining proof is similar to that of Lemma \ref{lemma6}. From the definition of $\M_{\y}^{t\gamma+\gamma,\gamma}$ and (\ref{cont24}), we have
\begin{eqnarray}
\begin{aligned}\label{cont35}
&\sum_{k=1}^{T\gamma}\frac{L}{2m\theta_k^2}\|\Pi\y^k\|^2=\sum_{t=0}^{T-1}\sum_{r=1}^{\gamma}\frac{L}{2m\theta_{t\gamma+r}^2}\|\Pi\y^{t\gamma+r}\|^2\leq\sum_{t=0}^{T-1}\sum_{r=1}^{\gamma}\frac{L}{2m\theta_{t\gamma+r}^2}\M_{\y}^{(t+1)\gamma,\gamma}\\
&~~\leq \sum_{t=0}^{T-1}\sum_{r=1}^{\gamma}\frac{C_3L\rho^{t\gamma}}{2m\theta_{t\gamma+r}^2}+\sum_{t=0}^{T-1}\sum_{r=1}^{\gamma}\frac{C_4L}{2m\theta_{t\gamma+r}^2} \sum_{s=0}^{(t+1)\gamma-1}\rho^{(t-1)\gamma-s}\theta_s^2\Phi^s\\
&~~\leq \frac{C_3L}{2m}\sum_{t=0}^{T-1}\sum_{r=1}^{\gamma}\frac{\rho^{t\gamma+r}}{\rho^{\gamma}\theta_{t\gamma+r}^2}+\frac{C_4L}{2m}\sum_{t=0}^{T-1}\sum_{r=1}^{\gamma}\frac{\rho^{t\gamma+r}}{\rho^{2\gamma}\theta_{t\gamma+r}^2} \sum_{s=0}^{(t+1)\gamma-1}\frac{\theta_s^2}{\rho^s}\Phi^s\\
&~~\overset{a}= \frac{C_3L}{2m\rho^{\gamma}}\sum_{k=1}^{T\gamma}\frac{\rho^k}{\theta_k^2}+\frac{C_4L}{2m\rho^{2\gamma}}\sum_{k=1}^{T\gamma}\frac{\rho^{k}}{\theta_k^2} \sum_{s=0}^{\lceil\frac{k}{\gamma}\rceil\gamma-1}\frac{\theta_s^2}{\rho^s}\Phi^s,
\end{aligned}
\end{eqnarray}
where $(t+1)\gamma-1=\lceil\frac{k}{\gamma}\rceil\gamma-1$ in $\overset{a}=$ comes from the variable substitution $k=t\gamma+r$ with $r=1,2,...,\gamma$. Next, we compute the second part in $\overset{a}=$. It gives
\begin{eqnarray}
\begin{aligned}\label{cont36}
&\sum_{k=1}^{T\gamma}\frac{\rho^{k}}{\theta_k^2} \sum_{s=0}^{\lceil\frac{k}{\gamma}\rceil\gamma-1}\frac{\theta_s^2}{\rho^s}\Phi^s\\
&~~\leq\sum_{k=1}^{T\gamma-\gamma}\frac{\rho^{k}}{\theta_k^2} \sum_{s=0}^{k+\gamma-1}\frac{\theta_s^2}{\rho^s}\Phi^s+\sum_{k=T\gamma-\gamma+1}^{T\gamma}\frac{\rho^{k}}{\theta_k^2} \sum_{s=0}^{T\gamma-1}\frac{\theta_s^2}{\rho^s}\Phi^s\\
&~~=\left(\sum_{k=1}^{T\gamma-\gamma}\frac{\rho^{k}}{\theta_k^2} \sum_{s=0}^{\gamma-1}\frac{\theta_s^2}{\rho^s}\Phi^s+\sum_{k=1}^{T\gamma-\gamma}\frac{\rho^{k}}{\theta_k^2} \sum_{s=\gamma}^{k+\gamma-1}\frac{\theta_s^2}{\rho^s}\Phi^s\right)+\sum_{s=0}^{T\gamma-1}\frac{\theta_s^2}{\rho^s}\Phi^s\sum_{k=T\gamma-\gamma+1}^{T\gamma}\frac{\rho^{k}}{\theta_k^2}\\
&~~=\sum_{s=0}^{\gamma-1}\frac{\theta_s^2}{\rho^s}\Phi^s\sum_{k=1}^{T\gamma-\gamma}\frac{\rho^{k}}{\theta_k^2}+\sum_{s=\gamma}^{T\gamma-1}\frac{\theta_s^2}{\rho^s}\Phi^s\sum_{k=s-\gamma+1}^{T\gamma-\gamma}\frac{\rho^{k}}{\theta_k^2}\\
&~~\quad+\left(\sum_{s=0}^{\gamma-1}\frac{\theta_s^2}{\rho^s}\Phi^s\sum_{k=T\gamma-\gamma+1}^{T\gamma}\frac{\rho^{k}}{\theta_k^2}+\sum_{s=\gamma}^{T\gamma-1}\frac{\theta_s^2}{\rho^s}\Phi^s\sum_{k=T\gamma-\gamma+1}^{T\gamma}\frac{\rho^{k}}{\theta_k^2}\right)\\
&~~=\sum_{s=0}^{\gamma-1}\frac{\theta_s^2}{\rho^s}\Phi^s\sum_{k=1}^{T\gamma}\frac{\rho^{k}}{\theta_k^2}+\sum_{s=\gamma}^{T\gamma-1}\frac{\theta_s^2}{\rho^s}\Phi^s\sum_{k=s-\gamma+1}^{T\gamma}\frac{\rho^{k}}{\theta_k^2}.
\end{aligned}
\end{eqnarray}
Plugging (\ref{cont36}) into (\ref{cont35}), it follows from $\Pi\y^0=0$ that
\begin{eqnarray}
\hspace*{-0.7cm}\begin{aligned}\notag
&\sum_{k=0}^{T\gamma}\frac{L}{2m\theta_k^2}\|\Pi\y^k\|^2\\
&~~\leq\frac{C_3L}{2m\rho^{\gamma}}\sum_{k=1}^{T\gamma}\frac{\rho^k}{\theta_k^2}+\frac{C_4L}{2m\rho^{2\gamma}}\left(\sum_{s=0}^{\gamma-1}\frac{\theta_s^2}{\rho^s}\Phi^s\sum_{k=1}^{T\gamma}\frac{\rho^{k}}{\theta_k^2}+\sum_{s=\gamma}^{T\gamma-1}\frac{\theta_s^2}{\rho^s}\Phi^s\sum_{k=s-\gamma+1}^{T\gamma}\frac{\rho^{k}}{\theta_k^2}\right)\\
&~~\overset{b}\leq \frac{C_3L}{2m\rho^{\gamma}}\frac{3\rho}{(1-\rho)^3}+\frac{C_4L}{2m\rho^{2\gamma}}\left(\frac{3\rho}{(1-\rho)^3}\sum_{s=0}^{\gamma-1}\frac{\theta_s^2}{\rho^s}\Phi^s+\sum_{s=\gamma}^{T\gamma-1}\frac{\theta_s^2}{\rho^s}\Phi^s\frac{3\rho^{s-\gamma+1}}{(1-\rho)^3\theta_{s-\gamma}^2}\right)
\end{aligned}
\end{eqnarray}
\begin{eqnarray}
\hspace*{-1cm}\begin{aligned}
&~~\overset{c}\leq \frac{3C_3L}{2m\rho^{\gamma-1}(1-\rho)^3}+\frac{C_4L}{2m\rho^{2\gamma}}\left(\frac{3\rho}{(1-\rho)^3\rho^{\gamma-1}}\sum_{s=0}^{\gamma-1}\Phi^s+\frac{3\rho^{-\gamma+1}}{(1-\rho)^3}\sum_{s=\gamma}^{T\gamma-1}\Phi^s\right)\\
&~~\leq\frac{3C_3L}{2m\rho^{\gamma-1}(1-\rho)^3}+\frac{3C_4L}{2m\rho^{3\gamma-1}(1-\rho)^3}\sum_{s=0}^{T\gamma-1}\Phi^s\\
&~~\overset{d}\leq \frac{235\gamma^3C_3L}{m(1-\sigma_{\gamma})^3}+\frac{10\gamma^2}{mL(1+\frac{1}{\tau})(1-\sigma_{\gamma})^2}\sum_{s=0}^{T\gamma-1}\Phi^s,
\end{aligned}
\end{eqnarray}
where $\overset{b}\leq$ uses (\ref{cont25}), $\overset{c}\leq$ uses $\theta_s\leq 1$ for $s\leq\gamma-1$ and $\theta_s\leq\theta_{s-\gamma}$ for $s\geq\gamma$, $\overset{d}\leq$ uses (\ref{cont26}) and the definition of $C_4$ given in Lemma \ref{lemma3}. Replacing $\|\Pi\y^k\|$ by $\|\Pi\x^k\|$ in the above analysis, we have the same bound for $\|\Pi\x^k\|^2$.
\end{proof}

The next lemma is an analogy counterpart of Lemma \ref{lemma7}, and the proof is similar to that of the above Lemma \ref{lemma9}.
\begin{lemma}\label{lemma10}
Suppose that Assumptions \ref{assumption_f} and \ref{assumption_w_tv} hold with $\mu>0$. Let $\alpha\leq\frac{(1-\sigma_{\gamma})^3}{3385L\gamma^3\sqrt{1+\frac{1}{\tau}}}$ and $\theta_k\equiv\theta=\frac{\sqrt{\mu\alpha}}{2}$. Then for algorithm (\ref{alg-tv-s1})-(\ref{alg-tv-s4}), we have
\begin{eqnarray}
\begin{aligned}\label{cont31}
&\max\left\{\sum_{k=0}^{T\gamma}\frac{L}{2m(1-\theta)^{k+1}}\|\Pi\y^k\|^2,\sum_{k=0}^{T\gamma}\frac{L}{2m(1-\theta)^{k+1}}\|\Pi\x^k\|^2\right\}\\
&~~\leq \frac{3.3C_3L\gamma}{m(1-\theta)(1-\sigma_{\gamma})}+\frac{\theta^2}{7mL(1+\frac{1}{\tau})}\sum_{s=0}^{T\gamma-1}\frac{\Phi^s}{(1-\theta)^{s+1}},
\end{aligned}
\end{eqnarray}
where $\tau$ and $\Phi^r$ are defined in Lemma \ref{lemma2}, and $C_3$ is defined in Lemma \ref{lemma3}.
\end{lemma}
\begin{proof}
From the definition of $\M_{\y}^{t\gamma,\gamma}$ and (\ref{cont24}), we have
\begin{eqnarray}
\begin{aligned}\notag
&\sum_{k=1}^{T\gamma}\frac{L}{2m(1-\theta)^{k+1}}\|\Pi\y^k\|^2\\
&~~=\sum_{t=0}^{T-1}\sum_{r=1}^{\gamma}\frac{L}{2m(1-\theta)^{t\gamma+r+1}}\|\Pi\y^{t\gamma+r}\|^2\leq\sum_{t=0}^{T-1}\sum_{r=1}^{\gamma}\frac{L}{2m(1-\theta)^{t\gamma+r+1}}\M_{\y}^{(t+1)\gamma,\gamma}\\
&~~\leq \sum_{t=0}^{T-1}\sum_{r=1}^{\gamma}\frac{C_3L\rho^{t\gamma}}{2m(1-\theta)^{t\gamma+r+1}}+\sum_{t=0}^{T-1}\sum_{r=1}^{\gamma}\frac{C_4L}{2m(1-\theta)^{t\gamma+r+1}} \sum_{s=0}^{(t+1)\gamma-1}\rho^{(t-1)\gamma-s}\theta^2\Phi^s\\
&~~\leq \frac{C_3L}{2m(1-\theta)}\sum_{t=0}^{T-1}\sum_{r=1}^{\gamma}\frac{\rho^{t\gamma+r}}{\rho^{\gamma}(1-\theta)^{t\gamma+r}}+\frac{C_4L\theta^2}{2m(1-\theta)}\sum_{t=0}^{T-1}\sum_{r=1}^{\gamma}\frac{\rho^{t\gamma+r}}{\rho^{2\gamma}(1-\theta)^{t\gamma+r}} \sum_{s=0}^{(t+1)\gamma-1}\frac{\Phi^s}{\rho^s}\\
&~~= \frac{C_3L}{2m\rho^{\gamma}(1-\theta)}\sum_{k=1}^{T\gamma}\left(\frac{\rho}{1-\theta}\right)^k+\frac{C_4L\theta^2}{2m\rho^{2\gamma}(1-\theta)}\sum_{k=1}^{T\gamma}\left(\frac{\rho}{1-\theta}\right)^k \sum_{s=0}^{\lceil\frac{k}{\gamma}\rceil\gamma-1}\frac{\Phi^s}{\rho^s}.
\end{aligned}
\end{eqnarray}
Similar to (\ref{cont36}), we have
\begin{eqnarray}
\begin{aligned}\notag
\sum_{k=1}^{T\gamma}\left(\frac{\rho}{1-\theta}\right)^k \sum_{s=0}^{\lceil\frac{k}{\gamma}\rceil\gamma-1}\frac{\Phi^s}{\rho^s}\leq \sum_{s=0}^{\gamma-1}\frac{\Phi^s}{\rho^s}\sum_{k=1}^{T\gamma}\left(\frac{\rho}{1-\theta}\right)^k+\sum_{s=\gamma}^{T\gamma-1}\frac{\Phi^s}{\rho^s}\sum_{k=s-\gamma+1}^{T\gamma}\left(\frac{\rho}{1-\theta}\right)^k.
\end{aligned}
\end{eqnarray}
From the settings of $\theta$ and $\alpha$, we know $\theta\leq\frac{1-\sigma_{\gamma}}{116\gamma}$. From (\ref{cont26}), we further have $\frac{\rho}{1-\theta}< 1$ and $1-\rho-\theta\geq \frac{0.19(1-\sigma_{\gamma})}{\gamma}$. So we have $\sum_{k=r+1}^K\left(\frac{\rho}{1-\theta}\right)^k\leq \left(\frac{\rho}{1-\theta}\right)^r\frac{\rho}{1-\theta-\rho}$. It follows from $\|\Pi\y^0\|=0$ that
\begin{eqnarray}
\begin{aligned}\notag
&\sum_{k=0}^{T\gamma}\frac{L}{2m(1-\theta)^{k+1}}\|\Pi\y^k\|^2\leq \frac{C_3L}{2m\rho^{\gamma-1}(1-\theta)(1-\theta-\rho)}\\
&~~\quad+\frac{C_4L\theta^2}{2m\rho^{2\gamma-1}(1-\theta)(1-\theta-\rho)}\left(\sum_{s=0}^{\gamma-1}\frac{\Phi^s}{(1-\theta)^s}\left(\frac{1-\theta}{\rho}\right)^s+\left(\frac{\rho}{1-\theta}\right)^{-\gamma}\sum_{s=\gamma}^{T\gamma-1}\frac{\Phi^s}{(1-\theta)^s}\right)\\
&~~\overset{a}\leq \frac{C_3L}{2m\rho^{\gamma-1}(1-\theta)(1-\theta-\rho)}+\frac{C_4L\theta^2(1-\theta)^{\gamma}}{2m\rho^{3\gamma-1}(1-\theta-\rho)}\sum_{s=0}^{T\gamma-1}\frac{\Phi^s}{(1-\theta)^{s+1}}\\
&~~\overset{b}\leq \frac{3.3C_3L\gamma}{m(1-\theta)(1-\sigma_{\gamma})}+\frac{\theta^2}{7mL(1+\frac{1}{\tau})}\sum_{s=0}^{T\gamma-1}\frac{\Phi^s}{(1-\theta)^{s+1}}.
\end{aligned}
\end{eqnarray}
where $\overset{a}\leq$ uses $\frac{1-\theta}{\rho}>1$ such that $\left(\frac{1-\theta}{\rho}\right)^s\leq \left(\frac{1-\theta}{\rho}\right)^{\gamma}$ for all $s\leq\gamma-1$, $\overset{b}\leq$ uses (\ref{cont26}), $1-\rho-\theta\geq \frac{0.19(1-\sigma_{\gamma})}{\gamma}$, and the definition of $C_4$ given in Lemma \ref{lemma3}. Replacing $\|\Pi\y^k\|$ by $\|\Pi\x^k\|$ in the above analysis, we have the same bound for $\|\Pi\x^k\|^2$.
\end{proof}

Now, we are ready to prove Theorems \ref{theorem3} and \ref{theorem4}. We first prove Theorem \ref{theorem3}.
\begin{proof}
Plugging (\ref{cont30}) into (\ref{cont3}) and using the definition of $\Phi^r$ in (\ref{cont18}), we have
\begin{eqnarray}
\begin{aligned}\notag
&\frac{F(\overline x^{T\gamma+1})-F(x^*)}{\theta_{T\gamma}^2}+\frac{1}{2\alpha}\|\overline z^{T\gamma+1}-x^*\|^2\\
&~~\leq \frac{1}{2\alpha}\|\overline z^0-x^*\|^2+\frac{235\gamma^3C_3L}{m(1-\sigma_{\gamma})^3}-\sum_{k=0}^{T\gamma}\left(\left(\frac{1}{2\alpha}-\frac{L}{2}-\frac{20L\gamma^2}{(1-\sigma_{\gamma})^2}\right)\|\overline z^{t+1}-\overline z^t\|^2\right.\\
&~~\hspace*{6.9cm}\left.+\frac{1}{\theta_{k-1}^2}\left(1-\frac{20(1+\tau)\gamma^2}{(1+\frac{1}{\tau})(1-\sigma_{\gamma})^2}\right)D_f(\overline x^k,\y^k)\right)\\
&~~\overset{a}\leq \frac{1}{2\alpha}\|\overline z^0-x^*\|^2+\frac{235\gamma^3C_3L}{m(1-\sigma_{\gamma})^3}-\sum_{k=0}^{T\gamma}\left(\frac{1}{4\alpha}\|\overline z^{t+1}-\overline z^t\|^2+\frac{1}{2\theta_{k-1}^2}D_f(\overline x^k,\y^k)\right)\\
&~~\leq \frac{1}{2\alpha}\|\overline z^0-x^*\|^2+\frac{235\gamma^3C_3L}{m(1-\sigma_{\gamma})^3}-\frac{1}{5mL}\sum_{r=0}^{T\gamma-1}\Phi^r,
\end{aligned}
\end{eqnarray}
where in $\overset{a}\leq$ we let $\tau=\frac{(1-\sigma_{\gamma})^2}{40\gamma^2}$ so to have $\frac{20(1+\tau)\gamma^2}{(1+\frac{1}{\tau})(1-\sigma_{\gamma})^2}=\frac{1}{2}$,  $\alpha=\frac{(1-\sigma_{\gamma})^4}{21675L\gamma^4}\leq\frac{(1-\sigma_{\gamma})^3}{3385L\gamma^3\sqrt{1+\frac{1}{\tau}}}$, and $\frac{1}{4\alpha}\geq \frac{L}{2}+\frac{20L\gamma^2}{(1-\sigma_{\gamma})^2}$. So we have
\begin{eqnarray}
&&F(\overline x^{T\gamma+1})-F(x^*)\leq \theta_{T\gamma}^2\left(\frac{1}{2\alpha}\|\overline z^0-x^*\|^2+\frac{235\gamma^3C_3L}{m(1-\sigma_{\gamma})^3}\right),\label{cont32}\\
&&\frac{1}{5mL}\sum_{r=0}^{T\gamma-1}\Phi^r\leq \frac{1}{2\alpha}\|\overline z^0-x^*\|^2+\frac{235\gamma^3C_3L}{m(1-\sigma_{\gamma})^3}.\label{cont33}
\end{eqnarray}
It follows from (\ref{cont30}) that
\begin{eqnarray}
\begin{aligned}\label{cont34}
&\max\left\{\sum_{k=0}^{T\gamma}\frac{L}{2m\theta_k^2}\|\Pi\y^k\|^2,\sum_{k=0}^{T\gamma}\frac{L}{2m\theta_k^2}\|\Pi\x^k\|^2\right\}\\
&~~\leq\frac{235\gamma^3C_3L}{m(1-\sigma_{\gamma})^3}+\frac{10\gamma^2}{mL(1+\frac{1}{\tau})(1-\sigma_{\gamma})^2}\sum_{s=0}^{T\gamma-1}\Phi^s\\
&~~\leq\frac{235\gamma^3C_3L}{m(1-\sigma_{\gamma})^3}+\frac{1}{4mL}\sum_{s=0}^{T\gamma-1}\Phi^s\\
&~~\leq\frac{9}{4}\left(\frac{1}{2\alpha}\|\overline z^0-x^*\|^2+\frac{235\gamma^3C_3L}{m(1-\sigma_{\gamma})^3}\right).
\end{aligned}
\end{eqnarray}
From the definition of $C_3$ given in Lemma \ref{lemma3}, we have
\begin{eqnarray}
\begin{aligned}\label{cont41}
&\frac{1}{2\alpha}\|\overline z^0-x^*\|^2+\frac{235\gamma^3C_3L}{m(1-\sigma_{\gamma})^3}\\
&~~\leq\frac{1}{2\alpha}\|\overline z^0-x^*\|^2+\frac{1-\sigma_{\gamma}}{27mL\gamma}\M_{\s}^{\gamma,\gamma}+\frac{103870L\gamma^5}{m(1-\sigma_{\gamma})^5}\M_{\z}^{\gamma,\gamma}+\frac{470L\gamma^3}{m(1-\sigma_{\gamma})^3}\M_{\x}^{\gamma,\gamma}\equiv C_5.
\end{aligned}
\end{eqnarray}
The conclusion follows from Lemma \ref{lemma5}.
\end{proof}

The next lemma gives a sharper bound of the constant $C_5$ appeared in the above proof.
\begin{lemma}\label{lemma5}
Under the settings of Theorem \ref{theorem3}, we can further bound $C_5$ by
\begin{eqnarray}
C_5\leq\frac{1}{2\alpha}\|\overline z^0-x^*\|^2+\frac{1-\sigma_{\gamma}}{20mL\gamma}\max_{r=0,...,\gamma}\|\Pi\s^r\|^2.\label{cont42}
\end{eqnarray}
\end{lemma}
\begin{proof}
From step (\ref{alg-tv-s3}) with $\mu=0$, we have for any $k\leq \gamma-1$,
\begin{eqnarray}
\begin{aligned}\notag
\theta_{k+1}\|\Pi\z^{k+1}\|\leq& \theta_k\|\Pi\z^{k+1}\|\leq \theta_k\|\Pi\z^k\|+\alpha\|\Pi\s^k\|\\
\leq&\theta_0\|\Pi\z^0\|+\alpha\sum_{t=0}^k\|\Pi\s^t\|\leq\alpha\sum_{t=0}^{\gamma-1}\|\Pi\s^t\|
\end{aligned}
\end{eqnarray}
where we use $\Pi\z^0=0$. Squaring both sides gives
\begin{eqnarray}
\theta_{k+1}^2\|\Pi\z^{k+1}\|^2\leq \alpha^2\gamma\sum_{t=0}^{\gamma-1}\|\Pi\s^t\|^2\leq \alpha^2\gamma^2\max_{r=0,...,\gamma}\|\Pi\s^r\|^2.\notag
\end{eqnarray}
From the setting of $\alpha$ and the definition of $\M_{\z}^{\gamma,\gamma}$, we have
\begin{eqnarray}
\frac{103870L\gamma^5}{m(1-\sigma_{\gamma})^5}\M_{\z}^{\gamma,\gamma}\leq \frac{1-\sigma_{\gamma}}{4523mL\gamma}\max_{r=0,...,\gamma}\|\Pi\s^r\|^2.\notag
\end{eqnarray}
On the other hand, it follows from step (\ref{alg-tv-s4}) that
\begin{eqnarray}
\|\Pi\x^{k+1}\|\leq\theta_k\|\Pi\z^{k+1}\|+\|\Pi\x^k\|\leq \sum_{t=0}^k\theta_t\|\Pi\z^{t+1}\|\leq 1.62\sum_{t=0}^{\gamma-1}\theta_{t+1}\|\Pi\z^{t+1}\|.\notag
\end{eqnarray}
Squaring both sides gives
\begin{eqnarray}
\|\Pi\x^{k+1}\|^2\leq 2.63\gamma\sum_{t=0}^{\gamma-1}\theta_{t+1}^2\|\Pi\z^{t+1}\|^2\leq 2.63\gamma^4\alpha^2\max_{r=0,...,\gamma}\|\Pi\s^r\|^2.\notag
\end{eqnarray}
From the setting of $\alpha$ and the definition of $\M_{\x}^{\gamma,\gamma}$, we have
\begin{eqnarray}
\frac{470L\gamma^3}{m(1-\sigma_{\gamma})^3}\M_{\x}^{\gamma,\gamma}\leq \frac{1-\sigma_{\gamma}}{380070mL\gamma}\max_{r=0,...,\gamma}\|\Pi\s^r\|^2.\notag
\end{eqnarray}
So we have the conclusion.
\end{proof}

In the next lemma, we measure the convergence rate at $x_{(i)}^{t\gamma+1}$ for any $i=1,...,m$.
\begin{lemma}\label{lemma8}
Under the settings of Theorem \ref{theorem3}, we have for any $t\leq T-1$,
\begin{eqnarray}
\begin{aligned}\notag
&F(x_{(i)}^{t\gamma+1})-F(x^*)\\
&~~\leq \frac{1}{(t\gamma+1)^2}\max\left\{\frac{\sqrt{m}(1-\sigma_{\gamma})}{L\alpha\gamma},8m\right\}\left(\frac{2}{\alpha}\|\overline z^0-x^*\|^2+\frac{1-\sigma_{\gamma}}{5mL\gamma}\max_{r=0,...,\gamma}\|\Pi\s^r\|^2\right).
\end{aligned}
\end{eqnarray}
\end{lemma}
\begin{proof}
We first bound $F(x_{(i)}^k)-F(\overline x^k)$ for any $i$. From Lemma \ref{lemma_inexact}, we have
\begin{eqnarray}
\begin{aligned}\notag
F(x_{(i)}^k)\leq& f(\overline y^{k-1},\y^{k-1})+\<\overline s^{k-1},x_{(i)}^k-\overline y^{k-1}\>+\frac{L}{2}\|x_{(i)}^k-\overline y^{k-1}\|^2+\frac{L}{2m}\|\Pi\y^{k-1}\|^2\\
\leq& F(\overline x^k)+\<\overline s^{k-1},x_{(i)}^k-\overline x^k\>+L\|x_{(i)}^k-\overline x^k\|^2+L\|\overline x^k-\overline y^{k-1}\|^2+\frac{L}{2m}\|\Pi\y^{k-1}\|^2\\
\overset{a}\leq& F(\overline x^k)\hspace*{-0.03cm}+\hspace*{-0.03cm}\frac{\theta_{k-1}}{\alpha}\|\overline z^k\hspace*{-0.03cm}-\hspace*{-0.03cm}\overline z^{k-1}\|\|\Pi\x^k\|\hspace*{-0.03cm}+\hspace*{-0.03cm}L\|\Pi\x^k\|^2\hspace*{-0.03cm}+\hspace*{-0.03cm}L\theta_{k-1}^2\|\overline z^k\hspace*{-0.03cm}-\hspace*{-0.03cm}\overline z^{k-1}\|^2\hspace*{-0.03cm}+\hspace*{-0.03cm}\frac{L}{2m}\|\Pi\y^{k-1}\|^2,
\end{aligned}
\end{eqnarray}
where we use (\ref{alg-ag-s3}) with $\mu=0$, (\ref{alg-ag-s1}), and (\ref{alg-ag-s4}) in $\overset{a}\leq$. From the definition of $\Phi^r$ in (\ref{cont18}), it follows from (\ref{cont33}) that for any $k\leq T\gamma$,
\begin{equation}
\|\overline z^k-\overline z^{k-1}\|^2\leq\frac{\Phi^{k-1}}{2mL^2(1+\frac{1}{\tau})}\overset{b}\leq \frac{5(1-\sigma_{\gamma})^2}{80L\gamma^2}\left(\frac{1}{2\alpha}\|\overline z^0-x^*\|^2+\frac{235\gamma^3C_3L}{m(1-\sigma_{\gamma})^3}\right),\notag
\end{equation}
where $\overset{b}\leq$ uses the setting of $\tau=\frac{(1-\sigma_{\gamma})^2}{40\gamma^2}$ given in the proof of Theorem \ref{theorem3}. From (\ref{cont32}) and (\ref{cont34}), we have for any $t\gamma+1$ with $t\leq T-1$
\begin{eqnarray}
\begin{aligned}\notag
F(x_{(i)}^{t\gamma+1})-F(x^*)\leq \theta_{t\gamma}^2\max\left\{\frac{\sqrt{m}(1-\sigma_{\gamma})}{L\alpha\gamma},8m\right\}\left(\frac{1}{2\alpha}\|\overline z^0-x^*\|^2+\frac{235\gamma^3C_3L}{m(1-\sigma_{\gamma})^3}\right).
\end{aligned}
\end{eqnarray}
From (\ref{cont41}), (\ref{cont42}), and (\ref{theta_pro}), we have the conclusion.
\end{proof}

Next, we prove Theorem \ref{theorem4}.
\begin{proof}
Plugging (\ref{cont31}) into (\ref{cont4}) and using the definition of $\Phi^r$ in (\ref{cont18}), we have
\begin{eqnarray}
\begin{aligned}\notag
&\frac{1}{(1-\theta)^{T\gamma+1}}\left(F(\overline x^{T\gamma+1})-F(x^*)+\left(\frac{\theta^2}{2\alpha}+\frac{\mu\theta}{2}\right)\|\overline z^{T\gamma+1}-x^*\|^2\right)\\
&~~\leq F(\overline x^0)-F(x^*)+\left(\frac{\theta^2}{2\alpha}+\frac{\mu\theta}{2}\right)\|\overline z^0-x^*\|^2+\frac{3.3C_3L\gamma}{m(1-\theta)(1-\sigma_{\gamma})}\\
&~~\quad-\sum_{k=0}^{T\gamma}\left(\frac{1}{(1-\theta)^k}\left(1-\frac{2(1+\tau)}{7(1+\frac{1}{\tau})}\right)D_f(\overline x^k,\y^k)+\frac{1}{(1-\theta)^{k+1}}\left(\frac{\theta^2}{2\alpha}-\frac{L\theta^2}{2}-\frac{2L\theta^2}{7}\right)\|\overline z^{k+1}-\overline z^k\|^2\right)\\
&~~\overset{a}\leq F(\overline x^0)-F(x^*)+\left(\frac{\theta^2}{2\alpha}+\frac{\mu\theta}{2}\right)\|\overline z^0-x^*\|^2+\frac{3.3C_3L\gamma}{m(1-\theta)(1-\sigma_{\gamma})}\\
&~~\quad-\sum_{k=0}^{T\gamma}\left(\frac{1}{2(1-\theta)^k}D_f(\overline x^k,\y^k)+\frac{\theta^2}{4\alpha(1-\theta)^{k+1}}\|\overline z^{k+1}-\overline z^k\|^2\right)\\
&~~\leq F(\overline x^0)-F(x^*)\hspace*{-0.02cm}+\hspace*{-0.02cm}\left(\frac{\theta^2}{2\alpha}\hspace*{-0.02cm}+\hspace*{-0.02cm}\frac{\mu\theta}{2}\right)\|\overline z^0\hspace*{-0.02cm}-\hspace*{-0.02cm}x^*\|^2\hspace*{-0.02cm}+\hspace*{-0.02cm}\frac{3.3C_3L\gamma}{m(1-\theta)(1-\sigma_{\gamma})}\hspace*{-0.02cm}-\hspace*{-0.02cm}\frac{\theta^2}{11mL}\sum_{r=0}^{T\gamma-1}\frac{\Phi^r}{(1-\theta)^{r+1}},
\end{aligned}
\end{eqnarray}
where in $\overset{a}\leq$ we let $\tau=\frac{7}{4}$ so to have $\frac{2(1+\tau)}{7(1+\frac{1}{\tau})}=\frac{1}{2}$, $\alpha=\frac{(1-\sigma_{\gamma})^3}{4244L\gamma^3}\leq\frac{(1-\sigma_{\gamma})^3}{3385L\gamma^3\sqrt{1+\frac{1}{\tau}}}$, and $\frac{1}{4\alpha}\geq\frac{L}{2}+\frac{2L}{7}$. Thus, we have the first conclusion and
\begin{eqnarray}
\frac{\theta^2}{11mL}\sum_{r=0}^{T\gamma-1}\frac{\Phi^r}{(1-\theta)^{r+1}}\leq F(\overline x^0)-F(x^*)+\left(\frac{\theta^2}{2\alpha}+\frac{\mu\theta}{2}\right)\|\overline z^0-x^*\|^2+\frac{3.3C_3L\gamma}{m(1-\theta)(1-\sigma_{\gamma})}.\notag
\end{eqnarray}
It follows from (\ref{cont31}) that
\begin{eqnarray}
\begin{aligned}\notag
&\sum_{k=0}^{T\gamma}\frac{L}{2m(1-\theta)^{k+1}}\|\Pi\x^k\|^2\\
&~~\leq\frac{3.3C_3L\gamma}{m(1-\theta)(1-\sigma_{\gamma})}+\frac{\theta^2}{7mL(1+\frac{1}{\tau})}\sum_{s=0}^{T\gamma-1}\frac{\Phi^s}{(1-\theta)^{s+1}}\\
&~~\leq 2\left(F(\overline x^0)-F(x^*)+\left(\frac{\theta^2}{2\alpha}+\frac{\mu\theta}{2}\right)\|\overline z^0-x^*\|^2+\frac{3.3C_3L\gamma}{m(1-\theta)(1-\sigma_{\gamma})}\right).
\end{aligned}
\end{eqnarray}
Thus, we have the second conclusion by plugging the definition of $C_3$ in Lemma \ref{lemma3}.
\end{proof}

\begin{remark}\label{complexity-remark}
We rewrite the convergence rates in Theorems \ref{theorem3} and \ref{theorem4} in the form of complexities. For the nonstrongly convex case, letting $F(\overline x^{T\gamma+1})-F(x^*)\leq \frac{2C}{\alpha(T\gamma+1)^2}=\epsilon$, we have $T\gamma+1=\sqrt{\frac{2C}{\alpha\epsilon}}=\bO((\frac{\gamma}{1-\sigma_{\gamma}})^2\sqrt{\frac{LC}{\epsilon}})$. Each iteration only requires $\bO(1)$ communication round and gradient oracle call. For the strongly convex case, letting $F(\overline x^{T\gamma+1})-F(x^*)\leq (1-\theta)^{T\gamma+1}C=\epsilon$, we have $T\gamma+1=\bO(\frac{1}{\theta}\log\frac{C}{\epsilon})=\bO((\frac{\gamma}{1-\sigma_{\gamma}})^{1.5}\sqrt{\frac{L}{\mu}}\log\frac{1}{\epsilon})$.
\end{remark}

\section{Numerical Experiments}
In this section, we test the performance of the accelerated gradient tracking (Acc-GT) over time-varying graphs. The performance of Acc-GT over static graphs has already been verified in \citep{qu2017-2}. Moreover, \citet{qu2017-2} reported in their experiment that algorithm (\ref{qu-s1})-(\ref{qu-s4}) with fixed step size (our theoretical setting) performs faster than the one with vanishing step sizes (their theoretical setting). Thus, we omit the comparisons over static graphs.

We consider the following decentralized regularized logistic regression problem:
\begin{eqnarray}\notag
\min_{x\in\R^p} \sum_{i=1}^m f_{(i)}(x),\quad\mbox{where}\quad f_{(i)}(x)=\frac{\mu}{2}\|x\|^2+\frac{1}{n}\sum_{j=1}^n \log\left(1+\exp(-y_{(i),j}\A_{(i),j}^{\top}x)\right),
\end{eqnarray}
where $(\A_{(i),j},y_{(i),j})\in\R^p\times\{1,-1\}$ is the data point with $\A_{(i),j}$ being the feature vector, and $y_{(i),j}$ the label. We use the cifar10 dataset with $p=3072$, $n=50$, and $m=1000$. Each feature vector is normalized to have unit norm, and the data are divided into two classes to fit the logistic regression model. We observe that $L=\max_i\frac{\|\A_{(i)}\|_2^2}{4n}\approx0.215$. We consider both strongly convex ($\mu=10^{-6}$) and nonstrongly convex ($\mu=0$) problems. We test the performance on the 2D grid graphs, where at each iteration, $m$ nodes are uniformly placed in a $\lceil5\sqrt{m}\rceil\times \lceil5\sqrt{m}\rceil$ region in random, and each node is connected with the nodes around it within the distance of $d$. We test on $d=20$ and $d=2$, which correspond to $(\gamma,\sigma_{\gamma})\approx (1,0.9858)$ and $(\gamma,\sigma_{\gamma})\approx (32,0.9471)$, respectively. When $d=20$, the network is connected almost every time. When $d=2$, we observe that at each iteration, almost $61$ percent of the nodes drop out from the communication network in average, which means that they have no connection with the other nodes. We use the Metropolis gossip matrix given in (\ref{weight-matrix}).

For strongly convex problem, we compare Acc-GT and Acc-GT-C (Acc-GT with multiple consensus) with 
DIGing \citep{shi2017}, DAGD-C \citep{Alexander2221}, as well as the classical non-distributed accelerated gradient descent (AGD), where AGD runs on a single machine, and it gives the upper limit of the practical performance of the distributed algorithms.
We do not compare with the time-varying $\mathcal{A}\mathcal{B}$/push-pull method \citep{tvab20} and the push-sum based methods \citep{pushsum16,pushsum15,shi2017} because they are designed for directed graphs. We tune the step sizes $\alpha=\frac{0.1}{L}$ for Acc-GT and Acc-GT-C, $\alpha=\frac{0.5}{L}$ for DIGing, and $\alpha=\frac{1}{L}$ for AGD. For DAGD-C, when $d=2$, we test on the number of inner iterations as $T=\frac{\gamma}{3(1-\sigma_{\gamma})}\approx 201$ and $T=\frac{\gamma}{2(1-\sigma_{\gamma})}\approx 302$, and name the methods DAGD-C1 and DAGD-C2, respectively. When $d=20$, we test on $T=\frac{\gamma}{5(1-\sigma_{\gamma})}\approx 14$ and $T=\frac{\gamma}{4(1-\sigma_{\gamma})}\approx 17$, respectively. For Acc-GT-C, we set the number of inner iterations as $T=\frac{\gamma}{50(1-\sigma_{\gamma})}\approx 12$ and $T=\frac{\gamma}{10(1-\sigma_{\gamma})}\approx 7$ for $d=2$ and $d=20$, respectively. The other parameter settings follow the corresponding theorems of each method. For nonstrongly convex problem, we compare Acc-GT and Acc-GT-C with DIGing \citep{shi2017}, APM \citep{li-2018-pm}, and AGD, and set the same step sizes as above. We tune the step size $\alpha=\frac{1}{L}$ for APM, and set the number of inner iterations as $T_k=\frac{\gamma\log(k+1)}{100(1-\sigma_{\gamma})}$ and $T_k=\frac{\gamma\log(k+1)}{10(1-\sigma_{\gamma})}$ at each outer loop iteration for $d=2$ and $d=20$, respectively. Although the convergence of DIGing was only proved for strongly convex problem in \citep{shi2017}, it also converges for nonstrongly convex ones by using our proof techniques.

Figures \ref{fig1}-\ref{fig4} plot the results, where the objective function error is measured by $F(\overline x^k)-F(x^*)$, and the consensus error is measured by $\sqrt{\frac{\sum_{i=1}^m\|x_{(i)}^k-\overline x^k\|^2}{m\|\overline x^k\|^2}}$. Since $F(x^*)$ is unknown, we approximate it by the output of the classical non-distributed AGD with 50000 iterations for strongly convex problem, and 200000 iterations for nonstrongly convex one. One round of communications means that all the nodes, if they are active, receive information from their neighbors once, and one round of gradient computations means that all the nodes compute their gradient $\nabla f_{(i)}(x)$ once in parallel. Especially, for AGD, one round of gradient computations means computing the full gradient $\sum_{i=1}^m \nabla f_{(i)}(x)$ once. We have the following observations:
\begin{enumerate}
\item Acc-GT converges faster than DIGing, both on the decrease of the objective function errors and consensus errors. This verifies the efficiency of the acceleration technique. Moreover, for strongly convex problem, Acc-CT is only three times slower than the classical non-distributed AGD.
\item Acc-GT-C needs more communication rounds than Acc-GT to reach the same precision of the objective function error, although Acc-GT-C has lower theoretically communication round complexity. Thus, Acc-GT-C is only for the theoretical interest, and it is not suggested in practice.
\item DAGD-C and APM need less gradient computation rounds than Acc-GT to reach the same precision of the objective function error, but they require more communication rounds. This supports that the multiple consensus subroutine places more communication burdens in practice. But on the other hand, DAGD-C and APM have almost the same computation cost as the classical non-distributed AGD. Comparing DAGD-C1 with DAGD-C2, we see that less inner iterations give larger consensus errors, and our settings of the inner iteration numbers are fair to DAGD-C.
\item The network connectivity, that is, the different settings of $d$ in our experiment, has little influence on the decrease of the objective function errors for both DIGing and Acc-GT\footnotemark[3]. We think this is because we set the same step sizes for $d=2$ and $d=20$. From Theorems \ref{theorem3} and \ref{theorem4}, we see that the network connectivity constants impact on the step sizes, and the step sizes impact on the decrease speed of the objective function errors. On the other hand, from the proofs of Theorems \ref{theorem3} and \ref{theorem4}, we see that the decrease speed of the consensus errors given in the two theorems is not tight, and we observe in the experiment that the consensus errors decrease faster when $d=20$ for both DIGing and Acc-GT.
\end{enumerate}

\footnotetext[3]{This phenomenon depends on the data. We also test on the simulated data with $p=100$, $n=50$, and $m=1000$, where each element of the feature vectors is generated randomly in $[0,1]$ from the uniform distribution, we observe that Acc-GT with $d=20$ performs about 1.1 times as fast as that with $d=2$. The difference is not significant.}

\begin{figure}[pt]
\begin{tabular}{@{\extracolsep{0.001em}}c@{\extracolsep{0.001em}}c}
\includegraphics[width=0.5\textwidth,keepaspectratio]{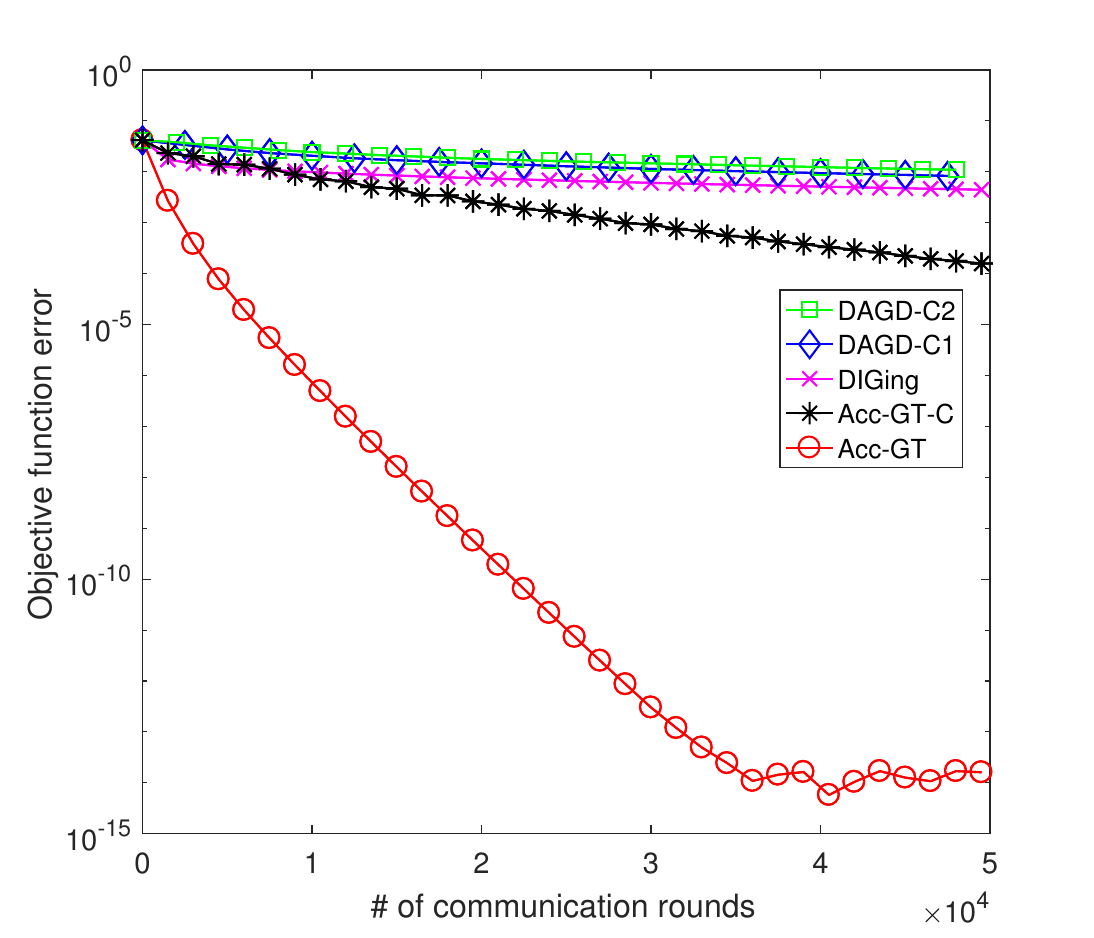}
&\includegraphics[width=0.5\textwidth,keepaspectratio]{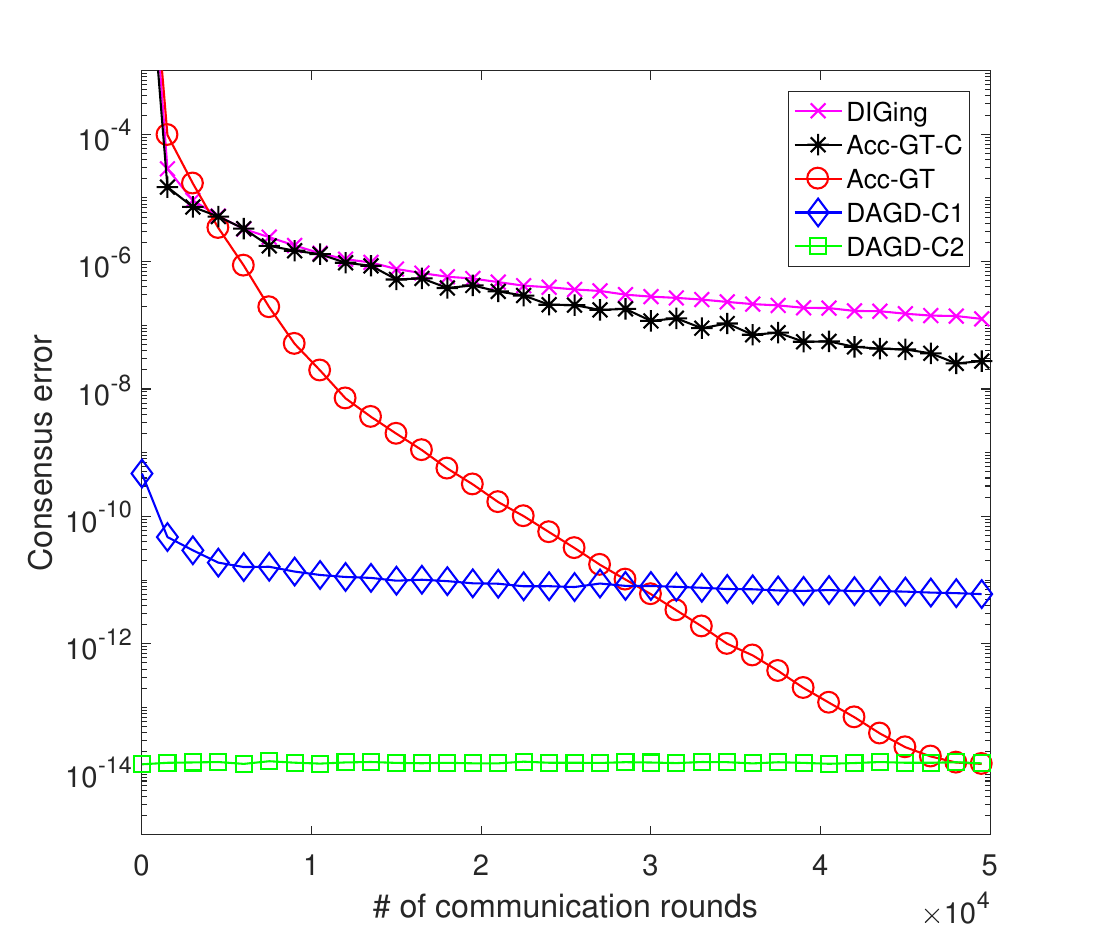}\\
\includegraphics[width=0.5\textwidth,keepaspectratio]{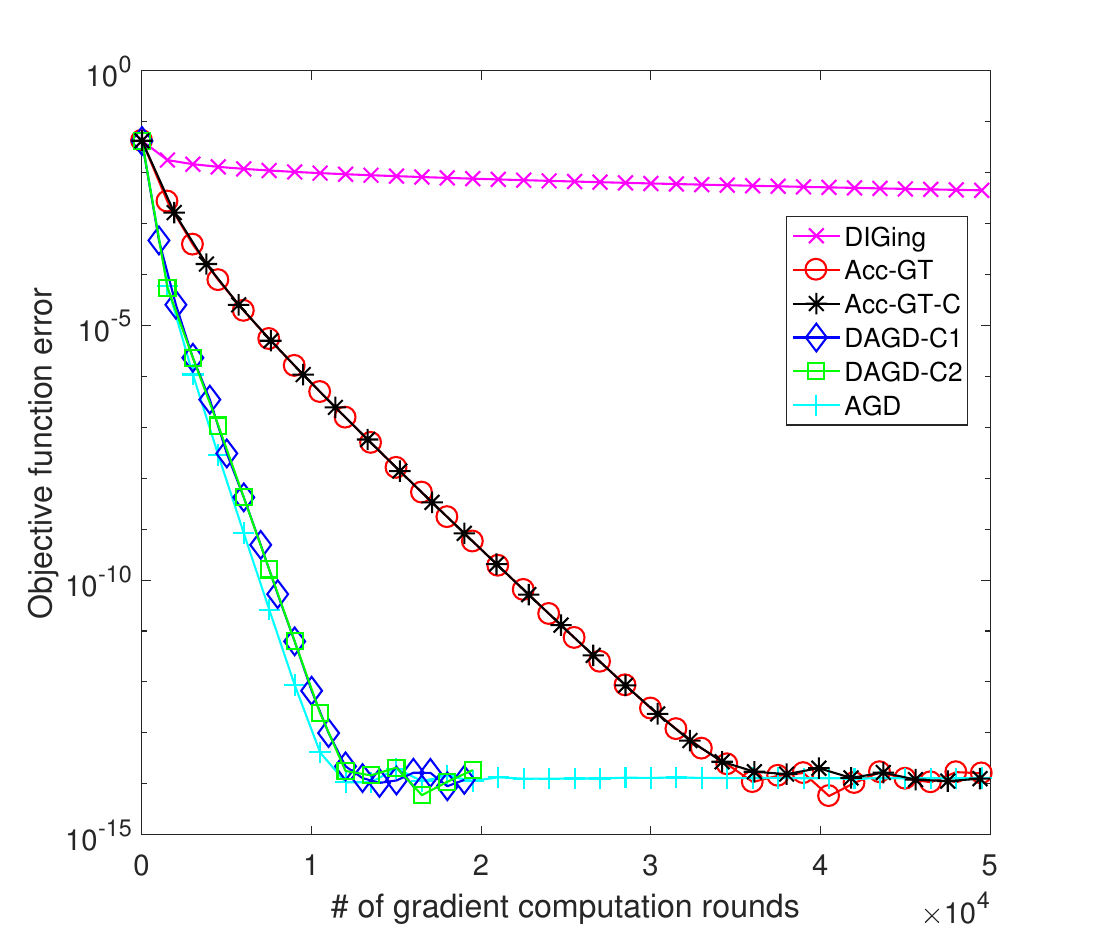}
&\includegraphics[width=0.5\textwidth,keepaspectratio]{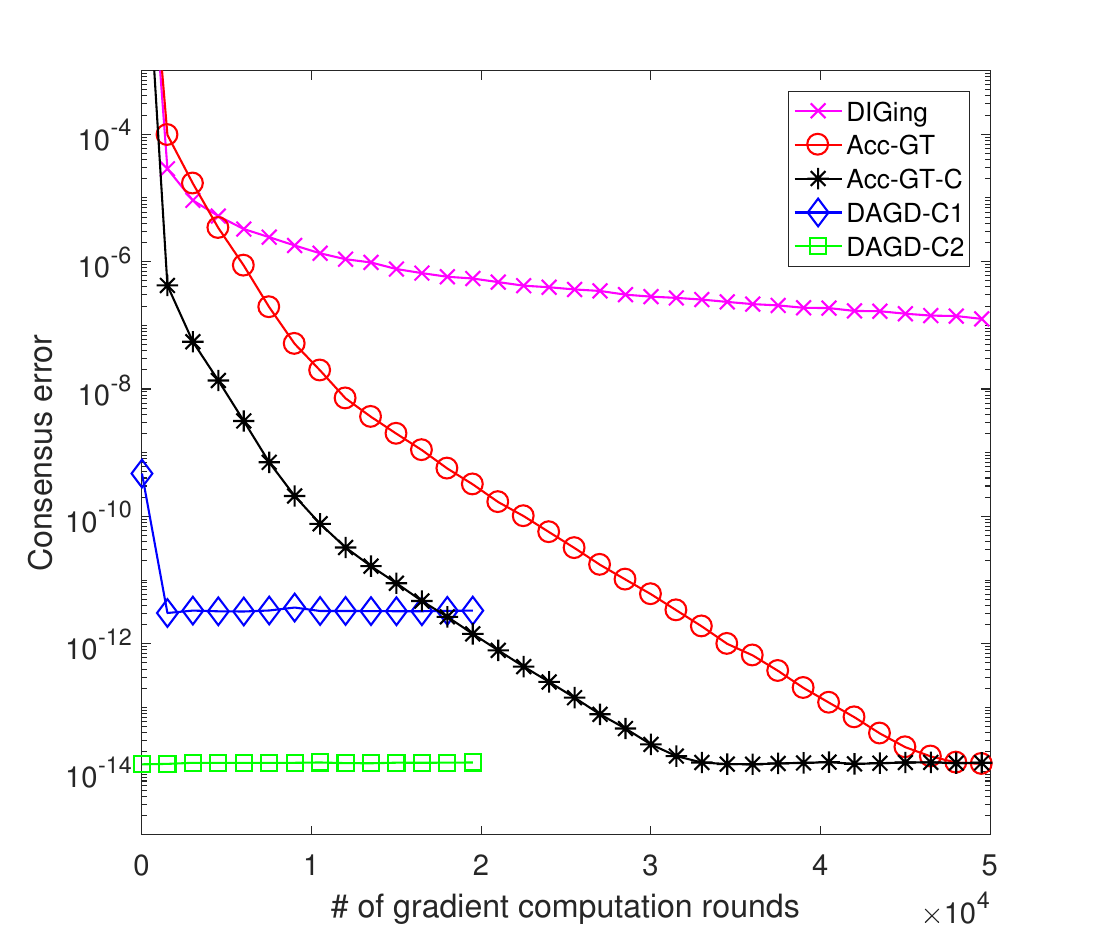}\\
\end{tabular}
\caption{Comparisons of the objective function errors (left) and consensus errors (right) with respect to the number of communication (top) and computation (bottom) rounds for strongly convex problem with $d=2$.}\label{fig1}
\end{figure}

\begin{figure}[pt]
\begin{tabular}{@{\extracolsep{0.001em}}c@{\extracolsep{0.001em}}c}
\includegraphics[width=0.5\textwidth,keepaspectratio]{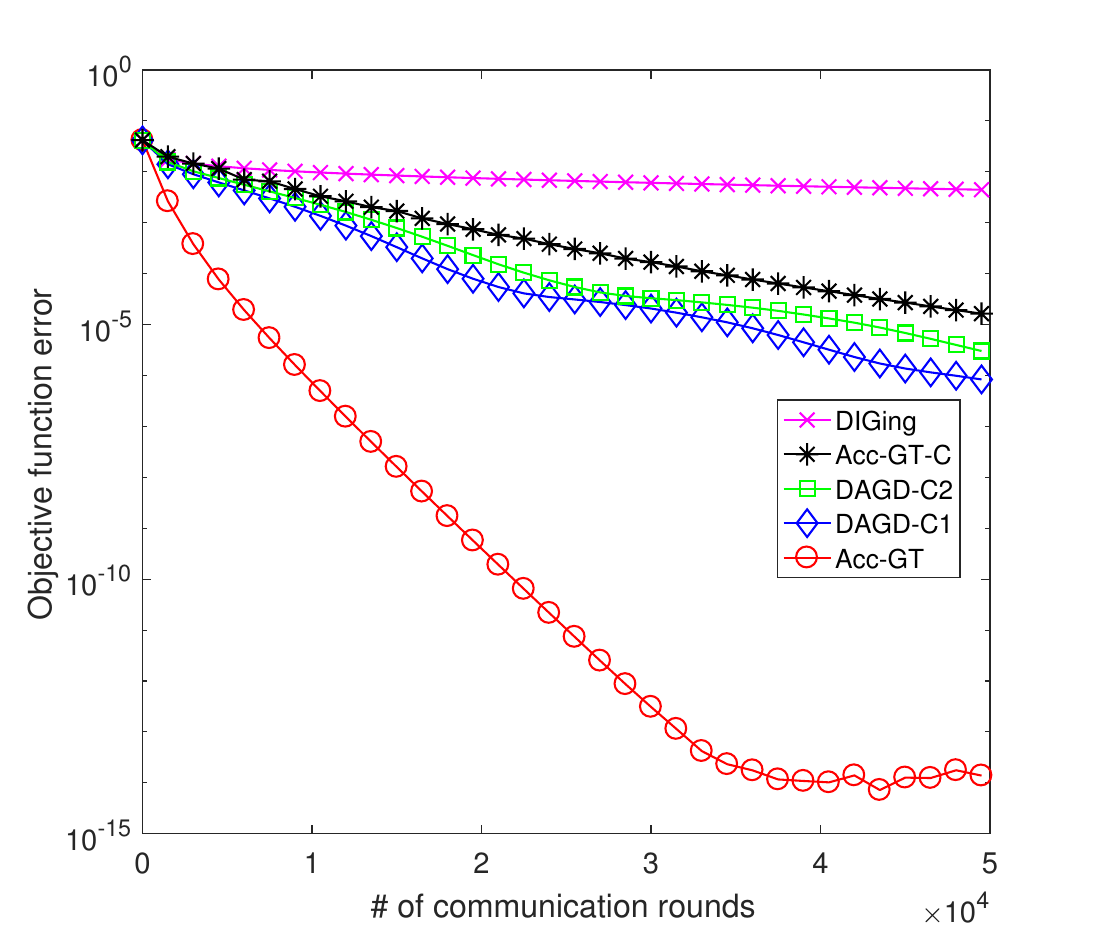}
&\includegraphics[width=0.5\textwidth,keepaspectratio]{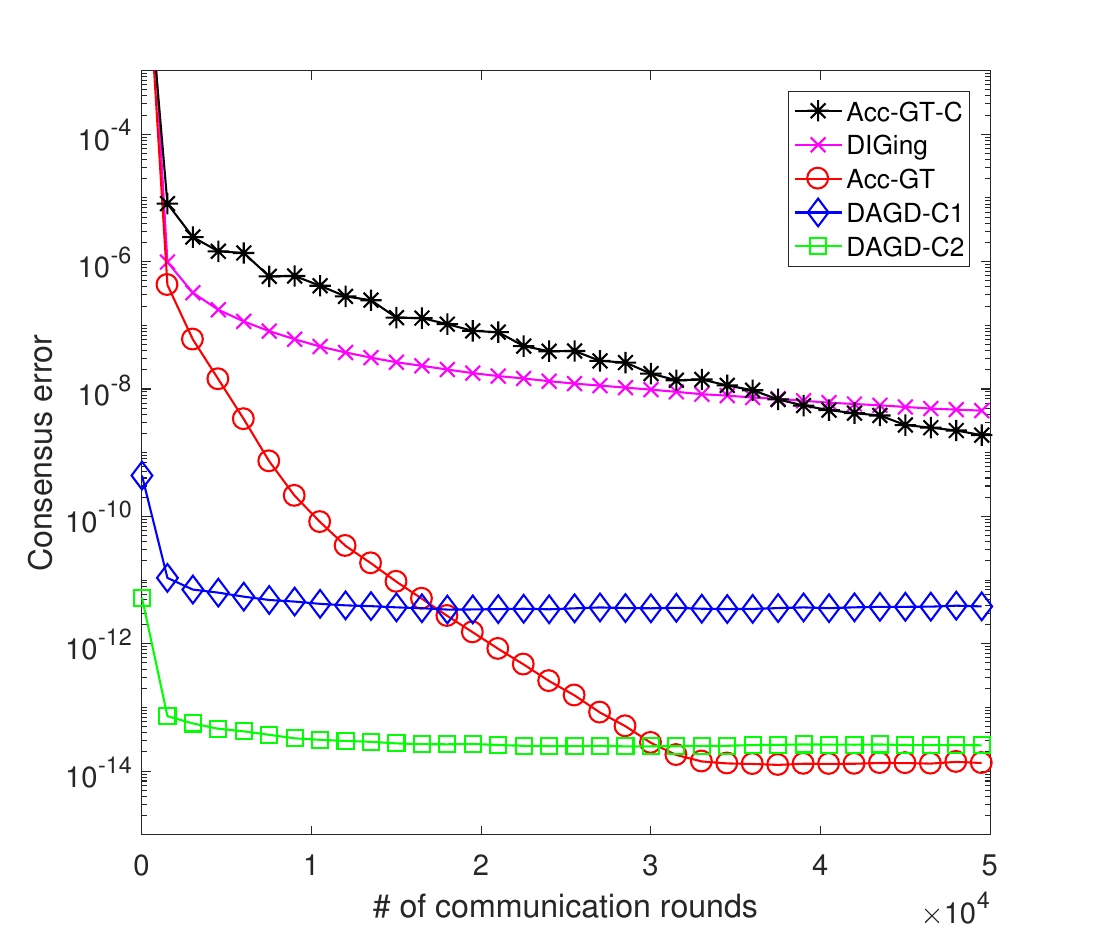}\\
\includegraphics[width=0.5\textwidth,keepaspectratio]{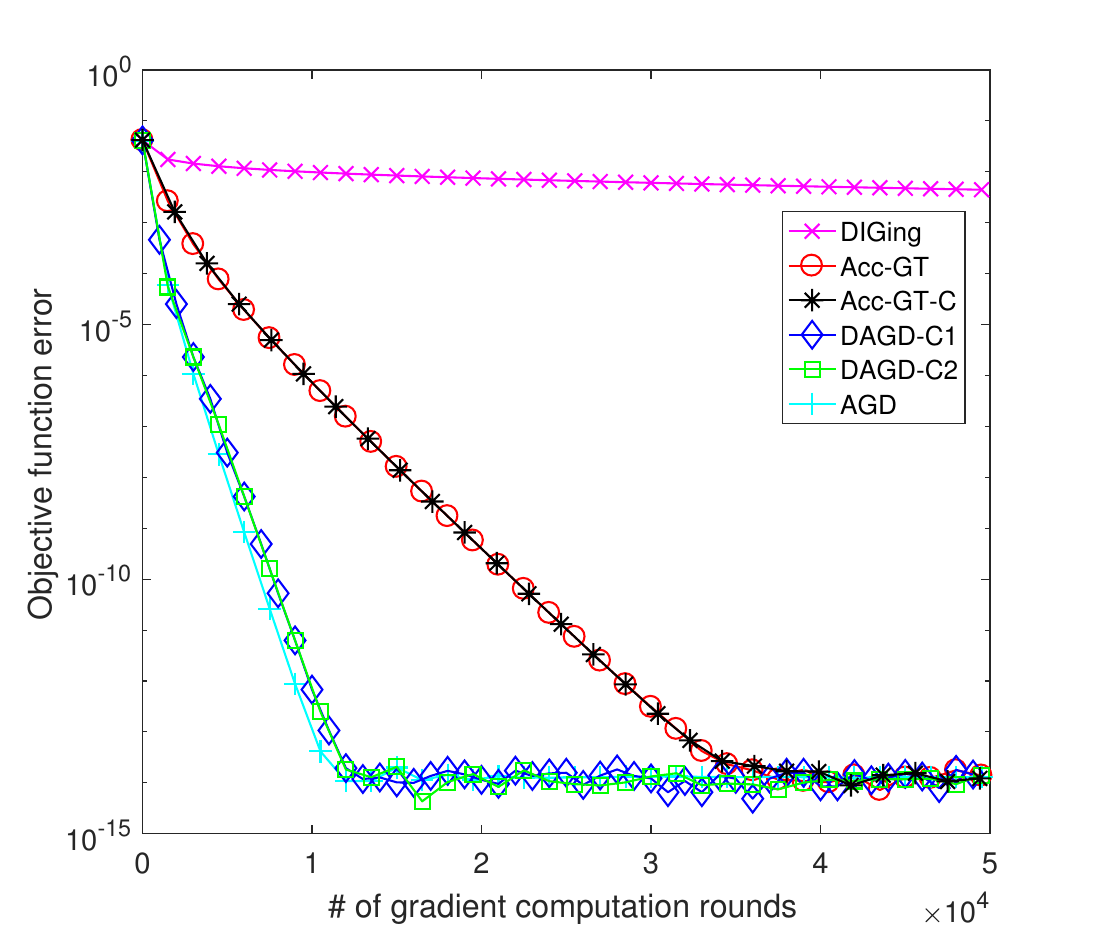}
&\includegraphics[width=0.5\textwidth,keepaspectratio]{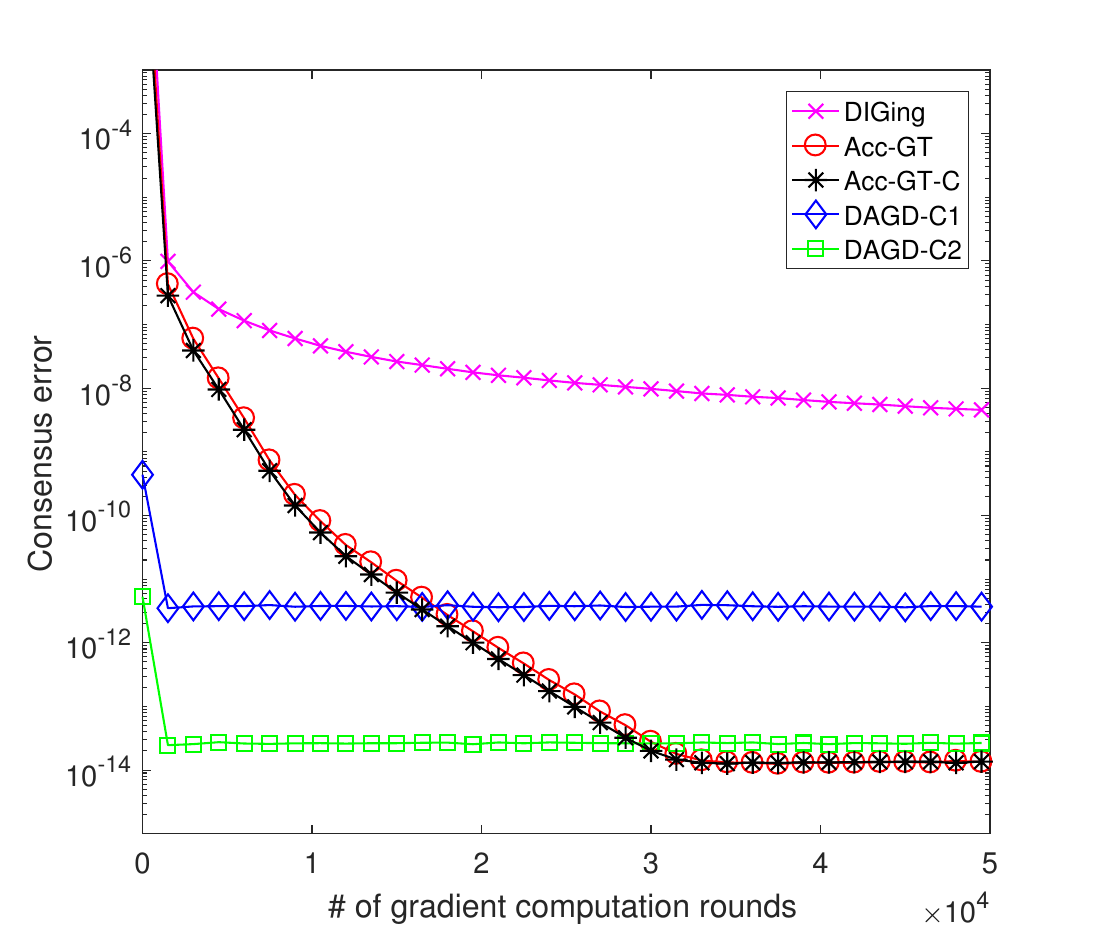}\\
\end{tabular}
\caption{Comparisons of the objective function errors (left) and consensus errors (right) with respect to the number of communication (top) and computation (bottom) rounds for strongly convex problem with $d=20$.}\label{fig2}
\end{figure}

\begin{figure}[pt]
\begin{tabular}{@{\extracolsep{0.001em}}c@{\extracolsep{0.001em}}c}
\includegraphics[width=0.5\textwidth,keepaspectratio]{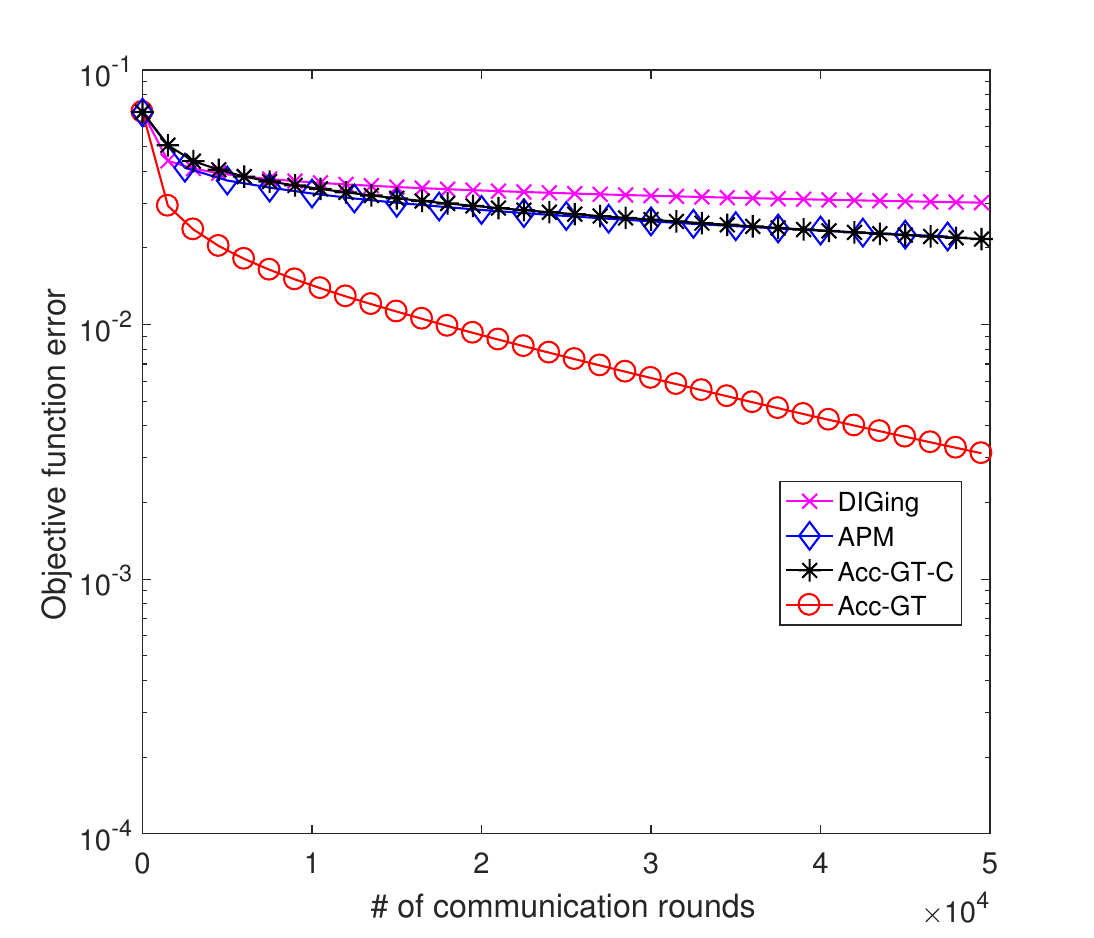}
&\includegraphics[width=0.5\textwidth,keepaspectratio]{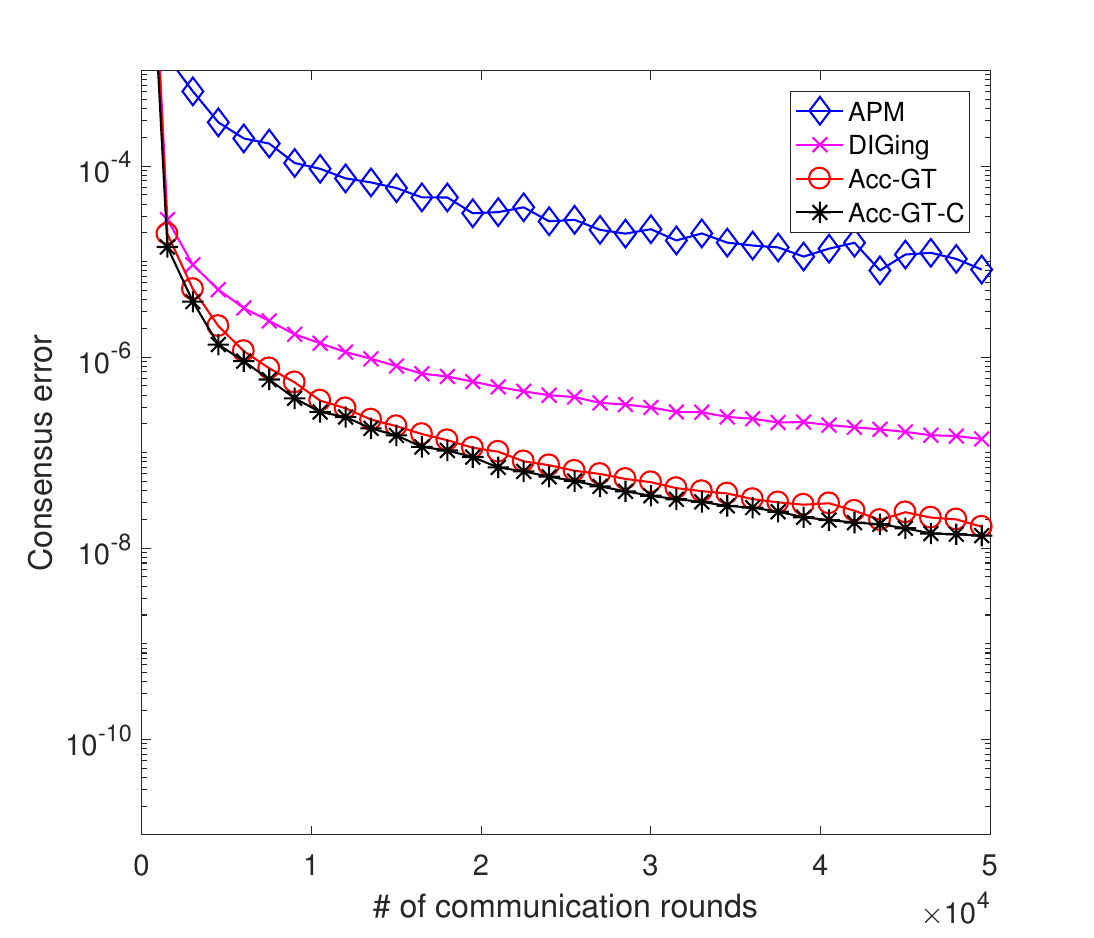}\\
\includegraphics[width=0.5\textwidth,keepaspectratio]{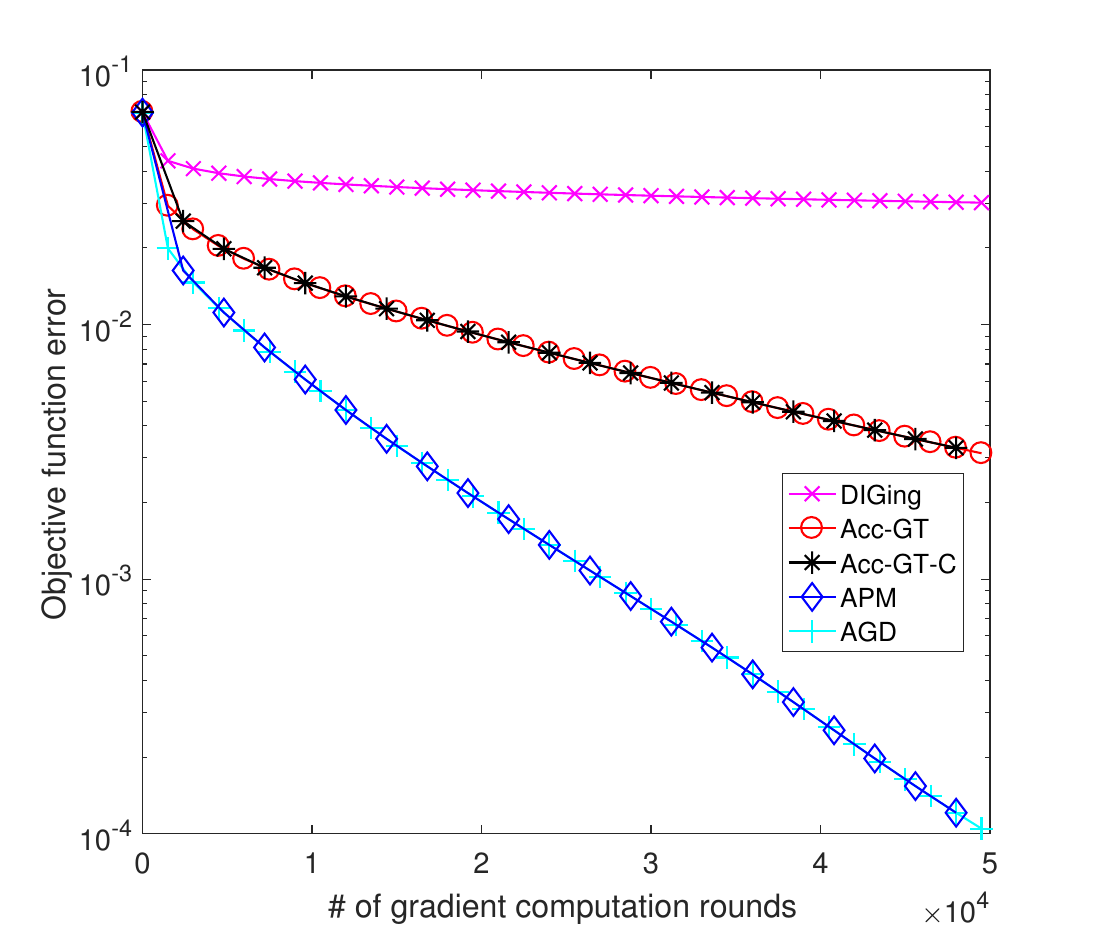}
&\includegraphics[width=0.5\textwidth,keepaspectratio]{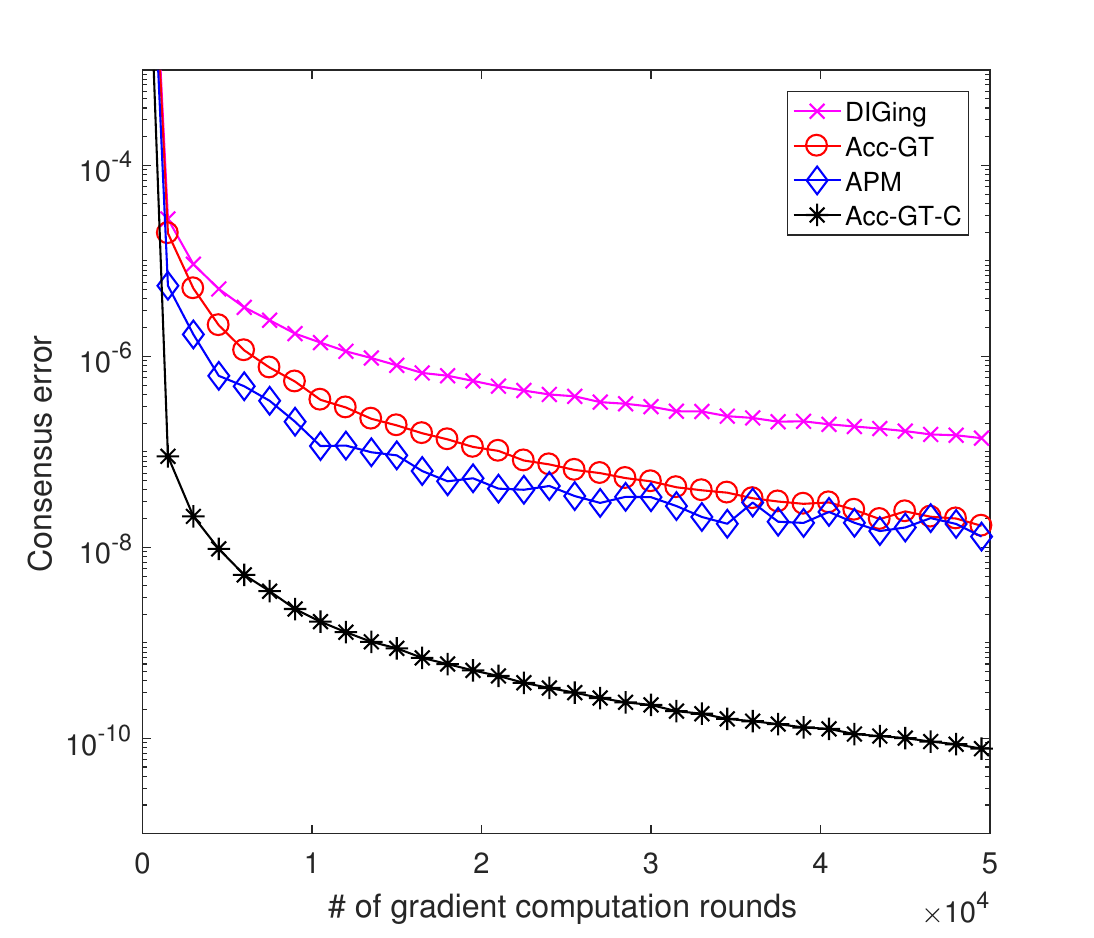}\\
\end{tabular}
\caption{Comparisons of the objective function errors (left) and consensus errors (right) with respect to the number of communication (top) and computation (bottom) rounds for nonstrongly convex problem with $d=2$.}\label{fig3}
\end{figure}

\begin{figure}[pt]
\begin{tabular}{@{\extracolsep{0.001em}}c@{\extracolsep{0.001em}}c}
\includegraphics[width=0.5\textwidth,keepaspectratio]{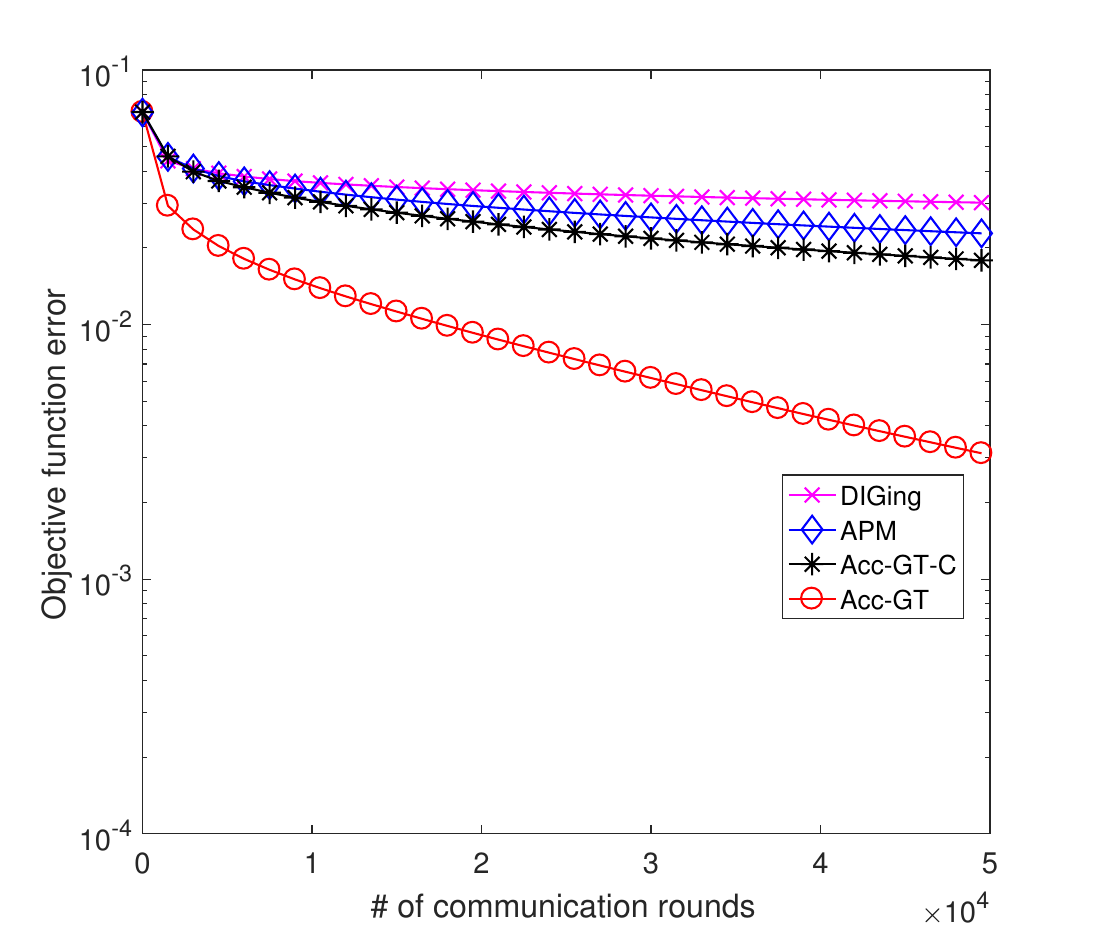}
&\includegraphics[width=0.5\textwidth,keepaspectratio]{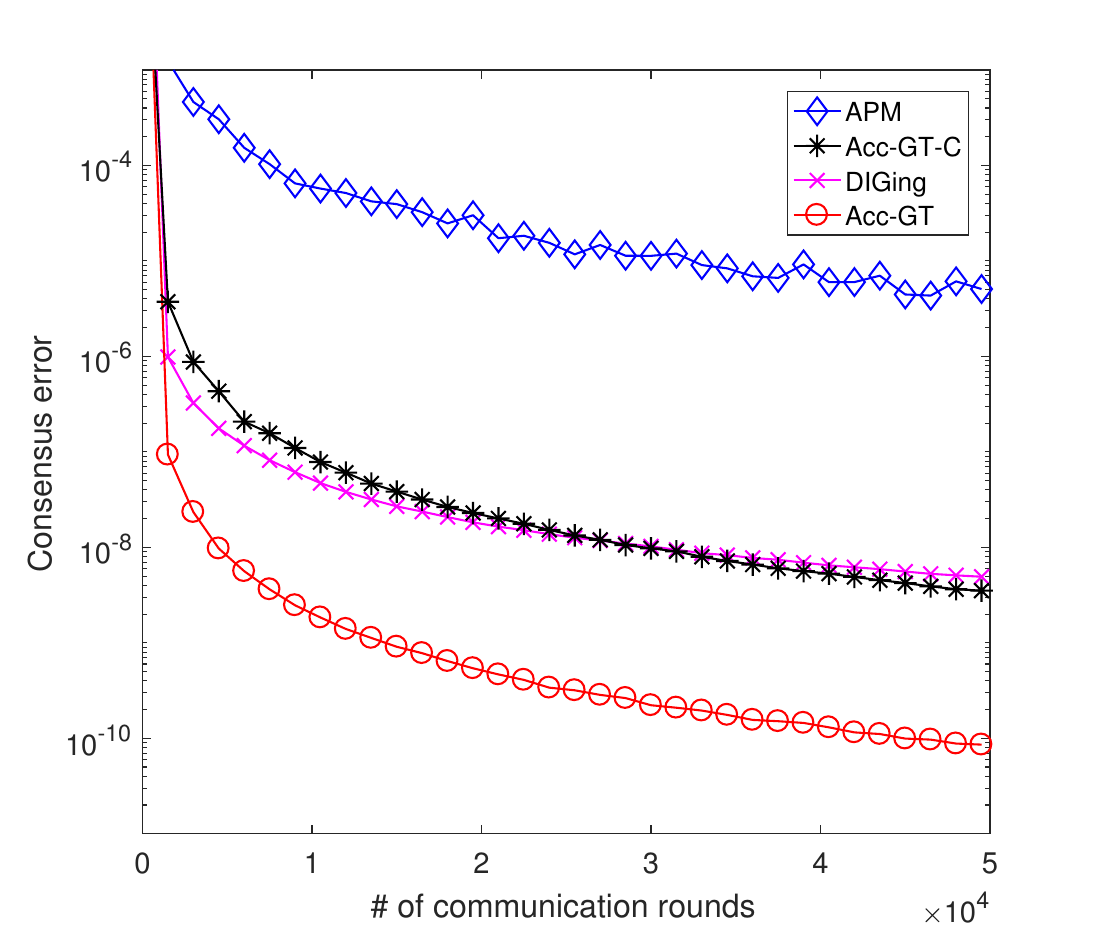}\\
\includegraphics[width=0.5\textwidth,keepaspectratio]{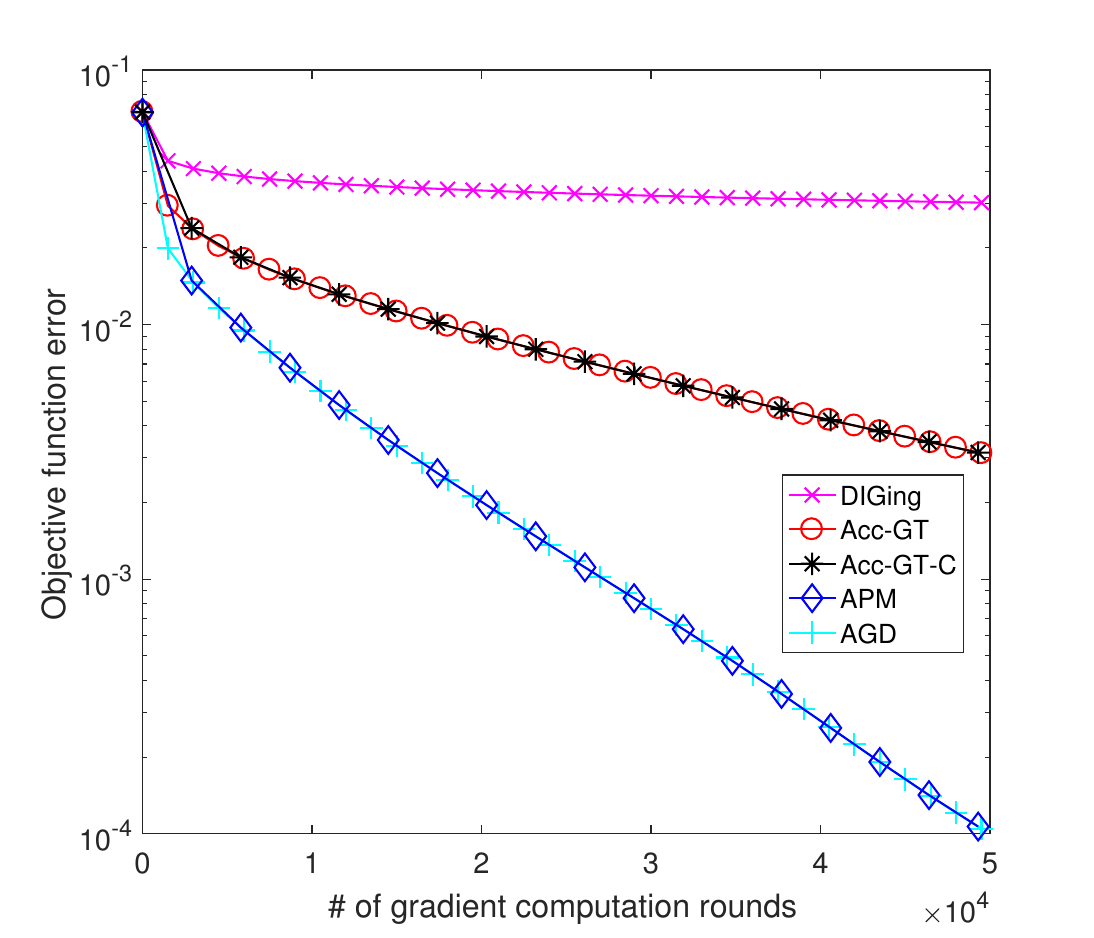}
&\includegraphics[width=0.5\textwidth,keepaspectratio]{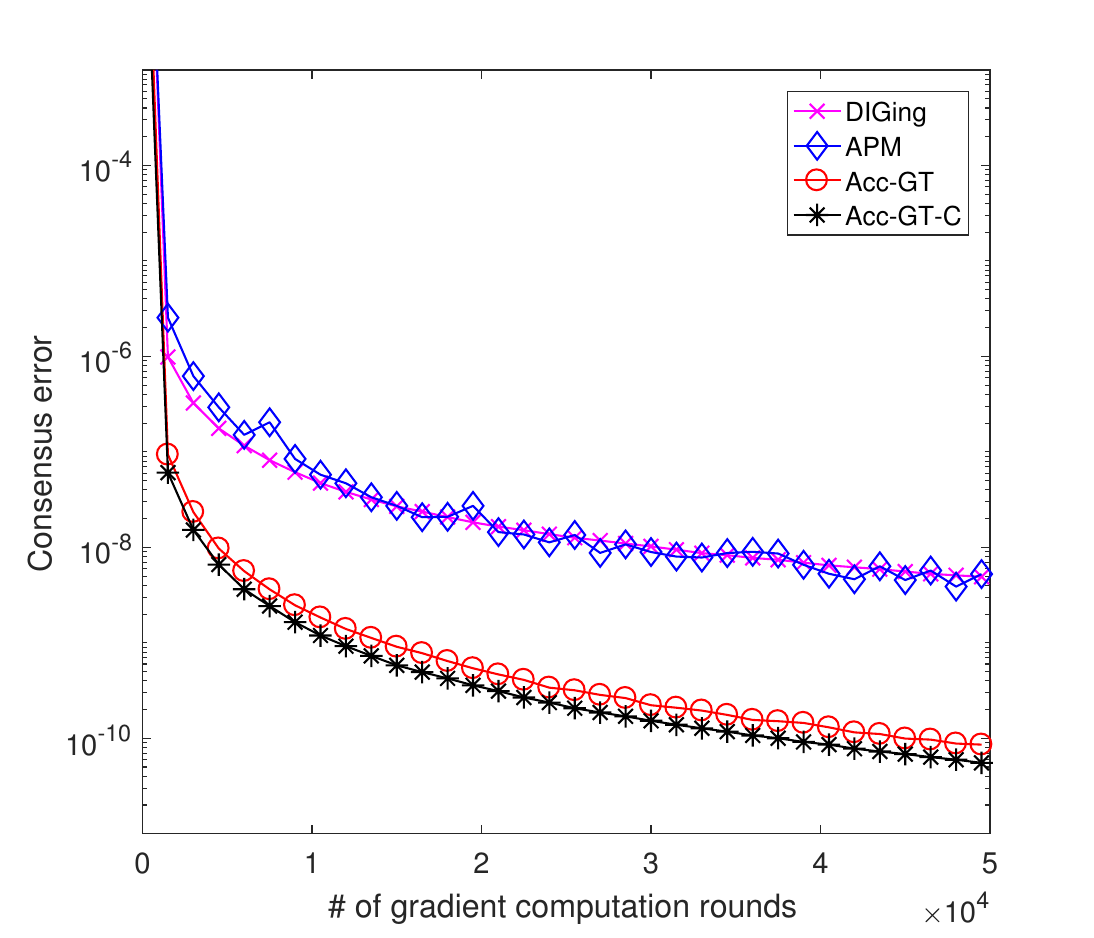}\\
\end{tabular}
\caption{Comparisons of the objective function errors (left) and consensus errors (right) with respect to the number of communication (top) and computation (bottom) rounds for nonstrongly convex problem with $d=20$.}\label{fig4}
\end{figure}

\section{Conclusion}\label{sec:conclusion}
This paper extends the widely used accelerated gradient tracking to time-varying network, which was originally proposed in \citep{qu2017-2} only for static network. We prove the state-of-the-art complexities for both nonstrongly convex and strongly convex problems with the optimal dependence on the precision $\epsilon$ and the condition number $L/\mu$, matching that of the classical centralized accelerated gradient descent. When the network is static, our complexities improve significantly over the previous ones proved in \citep{qu2017-2}. When combing with the Chebyshev acceleration, Our complexities exactly match the lower bounds for both nonstrongly convex and strongly convex problems over static graphs.

This paper only considers the $\gamma$-connectivity of time-varying graphs. Some researchers formulate the time-varying graphs as random graphs \citep{Hong17,Jakovetic14,Ananduta20} and use the mean connectivity in expectation. It is an interesting future work to extend our proof techniques to random graphs. Another interesting direction is to study the acceleration over time-varying unbalanced directed graphs \citep{shi2017}. Besides gradient tracking, EXTRA is another important family of decentralized optimization algorithms. However, it remains an open problem to extend EXTRA to time-varying graphs.

\appendix
\section{Proof of Lemma \ref{lemma_inexact}}
\begin{proof}
From the $\mu$-strong convexity and $L$-smoothness of $f_{(i)}$, we have
\begin{eqnarray}
\hspace*{-3.5cm}\begin{aligned}\notag
F(w)=&\frac{1}{m}\sum_{i=1}^m f_{(i)}(w)\\
\geq&\frac{1}{m}\sum_{i=1}^m\left( f_{(i)}(y_{(i)}^k)+\<\nabla f_{(i)}(y_{(i)}^k),w-y_{(i)}^k\>+\frac{\mu}{2}\|w-y_{(i)}^k\|^2 \right)
\end{aligned}
\end{eqnarray}
\begin{eqnarray}
\begin{aligned}\notag
=&\frac{1}{m}\sum_{i=1}^m\left( f_{(i)}(y_{(i)}^k)+\<\nabla f_{(i)}(y_{(i)}^k),w-y_{(i)}^k\>+\frac{\mu}{2}\|w-\overline y^k\|^2\right.\\
&\qquad\quad\left.+\frac{\mu}{2}\|\overline y^k-y_{(i)}^k\|^2+\mu\<w-\overline y^k,\overline y^k-y_{(i)}^k\> \right)\\
\overset{a}=&\frac{1}{m}\sum_{i=1}^m\left( f_{(i)}(y_{(i)}^k)+\<\nabla f_{(i)}(y_{(i)}^k),w-y_{(i)}^k\>+\frac{\mu}{2}\|w-\overline y^k\|^2+\frac{\mu}{2}\|\overline y^k-y_{(i)}^k\|^2 \right)\\
\geq&\frac{1}{m}\sum_{i=1}^m\left( f_{(i)}(y_{(i)}^k)+\<\nabla f_{(i)}(y_{(i)}^k),w-y_{(i)}^k\>+\frac{\mu}{2}\|w-\overline y^k\|^2 \right)\\
\overset{b}=& f(\overline y^k,\y^k)+\<\overline s^k,w-\overline y^k\>+\frac{\mu}{2}\|w-\overline y^k\|^2,
\end{aligned}
\end{eqnarray}
and
\begin{eqnarray}
\begin{aligned}\notag
F(w)\leq&\frac{1}{m}\sum_{i=1}^m\left( f_{(i)}(y_{(i)}^k)+\<\nabla f_{(i)}(y_{(i)}^k),w-y_{(i)}^k\>+\frac{L}{2}\|w-y_{(i)}^k\|^2 \right)\\
=&\frac{1}{m}\sum_{i=1}^m\left( f_{(i)}(y_{(i)}^k)+\<\nabla f_{(i)}(y_{(i)}^k),w-y_{(i)}^k\>+\frac{L}{2}\|w-\overline y^k\|^2\right.\\
&\qquad\quad\left.+\frac{L}{2}\|\overline y^k-y_{(i)}^k\|^2+L\<w-\overline y^k,\overline y^k-y_{(i)}^k\> \right)\\
\overset{c}=&\frac{1}{m}\sum_{i=1}^m\left( f_{(i)}(y_{(i)}^k)+\<\nabla f_{(i)}(y_{(i)}^k),w-y_{(i)}^k\>+\frac{L}{2}\|w-\overline y^k\|^2+\frac{L}{2}\|\overline y^k-y_{(i)}^k\|^2 \right)\\
\overset{d}=& f(\overline y^k,\y^k)+\<\overline s^k,w-\overline y^k\>+\frac{L}{2}\|w-\overline y^k\|^2+\frac{L}{2m}\|\Pi\y^k\|^2,
\end{aligned}
\end{eqnarray}
where $\overset{a}=$ and $\overset{c}=$ use the definition of $\overline y^k$ in (\ref{def_av_x}), $\overset{b}=$ and $\overset{d}=$ use the definition of $f(\overline y^k,\y^k)$ in (\ref{def_inexact_f}), (\ref{s-relation}), and the definition of $\Pi\y$ in (\ref{def_pi}).
\end{proof}

\bibliography{arxiv}
\bibliographystyle{icml2020}

\end{document}